\documentclass{amsart}

\usepackage{amssymb, amsmath}
\usepackage{mathrsfs}
\usepackage{amscd}
\usepackage{verbatim}
\usepackage{enumerate}
\usepackage{a4wide}
\usepackage{xcolor}
\usepackage[colorinlistoftodos]{todonotes}
\usepackage[colorlinks=true, linkcolor=black, citecolor=black,pdfborder={0 0 0}]{hyperref}

\def\1{{\mathchoice {\rm 1\mskip-4mu l} {\rm 1\mskip-4mu l}
{\rm 1\mskip-4.5mu l} {\rm 1\mskip-5mu l}}}

\theoremstyle{plain}
\newtheorem{theorem}{Theorem}[section]
\theoremstyle{remark}

\theoremstyle{plain}
\newtheorem{corollary}[theorem]{Corollary}
\newtheorem{lemma}[theorem]{Lemma}
\newtheorem{proposition}[theorem]{Proposition}
\newtheorem{definition}[theorem]{Definition}

\numberwithin{equation}{section}

\def\1{{\mathchoice {\rm 1\mskip-4mu l} {\rm 1\mskip-4mu l}
{\rm 1\mskip-4.5mu l} {\rm 1\mskip-5mu l}}}

\def\N{{\mathbb N}}
\def\Z{{\mathbb Z}}
\def\Q{{\mathbb Q}}
\def\R{{\mathbb R}}

\newcommand{\E}{{\mathbb E}}

\newcommand{\eps}{\varepsilon}

\newcommand{\textd}{\textrm{d}}

\newcommand{\beq}{\begin{equation}}
\newcommand{\eeq}{\end{equation}}
\usepackage{tcolorbox}

\author[N. Crawford]{\small Nicholas Crawford \textsuperscript{1}} \thanks{\noindent \textsuperscript{1} N.C. supported by Israel Science Foundation Grant No. 1557/21}
\address{Department of Mathematics, The Technion, Haifa, Israel}
\vspace{-1cm}
\email{nickc@tx.technion.ac.il}
\author[W. M. Ruszel]{\small Wioletta M. Ruszel}
\address{Utrecht University, Budapestlaan 6, 3584 CD Utrecht, The Netherlands}
\email{w.m.ruszel@uu.nl}
\date{\today}

\begin{document}

\title{Random field induced order in two dimensions}

\dedicatory{In memory of Dima Ioffe (1963--2020), friend and mentor.}

\begin{abstract}
In this article we prove that a classical $XY$ model subjected to weak i.i.d. random field pointing in a fixed direction  exhibits residual magnetic order in  $\mathbb{Z}^2$ and aligns perpendicular to the random field direction.  The paper is a sequel to \cite{NC} where the three-dimensional case was treated.  
Our approach is based on a multi-scale Peierls contour argument developed in \cite{NC}. On the microscopic scale we extract energetic costs from the occurrence of contours, which themselves are defined on a macroscopic scale.  The technical challenges in $\mathbb{Z}^2$ stem from difficulties controlling the size and roughness of the fluctuation fields which model the short length-scale oscillations of near-optimizers of the random field  Hamiltonian. 
 \end{abstract}

\keywords{XY model, random field,  Mermin-Wagner theorem, residual magnetic ordering, spin-flop transition}

\subjclass[2010]{82D40}

\maketitle

\tableofcontents

\makeatletter
\providecommand\@dotsep{5}
\makeatother

\section{Introduction}
This paper is a sequel to \cite{Craw_order} on the subject of random field induced ordering.  In it we shall prove that a classical $O(2)$ model subjected to a weak i.i.d. Gaussian field pointing in a fixed direction exhibits residual magnetic order on the square lattice $\Z^2$ and moreover aligns \textit{perpendicular} to the random field direction, whereas \cite{Craw_order} proves the same result on $\Z^3$.  This type of transition is also referred to as of \textit{spin-flop type}, see e.g. \cite{vE1, collet}.  

To facilitate our introductory discussion, let us set up the model and basic notation.  For each $x \in \Z^d$ let $\sigma_x \in \mathbb S^1$.   A vector of spins $\sigma=(\sigma_x)_{x \in \Z^d}$ will be called a spin configuration.  Let $(\alpha_x(\omega))_{x \in \Z^d}$ be an auxiliary i.i.d. family of standard normal random variables, with $\omega$ representing an element of an auxiliary sample space $\Omega$ on which the $\alpha_x$'s are defined.  For any fixed  spin configuration $\sigma^0$ and any bounded region $\Lambda \subset \Z^d$, we can then define a (random) Hamiltonian $-\mathcal{H}_{\Lambda}^{\omega}(\cdot |\sigma^0)$ by
\begin{equation*}
-\mathcal{H}_{\Lambda}^{\omega}(\sigma|\sigma^0)= -\frac 12\sum_{\substack{ x, y: x\sim y, \\ \langle x, y\rangle \cap \Lambda \neq \varnothing} } [\sigma_x- \sigma_y]^2 + \eps \sum_{x \in \Lambda} \alpha_x(\omega) e_2 \cdot \sigma_x
\end{equation*}
where we set $\sigma_x=\sigma_x^0$ for $x\in \Lambda^c$ and allow $\sigma_x \in \mathbb S^1$ to be arbitrary within $\Lambda$.
Here $\langle x, y \rangle$ indicates that $x, y$ are nearest neighbors with respect to the usual graph structure on $\Z^d$ and $e_1, e_2$ denote the standard orthonormal basis of $\R^2$.
Defining finite volume Gibbs measures by
\[
\mu_{\Lambda}^{\sigma^0}(A) =Z_{\Lambda}^{-1}\int_{A} \prod_{x \in \Lambda} \textd \nu(\sigma_x) \exp\{ - \beta \mathcal{H}_{\Lambda}^{\omega}(\sigma|\sigma^0)\}
\]
and denoting the corresponding Gibbs state by $\langle \cdot \rangle= \langle \cdot \rangle_{\Lambda}^{\sigma_0}$,
the question is whether and when, in terms of $\beta, \eps$ and $\sigma^0$, residual magnetic ordering occurs in the limit as $\Lambda \uparrow \Z^d$.  We will refer to this model as the RFO(2) model below.

To understand the effect of the random field 
in the RFO(2) model when $\eps$ is small and $\beta$ is large, we proceed under the hypothesis that there is a separation of scales:  Let $\sigma_x$ be given in polar coordinates by an angle $\theta_x$.  We suppose that $\theta_x = \psi_x+ \hat{\theta}_x$ with $\hat{\theta}_x$ small and $\nabla \psi_x$ small (the scale of the latter will be determined by a more detailed calculation later).  For a first approximation we take $\psi_x$ to be constant.   If $Q_{\ell}$ is a square of sidelength $\ell$, we expand $-\mathcal{H}_{Q_\ell}$ (with free boundary conditions) in $\hat{\theta}_x$ variables, keeping only terms up to second order in $\hat{\theta}_x$.  We find
\begin{equation}
\label{E:ObD}
\sup_{\stackrel{(\theta_x)_{x \in Q_{\ell}}}{\theta_x \approx \psi}} -\mathcal{H}_{Q_\ell}(\theta) = -\frac{\eps^2}{2}\cos^2(\psi)\sum_{x \in Q_{\ell}} \hat{\alpha}_x \Delta^{-1} \cdot \hat{\alpha}_x + \underbrace{\mathcal{O}(\eps|\sum_{z \in Q_{\ell}} \alpha_z|)}_{\textrm{I}}.
\end{equation}
where $\Delta$ is the Neumann Laplacian for $Q_{\ell}$ and $\hat{\alpha}_x= \alpha_x - |Q_{\ell}|^{-1} \sum_{z\in Q_{\ell}} \alpha_z$.
The directions of presumed ordering are obtained by optimizing the first term in $\psi$ and ignoring Term \textrm{I}. We find $\psi\in \{0, \pi\}$ are the preferred angles.  Our goal in this paper is to turn this back-of-the-envelope computation into rigorous mathematics.

Our interest in this question, especially when $d=2$, stems from the fact that the Hamiltonian given above combines a number of competing features in one model.  First, the spins themselves have an \textit{ a priori} continuous symmetry which calls to mind the Mermin-Wagner theorem and Kosterlitz-Thouless phase transition present in the pure $O(2)$ model.  Of course, the random field term in the Hamiltonian breaks the $O(2)$ symmetry, but it is not completely clear what its effect will be.  Here the reader may recall the behavior of the random field Ising model (RFIM) and the possibility that randomness can ``round" phase transitions.  The RFIM would arise if we constrain spins to $\pm e_2$.
When $d=2$ it was proved in \cite{AW}, see also \cite{IM}, that there is a unique infinite volume Gibbs state at arbitrary strength $\eps$ of the disorder.  Recent work, \cite{AP,Ding}, concerns the probability that the choice of constant $\pm1 $ spins at the boundary of a box of sidelength $L$ centered at $0$ influences the value of the value of the spin $\sigma^0$ in the corresponding $\pm$-ground state.  The paper \cite{Ding} confirms that this probability decays exponentially with $L$.  Hence in the RFIM in two dimensions, there is a strong ``screening" of information propagation from boundary to bulk.  

Another interesting connection is to the subject of ``Order-by-Disorder" in the physics literature.  First observed in \cite{Hen}, ``Order-by-Disorder" refers to an effect in which ground state degeneracy is lifted by lattice impurities as modeled by i.i.d. site dilution.  On $\Z^2$,  \cite{Hen} considers a model Hamiltonian  of the form
\[
-\mathcal{H}(\sigma )= \sum_{\|x-y\|_2=1} J_1 [\sigma_x- \sigma_y]^2 + \sum_{\|x-y\|_2=2} J_2 [\sigma_x- \sigma_y]^2
\]
with $|J_1|< 2J_2$.  The  ground-states for this (frustrated) system are obtained by choosing, on each of the even and odd sub-lattices of $\Z^2$,  purely anti-aligned configurations of spins and are thus parameterized by two angles. Here, the term frustration refers to the fact that a system  cannot minimize its total energy purely by minimizing the interaction energy between  pairs of spins.
Vertices of $\Z^2$ are then deleted from the system independently with probability $p\ll 1$.  If, as above,  we consider trial configurations which in polar coordinates consist of a pure ground state with a small amplitude fluctuation superimposed on top, the variational calculation shows that preferred ground states have a \textit{relative} angle of $\pm \frac \pi2$ between spins at $(0, 0)$ and $(0, 1)$.  

In the 1980's, random field ordering was addressed in the physics literature with varying degrees of reliability: the paper \cite{Aharony} presents a Ginzburg-Landau analysis applicable when $d\geq 4$.  For technical reasons we do not have a rigorous mathematical proof verifying the claims therein although many of the obstacles we face in this paper mollify as dimension increases. On the other hand multiple authors \cite{DF1, DF2, MP} attempted to determine the behavior of the RFO(2) and random field clock models when $d=2$.  The papers \cite{DF1, DF2} are clearly incorrect as they claimed there was no low temperature ordered phase, contrary to our main theorem.  The paper \cite{MP} also focused on the $d=2$ case and provides a picture broadly consistent with our work.  However we do not find its analysis convincing: the methods  do not apply when $\beta$ is large and the paper seems to  suggest the incorrect conclusion that the critical inverse temperature, $\beta_c$ is independent of the random field strength $\eps$.   An interesting prediction present in these papers is that the random field clock models have an intermediate phase, where there is no long range order but there is slow (algebraic) decay of correlations.  

From a more mathematical perspective,  van Enter and coauthors \cite{vE1, vE2, vE3} came to this question during investigations of whether the spatial Markov property which characterizes Gibbs measures is preserved under various coarse-graining procedures.  
Ground state behavior and the mean field theory were analyzed in  \cite{vE3, Wehr-et-al-1, Wehr-et-al-2}.

The first work on a microscopically defined model on $\Z^d$ is \cite{NC}, where the RFO(2) model with a Kac potential is treated. With the insights gained from \cite{NC}, the prequel paper \cite{Craw_order} demonstrated the existence of an ordered phase for the nearest neighbor RFO(2) model on $\Z^3$.
 The fundamental limitation of the Kac model work is that the results are only valid if the range of the interaction potential is taken to be or order $\eps^{-1}$.  Nevertheless, this limitation is not entirely technical: as we later discuss, even the nearest neighbor RFO(2) model has a fundamental length scale which is set by energetic considerations:  Taking \eqref{E:ObD} seriously, low energy excitations cost $O(\eps^2 \ell^3)$ for cubes $Q_{\ell}$ in $\Z^3$ but $O(\eps^2  \ell^2 \log\ell)$ for squares in $\Z^2$. On the other hand, to have such an excitation relax back to a ground state,  the Dirichlet energetic cost of the relaxation should not exceed the cost of the low energy excitation.  This translates into spin gradients of $O(\eps)$ in $\Z^3$, respectively $O(\eps |\log \eps|)$ in $\Z^2$.  Thus to see the average angle of a spin configuration change an appreciable amount, we must look at regions with thickness $\eps^{-1}$ in $\Z^3$ and  $\eps^{-1}|\log (\eps)|^{-1/2}$ in $\Z^2$.  The appearance of the extra $|\log (\eps)|$ in these considerations heralds a number of technical challenges which we overcome in adapting the broad proof strategy of \cite{Craw_order} to the present setting.

\subsection{Main Theorem}
Let us now turn to a precise formulation of our main theorem.
Let $\eps$ be fixed and choose $L$ so that
\[
L \sim
\eps^{-1} |\log (\eps)|^{-1/2+1/64}
\]
and $L=2^{k}$ for some $k \in \N$. Let
\[
Q_{L}(z)=
 z + \{0, 1 \dotsc, L-1\}^d.
\]
For $L_0\in \N$, a subset $\Lambda$ of
$\Z^d$ will be said to be $L_0$-measurable if $\Lambda$ is a
union of blocks $Q_{L_0}(r)$ so that $r \in L_0\Z^d$.
As an order parameter, we
define block average magnetizations via 
\begin{align*} &
M_z =
\frac{1}{|Q_{L}|}\sum_{ x \in Q_L(z)}
\sigma_x.
\end{align*}

\begin{theorem}
\label{T:Main2} 
Let $\xi\in (0, 1)$ be fixed and sufficiently small.  There is an $\eps_0(\xi)>0$ and $\Omega_0\subseteq \Omega$ with $\Bbb{P}(\Omega_0)=1$ so that the following properties hold for all $\omega \in \Omega_0$:
First, for every $\eps\in (0, \eps_0)$ there exists an
$L$-measurable subset $\mathbb D_{\omega} \subset \Z^2$,
an $N_0(\omega)\in \N$ and $\delta>0$ such that
\beq 
|\mathbb D_{\omega}
\cap \Lambda_N| \leq
Ce^{-c|\log (\eps)|^{\delta}} |\Lambda_N| \text{ for all }N \geq N_0(\omega).
\eeq
Second, given $\eps\in (0, \eps_0)$, there exists 
$\beta_0(\eps)$, depending on $\eps$ but not $\omega$, such that if $\beta> \beta_0$  then for each $z\in \Lambda_N$ with $Q_{L}(z) \cap \mathbb D_\omega=
\varnothing$,
\beq
\|\langle M_z
\rangle_N^{\omega, e_1} - e_1\|_2 \leq \xi.
\eeq
Here $C, c$ are universal constants, i.e. not dependant upon $\eps, \delta, N, \xi$.
\end{theorem}
Our bound on the transition temperature is $\beta(\eps)\leq  \eps^{-2}|\log (\eps)|^{-1} |\log (\eps)|^{\alpha}$ for any $\alpha>0$.  The first factor $ \eps^{-2}$ is what one expects from mean field theory, see for example \cite{NC}.  The extra factor   $|\log (\eps)|^{-1} $ is a two dimensional effect which may be traced to the fact that the Dirichlet energy of a two dimensional discrete Gaussian free field in a square of side length $\mathcal{O}(1/\eps)$ has an order one contribution from each dyadic annulus in Fourier space, in contrast with its behavior in dimensions $d\geq 3$.

\subsection{A proof sketch}
\label{S:Intuit}
In this subsection we give the heuristics and method of proof which lead to our result. As mentioned earlier, the basic template is taken from \cite{Craw_order}, but the proof needed substantial technical changes to carry through, so we will recapitulate the method while emphasizing the particular difficulties on $\Z^2$.

We already indicated why one might expect ordering perpendicular to the random field in the variational computation \eqref{E:ObD}. Taking another look at that analysis, we observe that the optimal choice for the deviation variables $\hat{\theta}_x$ given $\psi$ is $\hat{\theta}_x=  \cos(\psi) g^N_x$ where 
\[
g^N_x:=-\eps \Delta^{-1} \hat{\alpha}_x.
\]
The first question is, for which $\ell$ is this computation self consistent?
Typically,
\begin{equation*}
\sum_{x \in Q_{\ell}} \hat{\alpha}_x \Delta^{-1} \cdot \hat{\alpha}_x \sim   \ell^2\log(\ell).
\end{equation*}
Choosing $\gamma>0$ sufficiently small,  the central limit theorem implies that Term \textrm{I} in \eqref{E:ObD} will be of lower order (with high probability) if $\ell \gg \eps^{-1}|\log (\eps)|^{-1+\gamma}$ with probability exponentially small in $-|\log (\eps)|^{\delta}$ for some $\delta(\gamma)>0$.  

On the other hand, the behavior of $\|g^N\|_{\infty}$ provides an \textit{upper bound} on $\ell$. When $d=2$, this field has a typical order of magnitude  of $\eps \ell$.  Thus the calculation cannot be taken too seriously for large boxes as the maximizer is inconsistent with the starting assumption that $\hat{\theta}_x$ is (uniformly) small.  We arrive at the following constraints:
\[
 \eps^{-1}|\log (\eps)|^{-1+\gamma}\ll \ell \ll  \eps^{-1}|\log (\eps)|^{-\gamma}.
\]

Let us recall that the Dirichlet energy of spin configurations in $Q_{\ell}$ is defined as  \[
\mathcal E_{Q_\ell}(\sigma):= \sum_{\langle x, y\rangle \subset Q_{\ell}} \|\sigma_x - \sigma_y\|_2^2.
\]  
From a technical perspective, we can only control gradients of $\sigma$ in this sense: 

 For typical realizations of $\alpha$, it is energetically favorable for
 \begin{equation*}
 \mathcal E_{Q_\ell}(\sigma) \lesssim \eps^2  \ell^2\log (\ell)
 \end{equation*}
 but no better
since from the above heuristic our best guess on the maximizer of $-\mathcal{H}_{Q_\ell}$ is $g_x^N$ (which has Dirichlet energy on this order).

Using the above heuristic, we expect the low energy behavior of the RFO(2) model to mirror that of a ``soft" Ising model in which continuous scalar spins sit in a double well potential on $\mathbb S^1$.  So we aim to setup a coarse-grained Peierls argument using the scale $\ell$ as the microscopic scale and to extract energetic cost from the occurrence of contours.  The technical problem in implementing this idea is that the spin space is continuous and there is no microscopic surface tension.  The basic scheme we use to solve this problem was employed in  \cite{Craw_order}, and is inspired by the earlier paper on Kac models \cite{NC} and the book by Presutti \cite{Pres-Book}.  This method requires us to use two scales;  the scale $\ell$ and a second scale $L\gg\ell$ with contours defined relative to the second scale $L$.
It turns out that an appropriate choice of scales is
\[
\ell\sim \eps^{-1} |\log (\eps)|^{-\frac 12-\frac 1{64}} \quad \text{ and } \quad L\sim \eps^{-1} |\log (\eps)|^{-\frac 12+\frac 1{64}}.
\]

A particularly clear explanation of $\eps^{-1} |\log (\eps)|^{-\frac 12}$ as the fundamental length scale around which to work is obtained by making a clever change of variables.  Fix $L_0\in \N, \lambda\in (0, 1)$ with their precise values are to be determined. It is technically convenient and more faithful to our proof to replace $g^N_x$ with the field
\[
g^{\lambda, D}_{x, Q}  = \eps (-\Delta^D_{Q} + \lambda)^{-1}\cdot \alpha_x.
\]
This is done to make it easier to glue fluctuation fields $g^{\lambda, D}_{x, Q}$  from disjoint squares together to make an approximate ground-state in a larger region while also and to controlling  the magnitude of the approximate ground-state.  Of course this is only permissible if the spin configurations whose angles are given by $g^N_{x, Q}$ and $g^{\lambda, D}_{x, Q}  $ have the same Hamiltonian energy to leading order.

We now define a nonlinear change of variables $\phi_x$ by
\begin{equation}
\label{E:COVintro}
\phi_x = \theta_x-\cos(\theta_x)g^{\lambda, D}_{x, Q} \quad \text{ if $x\in Q$ and $Q$ is $L_0$-measurable.}
\end{equation}
If $R$ is a region which is covered by disjoint squares of sidelength $L_0$, the Hamiltonian $-\mathcal H_{R}$ transforms under this change of variable into a new energy functional
\[
\mathcal K_{R}(\phi_x)= \underbrace{\sum_{\langle x, y\rangle} \cos(\phi_x- \phi_y)-1}_{\textrm{II}} + \frac 14 \sum_{x} m_x \cos^2(\phi_x)
\]
with
\[
m_x= \sum_{y \sim x, y\in Q } [g^{\lambda, D}_{y, Q}-g^{\lambda, D}_{x, Q}]^2 \text{ if $x\in Q$}.
\]
There are errors made in this transformation which we can control using $\mathcal E_{R}(\sigma)$ and the behavior of the $g^{\lambda, D}_Q$'s.  In $\mathcal K_{R}$ the energy cost of a spin configuration which has average angle $\psi$ is $\frac 14 \sum_{x} m_x \cos^2(\psi)$.   For this cost to be the same order as the one square computation, we conclude we must take $ |\log(\lambda)|=\log(L_0)$.  Note here the useful point the modifying $\lambda $ by $\log(L_0)$ factors does not effect this constraint.
For an $L_0$-measurable square $Q$ to contribute an order one cost to $\mathcal K_{R}(\phi_x)$, either $\mathcal E_{Q}(\phi)$ must be large or else the average of $\phi$ on $Q$ must be different from $\{0, \pi\}$.  In the latter case the cost of $Q$ will be order one if $L_0\sim \eps^{-1} |\log (\eps)|^{-\frac 12}$.  
Now, while we expect that typically  $\phi_x\in \{0, \pi\}$, to implement the Peierls argument we need to work at a length scale which allows for a droplet of $\{x: \phi_x\equiv \pi\}$ to relax to the bulk $\{x: \phi_x\equiv 0\}$ or vice versa. Replacing $ \cos(\phi_x- \phi_y)-1$ by $1/2 (\phi_x- \phi_y)^2$ in $\mathcal K_{R}(\phi_x)$, the length scale at which these variations occur is determined by the inverse of the average magnitude of $m_x$ which,  from the above discussion, satisfies $m_x\sim \eps^2 |\log (\eps)|$. 

To define contours, we first classify squares $Q_{\ell}$ as either \textit{good} (indicated by $\psi^0_z \psi^1_z= \pm 1$ for all $z$ in some neighbourhood around $Q$) or \textit{bad} relative to a spin configuration $\sigma$ (indicated by $\psi^0_z \psi^1_z=0$).  Fixing $\sigma$, call $Q_\ell$ \textit{bad} for $\sigma$ if either  the Dirichlet energy $\mathcal E_{Q_\ell}(\sigma)$
is substantially larger than $\eps^2  \ell^2 \log (\ell)$ ($\psi^0_z=0$ in some neighbourhood of $Q$)
or if the average of spins in $Q_\ell$, $\sigma(Q_\ell):=\ell^{-2} \sum_{x \in Q_{\ell}} \sigma_x$, is far from  the set $\{ e_1, -e_1\}$ (i.e. $\psi^1_z=0$); see Section \ref{S:Contours} for further details.  

We tile $\Z^2$ by squares $Q_L$ of sidelength $L$.  Squares in this tiling will be called $L$ measurable. We define a region $\Gamma$ to be a contour for $\sigma$ if it is a maximally connected cluster of $L$ measurable squares $Q_L$ of side length $L$  so that within distance $2L$ of  $Q_L$ there is a cube $Q_\ell$ of side-length $\ell$ which is bad for $\sigma$.

Given a spin configuration $\sigma$ and an associated contour $\Gamma$ let $\delta(\Gamma)$ denote the $L$-measurable neighborhood of squares within distance say $10$ from $\Gamma$.

 The idea is to compare $\sigma$ with a new spin configuration $\tilde{\sigma}$ which agrees with either $\sigma$ or the reflection of $\sigma$ across the $e_2$ axis on each component of $\Lambda \backslash \delta(\Gamma)$. We construct $\tilde \sigma$ to have an angle uniformly close to either $0$ or $\pi$ on the whole of $\Gamma$ and interpolate smoothly to $\sigma$ or its reflection on $\delta(\Gamma)\backslash \Gamma$.  The goal is to show that the energetic advantage obtained from $\tilde \sigma$ being good on $\Gamma$ will be larger than the energetic cost incurred in forcing the interpolation.
 
The key step and bulk of the paper will be spent on the boundary surgery of a spin configuration $\sigma$ on $\delta(\Gamma) \backslash \Gamma$.
In the discussion of this procedure, we will assume the randomness is ``uniformly typical".   The actual analysis requires that, under certain conditions placed on the randomness in a contour $\Gamma$, we can show that the energetic gain from the bulk of a contour outweighs the cost from the boundary surgery. Not all contours satisfy these conditions, but we will show that regions where our requirements fail are sparse and non-percolating.  The probabilistic bounds required to prove these lemmas and the proofs of the lemmas themselves are postponed until Section \ref{S:dirty}.
  
By definition, $\sigma$ is good in $\delta(\Gamma)\backslash \Gamma$.  Thus, for each cube $Q_\ell$ within $L$ of $\delta(\Gamma)\backslash \Gamma$
\[
\mathcal E_{Q_\ell}(\sigma)\leq 4 \eps^2 \ell^2 \log (\ell),
\]
and $\ell^{-2} \sum_{x\in Q_\ell} \sigma_x$ is close to the set $\{-e_1,  e_1\}$.  Combining these to facts, one can show that the sign is constant over connected components of $\delta(\Gamma)\backslash \Gamma$.  

Fix a component $R$ of $\delta(\Gamma)\backslash \Gamma$ and suppose $\ell^{-2} \sum_{x\in Q_\ell} \sigma_x$ is close to $e_1$ throughout $R$ (the other case is similar).  Then because we are working in $\Z^2$, \textit{ a priori } bounds on $\mathcal E_{Q_\ell}(\sigma)$ allow us to find a small neighborhood $R' \supset R$ so that for vertices $x$ in the graph boundary of $R'$, $\theta_x \in (-\frac \pi8, \frac \pi8)$.  This point is addressed in Section \ref{sec:BoundSur}.

Next, if we start with such a boundary condition $\sigma^0$ for $R'$, we can look for maximizers of $-\mathcal{H}_{R'}(\sigma|\sigma^0)$.  Actually, it is better to use  \eqref{E:COVintro} and look for maximizers after a change of variables. 
There is a unique maximizer of $\mathcal K_{R'}$ and that this maximizer also has $\phi_x \in (- \frac \pi5, \frac \pi5)$ throughout $R'$.  Because of this and the fact that typically $m_x \sim \eps^2 |\log (\eps)|$, one should morally regard the maximizer as behaving like a solution to the discrete elliptic PDE with mass 
\begin{equation*}
(-\Delta+ \eps^2 |\log (\eps)|)  \phi_x=0.
\end{equation*}
First we modify $\sigma$ into a new configuration $\sigma^{(1)}$ on $\Lambda_N\backslash \Gamma$.  We keep $\sigma$ fixed on the exterior component $\mathscr E$ of $\Lambda_N\backslash \Gamma$ and note the value $\textrm{sgn}(\sigma\cdot e_1)$ on $\mathscr E \cap \delta(\Gamma) \backslash \Gamma$. We make the following transformation on the interior components $\mathcal I_j$.  If, on $\mathcal I_j \cap \delta(\Gamma) \backslash \Gamma$,   the \textrm{sign} of $\sigma\cdot e_1$ agrees with its value on the boundary of the exterior component, we keep $\sigma$ fixed on the entire interior component.  Otherwise, we reflect $\sigma$ across the $e_2$ axis on all of $\mathcal I_j$.  We call the result, which is defined only in $\Lambda_N\backslash \Gamma$, $\sigma^{(1)}$. 

The second modification $\sigma^{(2)}$ results from forcing spins in $\sigma^{(1)}$ which are near dirty boxes towards the optimizers via gradient descent.

After replacing $\sigma^{(2)}$ by the maximizer of $\mathcal K_{R'}$ inside $R'$ for each component $R$ of $\delta(\Gamma) \backslash \Gamma$, we can invert the change of variables, Lemma \ref{lem:cov}, to produce a new spin configuration $\sigma^{(3)}$, which disagrees with $\sigma$ only on $\delta(\Gamma) \backslash \Gamma$.
Because $\sigma^{(3)}$ is very close to one of $\pm e_1$ deep inside each thickened boundary component of $\Gamma$, at negligible further cost we can force it to be \textit{exactly} $\pm e_1$ in the middle of each of the boundary components, calling the result $\sigma^{(4)}$.

Let $\bar{\delta}(\Gamma)$ be the neighborhood of radius $\frac L2$ of $\Gamma$.  To obtain $\tilde{\sigma}$ from $\sigma^{(4)}$, we construct (almost-) maximizers for $-\mathcal{H}_{\bar{\delta}(\Gamma)}(\cdot)$.  This gives two configurations $\sigma^{\pm}_{\Gamma}$ uniformly close  to $\pm e_1$ throughout $\bar{\delta}(\Gamma)$.  The configuration  $\tilde{\sigma}$ is chosen to agree with $\sigma^{(4)}$ on $\bar{\delta}(\Gamma)^c$ and to agree with an appropriately chosen near maximizer of $-\mathcal{H}_{\bar{\delta}(\Gamma)}(\cdot)$ inside $\bar{\delta}(\Gamma)$.  

To construct the near maximizer, we reason as follows: if $Q\subset \bar{\delta}(\Gamma)$ is $\ell$ measurable, and in the absence of boundary conditions, the maximum \eqref{E:ObD} was well captured by $g_{Q, x}^N$.  It is tempting to glue these fields together in $\bar{\delta}(\Gamma)$ to form a patchwork near maximizer in $\bar{\delta}(\Gamma)$ but the cost from boundary mismatches is prohibitive.  This problem is circumvented by using, instead, $g^{\lambda, D}_{x, Q}$ with $\lambda=\eps^2 |\log (\eps)|^{1+1/16}$ as mentioned above.  The reason this is sufficient is that, because we work in two dimensions, each dyadic annulus $A_j=\{k: 2^j\leq \|k\|_2\leq 2^{j+1}\}$, $2^j\gg \sqrt{\lambda}$ in Fourier space provides roughly the same order one contribution to $ \mathcal E_{Q}(g^{\lambda, D}_{Q}), \mathcal E_{Q}(g^{N}_{Q})$.  As a result, for $\lambda$ not too large, we should have 
$ \mathcal E_{Q}(g^{\lambda, D}_{Q})=\mathcal E_{Q}(g^{N}_{Q})$ to leading order in $\ell$.  On the other hand due to the mass $\lambda$,  $\|g^{\lambda, D}_{Q}\|$ remains under control even if we increase the sidelength of $Q$ to $L$.  Finally, due to the Dirichlet boundary conditions,  we can glue $g^{\lambda, D}_{Q_i}$'s from distinct squares $(Q_i)_{i\in I}$ with minimal energetic cost from the boundary between squares in order to make a spin configuration in $\bar{\delta}(\Gamma)$.

After accounting for all errors and also what we gain in the comparison between $\sigma$ and $\tilde{\sigma}$ we are able to show
\[
-\mathcal{H}_{\Lambda_N}(\tilde{\sigma}|e_1)+\mathcal{H}_{\Lambda_N}(\sigma| e_1) \geq C \eps^2|\log (\eps)|^{1-\delta}|\Gamma|
\]
for some $C, \delta >0$.  
With this result in hand we can construct a Peierls argument and prove Theorem \ref{T:Main2}.

The remainder of this paper is organized as follows.  In Section \ref{S:Prelims}  we collect basic notation used throughout the rest of the article.  Section \ref{S:Contours} defines what are contours in our model in a careful way, while Section \ref{S:Prob} addresses what sort of requirements we need from the randomness $\alpha_x$ in order for a subregion of $\Z^2$ to be well behaved.  Regions satisfying these requirements are called clean regions and in Section \ref{S:dirty} we quantify the likelihood and sparsity of dirty regions.  Proofs related to defects are postponed to Section \ref{S:defects}. In Lemma \ref{lem:Peir} we state the Peierls estimate and prove Theorem \ref{T:Main2} on its basis.  Assuming that a given contour is clean, Section \ref{S:Groundstates} outlines the construction of almost maximizers to $-\mathcal{H}_{\bar{\delta}(\Gamma)}(\cdot)$ while Section \ref{sec:BoundSur} outlines the construction of $\sigma^{(1)}$ above.   In Section \ref{S:Peierls} we construct $\tilde{\sigma}$ and then prove the Peierls estimate.  Proofs of major estimates required in Sections \ref{S:surgery} and \ref{S:Peierls} are postponed until Sections \ref{S:defects} and \ref{S:Rand} respectively.  Section \ref{S:RandSub}, while technical, is key to the whole paper as it contains Proposition \ref{prop:ub_supg}.  Finally Section \ref{S:Tech} contains various technical lemmas needed elsewhere, for example in the analysis of maximizers of  $\mathcal K_{R'}(\phi_x|\phi_{{R'}^c})$.

\section{Notation: Hamiltonians, coarse-graining and contours}\label{S:Prelims}

\subsection{Scales and coarse-graining regions}
Fix $\eps>0$, which represents the random field strength in the model.  As explained above, this defines a length scale $ \eps^{-1} |\log(\eps)|^{-\frac{1}{2}}$ around which we will work.

\begin{comment}
Let us define two length scales which will form the basis for our coarse-graining scheme.

\begin{tcolorbox}
Let  $s, s \in (0,\frac{1}{2})$ and $\eta_{\lambda} \in (0,2)$ be parameters and define
\begin{equation}
\ell \sim \eps^{-1} |\log(\eps)|^{-\frac{1}{2} - s}, \, \, L \sim \eps^{-1} |\log(\eps)|^{-\frac{1}{2} + s},  \, \, \lambda \sim \eps^{2} |\log_2(\eps)|^{\eta_{\lambda}}
\end{equation}
\end{tcolorbox}

\end{comment}

Let $\ell, L$ be defined by
\begin{equation*} \label{scales}
\begin{split}
\log_2(\ell) & = \lfloor \log_2 \left( \eps^{-1}|\log(\eps)|^{-1/2 - s} \right) \rfloor, \\
\log_2(L) & = \lceil \log_2 \left( \eps^{-1}|\log(\eps)|^{-1/2 + s} \right) \rceil
\end{split}
\end{equation*}
where $s\in(0, 1/32)$ is fixed throughout and $\lfloor\cdot \rfloor$ resp. $\lceil \cdot \rceil$ are the floor resp. ceiling functions.  In the proof sketch given in the previous section, we fixed $s=1/64$ and also fixed various other exponents.  However, we feel that in deriving bounds, for clarity it is better to treat the exponents symbolically and will do so from now on.

Let  $L_0 \in \N$ be fixed. For a point $r\in L_0 \mathbb{Z}^2$ let $Q_{L_0}(r)$ denote the box of side-length $L_0$ with $r$ in its lower left corner. We say that in this case  the box $Q_{L_0}(r)$ is $L_0$-measurable. Let
\begin{align}\label{def:setQ}
\mathcal{Q}_{L_0} := &\biggl \{ Q: Q \text{ is } L_0 \text{ measurable }\biggr \}, \nonumber \\
\mathcal Q^*_{L_0}:=&\{\text{Squares $Q$ with side-length $L_0$ so that one of }\\
\nonumber&Q, Q+(0, L_0/2), Q+(L_0/2, 0), Q+(L_0/2, L_0/2)\text{ is $L_0$-measurable}\}.
\end{align}

To coarse-grain $\mathbb{Z}^2$, we first coarse-grain $\mathbb{R}^2$ by squares and then take the intersection of these squares with $\mathbb{Z}^2$.
For some set $A\subset \mathbb{Z}^2$ define $\widehat{A} \subset \mathbb{R}^2$ as
the union of closed boxes of side length 1 centered at elements of $A$ and call $A$ connected if $\widehat{A}$ is connected. The complement of $\widehat{A}$
decomposes into one infinite connected component $\textrm{Ext}(\widehat{A})$ and a union of finitely many finite components $\textrm{Int}(\widehat{A}) =\bigcup_{i=1}^m \textrm{Int}_i(\widehat{A})$.
We thicken the set $\widehat{A}$ by $L$  in the following way
\begin{equation}
\delta_{L_0} (\widehat{A}) := \bigcup_{\substack{r\in L_0 \mathbb{Z}^d: \\ \textrm{dH}(\widehat{Q}_{L_0}(r), \widehat{A}) < \ell }} \widehat{Q}_{L_0} (r)
\end{equation}
where $\textrm{dH}(\cdot, \cdot)$ is the Hausdorff distance in $\mathbb{R}^2$ between two sets w.r.t. the $\ell_{\infty}$-norm. The quantities $\delta_{L_0} (A), \textrm{Ext}(A), \textrm{Int}(A)$ are then just defined by $ \delta_{L_0} (A)=\delta_{L_0} (\widehat{A}) \cap \mathbb{Z}^2, \textrm{Ext}(A):= \textrm{Ext}(\widehat{A})\cap \mathbb{Z}^2$ resp. $\textrm{Int}(A):= \textrm{Int}(\widehat{A})\cap \mathbb{Z}^2$. The closed hull $\textrm{cl}(A)$ is defined as $\textrm{cl}(A) = \delta_{L} (A) \cup \textrm{Int}(A)$.

\subsection{Hamiltonians and Dirichlet energies}
We will use the notation
\[
g \lesssim f
\]
if there is some universal constant $C>0$ such that
\[
g \leq C f.
\]
Let $\{ \alpha_x \}_{x\in \Z^2}$ a collection of i.i.d. $N(0,1)$ random variables and let $\{ e_1,  e_2 \}$ denote the standard  unit basis vectors in $\R^2$. Our  spin variables, denoted by  $\sigma_x, x\in \Z^2$ take values in $\mathbb{S}^1$.
The space of spin configurations restricted to $R$ will be denoted by $\mathcal{S}_{R} = (\mathbb{S}^1)^{R}$, and on the full lattice simply by $\mathcal{S}=(\mathbb{S}^1)^{\Z^2}$. We say that two vertices $x,y$ are nearest neighbours, $x\sim y$ for $x,y\in \Z^2$, if $\textrm{dist}(x,y)=1$ where $\textrm{dist}(\cdot, \cdot)$ denotes the Euclidean distance. We will denote by $\| \cdot \|_p$
the $\ell_p(\R^2)$ norm in $\R^2$, $p>1$. If $R\subset \Z^2$ and $f: R\rightarrow \R^2$ we define
\[
\| f\|_{p, R} := \left ( \sum_{x\in R} \| f_x\|^p_2\right )^{1/p}.
\]
For $\sigma \in \mathcal{S}_{R}$, call
\begin{equation*}
\mathcal{E}_{R}(\sigma)=\sum_{{x\sim y; \atop x, y \in R}} \| \sigma_x - \sigma_y\|_2^2
\end{equation*}
the Dirichlet energy, or just the energy,  of $\sigma$.

We shall occasionally need to distinguish inner and outer boundaries of sets: the inner, resp. outer,  boundaries  of a region $R \subset \Z^2$ are defined as
\[
\begin{split}
\partial^i R  &= \{ x\in R: \textrm{dist}(x,R^c) \leq 1\},\\
\partial^o R  &= \{ x\in R^c: \textrm{dist}(x,R) \leq 1\}.
\end{split}
\]

For $Q$ a square of sidelength $L_0$ in $\Z^2$ let  $\Delta_Q^D$, resp. $\Delta_Q^N$, denote the Dirichlet, resp. Neumann, Laplacian in $Q$ :
\[
\begin{split}
-\Delta^D_{Q} \cdot f_x & = \sum_{{y\sim x, \atop  y\in Q\cup \partial^o Q }} (f_x - f_y), \text{  where $f$ is extended to be $0$ on } \partial^o Q\\
-\Delta^N_{Q} \cdot f_x & = \sum_{{y\sim x, \atop  y \in Q}} (f_x - f_y).
\end{split}
\]
Let $e=\langle x, y \rangle$ be an oriented edge, we sometimes use the notation $f_y-f_x= \nabla_e f$ .

For $\eps>0$ fixed and a mass $\lambda:=\eps^2 |\log (\eps)|^{1+\eta_{\lambda}}$ we define random fields $g$ by
\[
\begin{split}
g^{\lambda, D}_{x, Q}  &= \eps (-\Delta^D_{Q} + \lambda)^{-1}\cdot \alpha_x, \\
g^{\lambda, N}_{x, Q}  &= \eps (-\Delta^N_{Q} + \lambda)^{-1}\cdot \alpha_x.\\
\end{split}
\]
In the proof sketch, Section \ref{S:Intuit}, we fixed $\eta_{\lambda}=1/16$, but in general we allow $\eta_{\lambda}\in (2s, 1/8)$. We take it to be fixed throughout (recall $s\in (0,1/32)$).

\begin{definition}
Given $\tau \in \mathcal{S}$ and $\sigma \in \mathcal{S}_R$ we may extend $\sigma$ to be defined on $\mathcal{S}$ by setting $\sigma|_{R^c} = \tau|_{R^c}$. The Hamiltonian $\mathcal{H}$ in $R$ with boundary condition $\tau$ is defined by

\begin{equation}\label{def:Ham1}
-\mathcal{H}_{R}(\sigma|\tau) = -\frac{1}{2} \mathcal{E}_{R}(\sigma) + \eps \sum_{x\in R} \alpha_x e_2 \cdot \sigma_x + \frac{1}{2}\sum_{{x\sim y; \atop x\in R, y\in R^c}} \|\sigma_x - \tau_y \|^2
\end{equation}
resp. with free boundary condition
\begin{equation}
-\mathcal{H}_{R}(\sigma) = -\frac{1}{2} \mathcal{E}_{R}(\sigma) + \eps \sum_{x\in R} \alpha_x e_2 \cdot \sigma_x
\end{equation}
for some $\eps >0$.
\end{definition}

\subsection{Bad behaviour from the spins: an auxiliary Ising-type spin system}

In this subsection we introduce phase variables $\Psi$ which will differentiate, with respect to $\sigma$, well behaved versus poorly behaved boxes. 

Typically, we expect the Dirichlet energy to be of order
\[
\mathcal{E}_Q(\sigma) \leq \eps^2 |\log (\eps)| |Q|
\]
and we also expect $\sigma$ to have a significant alignment with $\pm e_1$.
A box $Q$ will be \textit{bad} if $\sigma$ violates one of these two conditions on $Q$.

More precisely, let $\chi\in (0, 1/16)$ be fixed. On the scale $\ell$ and  fixing 
 $\sigma \in \mathcal{S}$, $z\in \mathbb{Z}^2$ set

\begin{equation*}\label{def:Psi0}
\psi^0_z (\sigma) = \begin{cases}
1 & \text{ if } \forall r\in \frac{\ell}{2} \mathbb Z^2 \text{ s.t. } \| z- r \|_2 \leq 2 \ell, \, \, \mathcal{E}_{Q_{\ell}(r)} (\sigma) \leq \eps^2 |\log(\eps)|^{1+\chi} \ell^{2},  \\
0. & \text{otherwise.}
\end{cases}
\end{equation*}

Next, fixing $\xi\in (0, 1)$, let $\sigma(Q):= \frac{1}{|Q|} \sum_{x\in Q} \sigma_x$ and define
\begin{equation*}
\psi^{1,\xi}_z (\sigma) = \begin{cases}
1 & \text{ if } \forall r \in \frac{\ell}{2} \mathbb Z^2 \text{ s.t. } \|z-r \|_2 \leq 2 \ell: \, \, \sigma(Q_{\ell}(r)) \cdot e_1 \in [1-\xi, 1], \\
-1 & \text{ if } \forall r \in \frac{\ell}{2} \mathbb Z^2 \text{ s.t. } \|z-r \|_2 \leq 2 \ell: \, \, \sigma(Q_{\ell}(r)) \cdot e_1 \in [-1, -1+\xi], \\
0 & \text{ otherwise. }
\end{cases}
\end{equation*}
Finally we set $\psi_y=\psi^0_y \psi^{1, \xi}_y$.

We use these functions coarse-grain on the scale $L$ as follows.  Let $Q_L(r)$ be an $L$-measurable box.  Then for all $z\in Q_L(r)$ we set
\begin{equation} \label{def: Psi}
\Psi^{\xi}_z (\sigma) = \begin{cases}
1 & \text{ if } \forall r' \text{ s.t. } \| r'-r \|_2 \leq 2L \text{ and } y\in Q_L(r') \text{ we have } \psi^0_y \psi^{1, \xi}_y =1, \\
-1 & \text{ if } \forall r' \text{ s.t. } \| r'-r \|_2 \leq 2L\text{ and } y\in Q_L(r') \text{ we have } \psi^0_y \psi^{1, \xi}_y =-1, \\
0. & \text{ otherwise }
\end{cases}
\end{equation}
 We illustrate this in the following Figure \ref{fig:good_bad}.

\begin{figure}[htb]
\begin{minipage}{0.45\textwidth}
\centering
\includegraphics[scale=0.6]{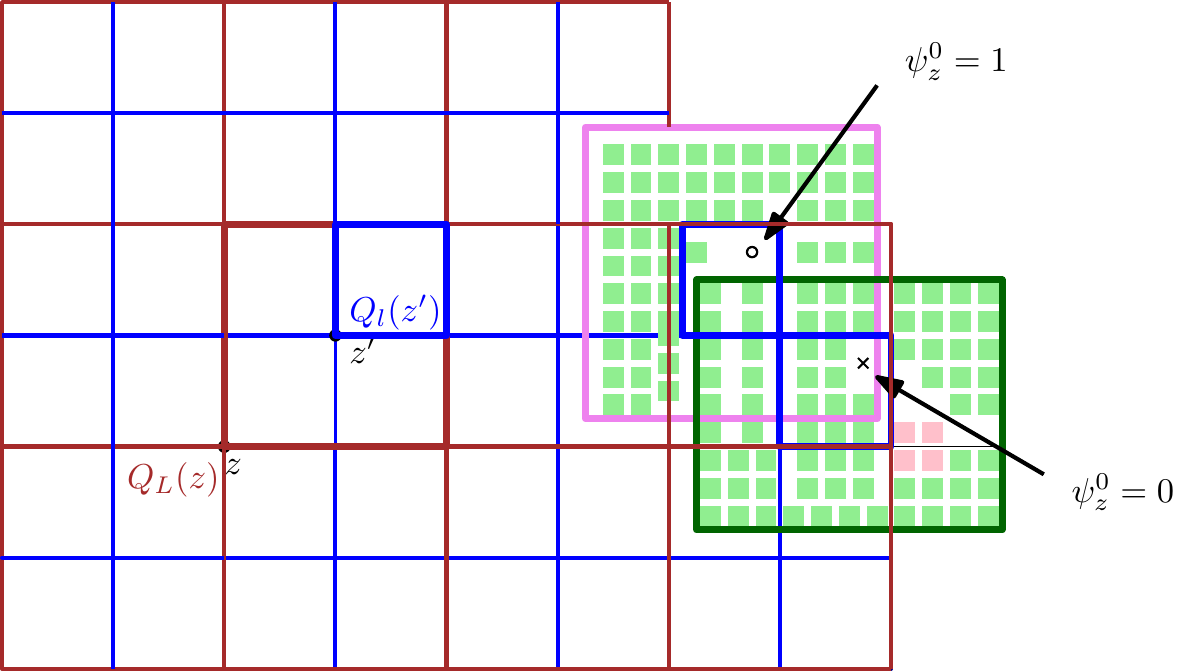}
\end{minipage}
\hspace{0.05\textwidth}
\begin{minipage}{0.45\textwidth}
\centering
\includegraphics[scale=0.6]{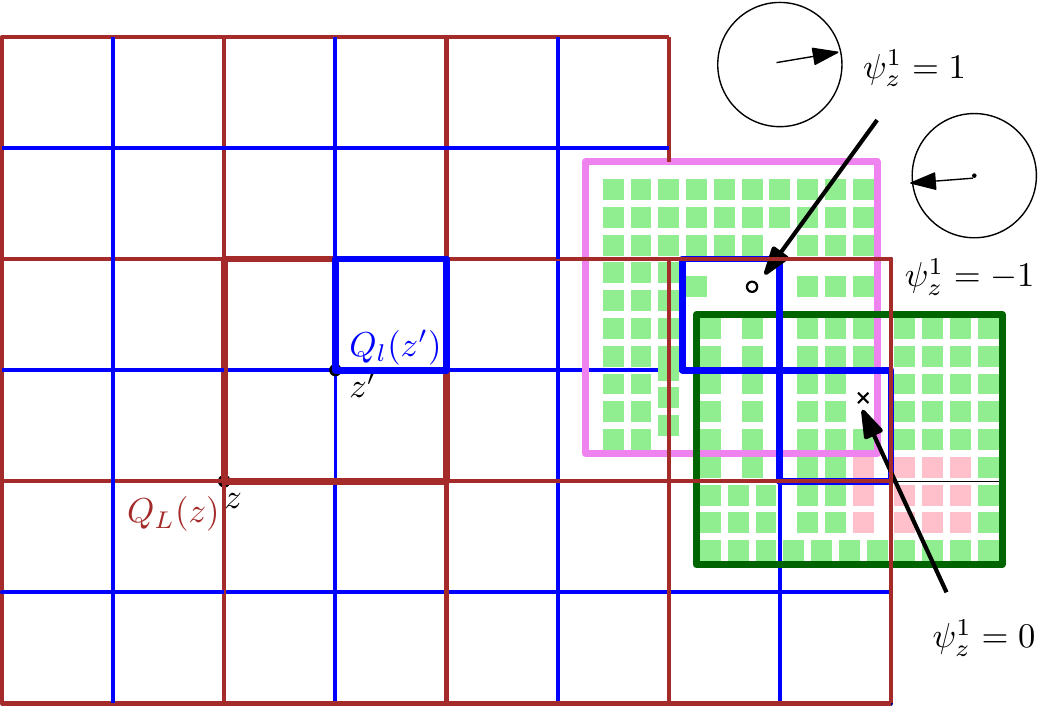}
\end{minipage}

\caption{Example of tiling the lattice and definition of good and bad boxes}
\label{fig:good_bad}

\end{figure}

\subsection{Contours}\label{S:Contours}

Given a configuration $\sigma \in \mathcal{S}$ and fixing an error $\xi >0$, with $\psi, \Psi$ defined above, we partition the lattice $\mathbb{Z}^2$ in three  regions: $R^+:= \{ z\in \mathbb{Z}^2: \Psi^{\xi}_z =1 \}, R^-:=\{ z\in \mathbb{Z}^2: \Psi^{\xi}_z =-1\} $ and $R^0:=\{ z\in \mathbb{Z}^2: \Psi^{\xi}_z =0\}$.  The  region $R^0$ separates $\Psi^{\xi}=1$ from $\Psi^{\xi}=-1$ and will play the role of a contour.  Note that by definition, all $3$ regions are $L$-measurable.

\begin{definition}\label{def:contoursPsi}
An abstract contour $\Gamma$ is a pair $(\textrm{sp}(\Gamma), \psi^{\xi}_{\Gamma})$, where $\textrm{sp}(\Gamma) \subset \Z^2$ is connected and $L$-measurable and $\psi_z(\Gamma): \delta_L(\Gamma) \rightarrow \{-1,0,+1\}$ is a map specifying the values of $\psi_z(\sigma)$ on $\delta_L(\Gamma)$ on the scale $\ell$. 
\end{definition}

We set $|\Gamma|= |\textrm{sp}(\Gamma)|$ and introduce  the notation $\delta_{\textrm{ext}}(\Gamma)$ resp. $\delta^{\pm}_{\textrm{int}}(\Gamma)$
\[
\delta_{\textrm{ext}}(\Gamma) :=  \delta_L(\Gamma) \cap \textrm{Ext}(\Gamma)
\]
resp.
\[
\delta^{\pm}_{\textrm{int}}(\Gamma) :=  \delta_L(\Gamma) \cap \textrm{Int}(\Gamma)\cap \{z: \psi_z(\Gamma)=\pm1\}.
\]

Now we define a \textit{concrete contour} or just a \textit{contour}. 
\begin{definition}
Given an abstract contour  $\Gamma$ we call it a concrete contour for the configuration $\sigma$, or just a contour for $\sigma$ if $sp(\Gamma)$ is a maximal connected component of the set $R^0(\sigma)$ and $\psi^{\chi}_z(\sigma)\equiv\psi_z(\Gamma)$ for every $z \in sp(\Gamma)$. Let further
\[
\mathbb{X}(\Gamma) = \{ \sigma: \Gamma \text{ is a contour for  } \sigma\}.
\]
\end{definition}

\begin{definition}
The contours $\Gamma_1$, $\Gamma_2$ are said to be compatible if 
$\delta(\Gamma_1) \cap \textrm{sp}(\Gamma_2) =\varnothing$
and
\[
\psi_z(\Gamma_1) = \psi_z(\Gamma_2), \, \, \forall z\in \delta(\Gamma_1) \cap \delta(\Gamma_2). 
\]
\end{definition}

An example of a contour can be found in Figure \ref{fig:cont}.

\begin{figure}[htb]
\centering
\includegraphics[scale=0.5]{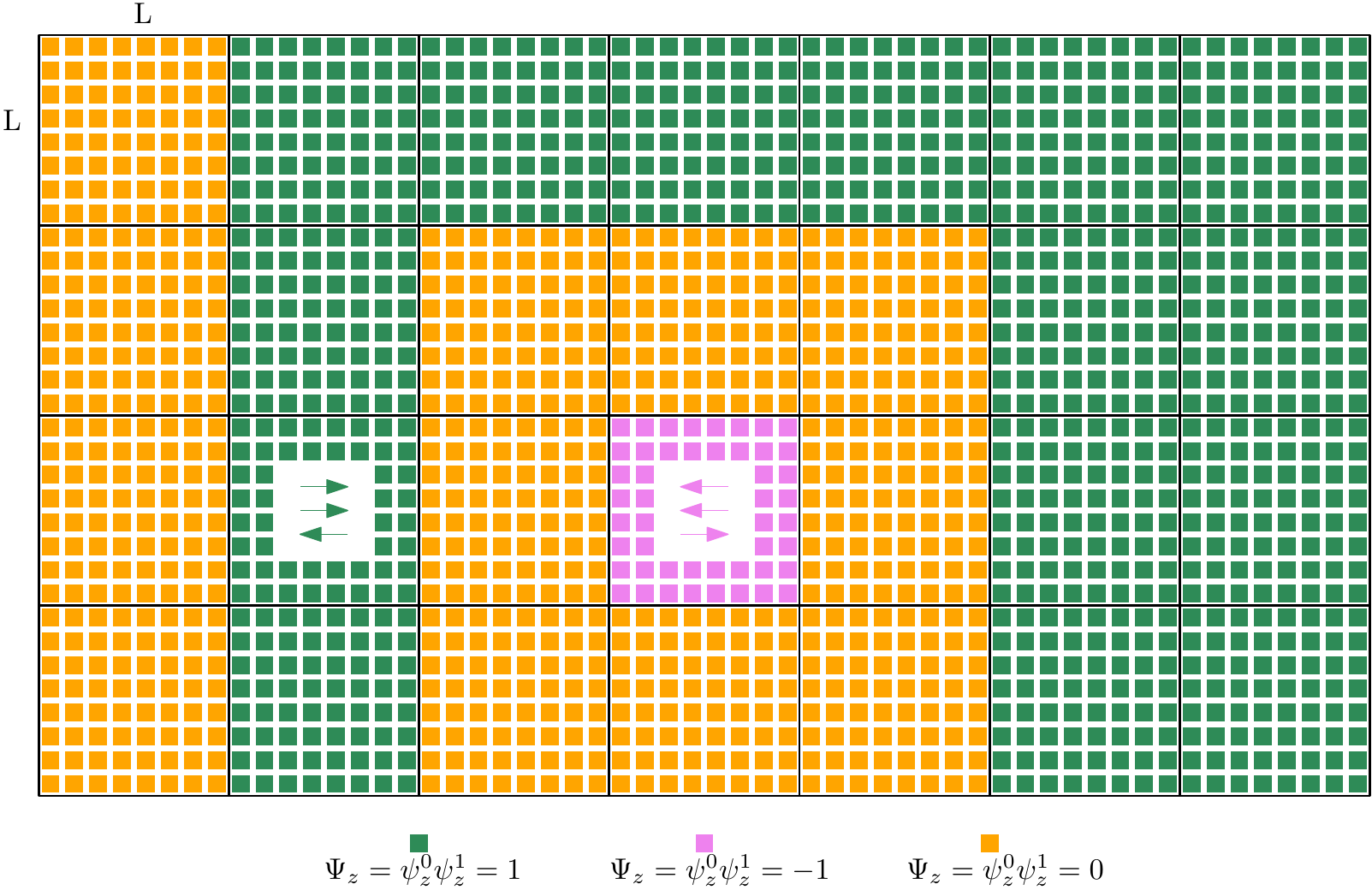}
\caption{Example of a contour (orange) separating configurations which are mostly oriented towards $e_1$ resp. $-e_1$}
\label{fig:cont}
\end{figure}

\subsection{Bad behaviour from the random field: clean and dirty boxes and regular regions}
\label{S:Prob}
Given a box  $Q$ of sidelength $L_0$  and $x\in Q$, let 
\[
m_{x} (\omega) = \sum_{y\sim x: y \in Q} (g^{\lambda, D}_{x, Q}(\omega) - g^{\lambda, D}_{y, Q}(\omega))^2.
\]
Given $A>0$ we define the event
\begin{multline}
\mathcal{A}_{Q} : = \biggl \{\omega: \forall x\in Q: \textrm{dist}(x, \partial^o Q) \geq L_0^{3/4}  \text{ and for }   r\in[L_0^{1/2},L_0^{3/4}],\\ \; \frac{1}{\eps^2 r^2} \sum_{\| y-x\|_2 \leq r} m_{y}(\omega) \geq A |\log(r)|\biggr \}.
\end{multline}
\begin{definition}\label{def:nice}
Fixing $C>c>0$,  we will call a square $Q \in \mathcal{Q}_{L_0}$ \textit{clean} if the following assumptions are satisfied for all squares $Q'\in \mathcal Q^*_{L_0}$ intersecting $Q$:
\begin{itemize}
\item[(C1): ] $\mathcal{A}_R$ occurs.
\item[(C2): ] $\| g^{\lambda, i}_{R} \|_{\infty} \leq  \eps \lambda^{-1/2} |\log(\eps)|^{\eta}$ for $i\in \{N, D\}$ 
\item[(C3): ] $\| \nabla g^{\lambda, i}_{Q'} \|_{\infty} \leq C \eps |\log(\eps)|$ for $i\in \{N, D\}$
\item[(C4): ] $c \eps^2 |\log(\eps)| \leq \frac{\|\nabla g^{\lambda, i}_{Q'} \|^2_2}{|Q'|} \leq C \eps^2 |\log(\eps)|$, $i\in \{N,D\}$
\item[(C5): ] $\|\alpha \|_{\infty} \leq C |\log(\eps)|$
\item[(C6): ] $|\mathcal E_{Q'}(g^{\lambda, N}_{Q'})- \mathcal E_{Q'}(g^{\lambda, D}_{Q'})|\leq  \frac{\|\nabla g^{\lambda, N}_{Q'}\|^2_2}{(\log (L_0))^{1/4}}$.
\end{itemize}
where we choose $\eta\in (0,\frac{\eta_{\lambda}}{2})$ and recall that $\lambda=\eps^2|\log(\eps)|^{1+\eta_{\lambda}}$,  $\eta_{\lambda} \in (2s,\frac{1}{8})$ and $s\in (0,\frac{1}{32})$. 
\end{definition}
The above conditions quantify the bounds we need to be able to perform surgery at the boundary of contours.  We will later show that, provided we choose $A, c$ sufficiently small and $C$ sufficiently large, there is $\delta, \eps_0>0$ so that the probability that one of these conditions fails in a given box is exponentially unlikely in $|\log (\eps)|^{\delta}$ if $\eps\in (0, \eps_0)$, see Section \ref{S:RandSub}.

We will call $Q\in \mathcal Q_{L_0}$ \textit{dirty} if it is not clean.  We introduce the following indicator function:
\begin{equation*}
\Xi_{L_0} (Q) = \begin{cases}
1 & \text{ if } Q \text{ is clean} \\
0 & \text{ if } Q \text{ is dirty}.
\end{cases}
\end{equation*}

Since the $\alpha, g^{\lambda, i}_{Q_{\eta}}$ are unbounded, specifying whether a box is clean or dirty is insufficient to control their contribution to the energy of a spin configuration in a given region.  Thus our next task is to define \textit{regular regions}.  Regular regions  may contain dirty boxes, but their overall influence on the energy of the region will be mild. 
Let $Y\subset \mathbb{Z}^2$ be a $L_0$-measurable bounded and connected set. If $f, g$ are functions on $\mathcal Q_{L_0}$,
let
\[
\langle f,g \rangle_Y = \sum_{Q\in \mathcal Q_{L_0}: Q \subseteq Y} f(Q) g(Q).
\]
Let  $N^Y_{L_0} = |Y|/L_0^2$ be the number of boxes in an $L_0$ measurable cover of $Y$. For $Q\in \mathcal Q_{L_0}$ consider the functions
\begin{equation*}\label{def:F}
F_{\lambda, L_0}(Q):=  \max_{\substack{i\in \{D, N\},\\ Q'\in \mathcal Q^*_{L_0}:\\ Q'\cap Q\neq \varnothing}} \frac{\| g^{\lambda, i}_{Q'}\|^2_2}{L_0^2}, \, \, F^{\nabla}_{\lambda, L_0}( Q):= \max_{\substack{i\in \{D, N\}, \\Q'\in \mathcal Q^*_{L_0}:\\ Q'\cap Q\neq \varnothing}} \frac{\| \nabla g^{\lambda, i}_{Q'}\|^2_2}{L_0^2}, \, \, \, F_{\infty, L_0}(Q) :=  \max_{\substack{Q'\in \mathcal Q^*_{L_0}:\\ Q'\cap Q\neq \varnothing}} \| \alpha \|_{\infty, Q'}.
\end{equation*}

For the next definition, we set $\rho= 5/4$ and choose $\zeta, \chi \in (0, 1/16)$.  Again these parameters will be fixed throughout.
\begin{definition}\label{def:regularY}
We call a bounded and connected set $Y$ controlled if for some $\lambda$  the following four conditions are satisfied:
\begin{itemize}

\item[(R0): ]  $\langle 1-\Xi_{L_0}, 1 \rangle_Y  \leq |\log(\eps)|^{-\rho} N^Y_{L_0}$
\item[(R1): ] $ \langle F^{\nabla}_{\lambda, L_0}, \1_{F^{\nabla}_{\lambda, L_0} \geq \eps^2 |\log(\eps)|^{1+\chi}} \rangle_Y  \leq \eps^2 |\log(\eps)|^{\zeta} N^Y_{L_0} $
\item[(R2): ] $\langle F_{\lambda, L_0},\1_{F_{\lambda, L_0} \geq \eps^2 \lambda^{-1}  |\log(\eps)|^{\chi}}\rangle_Y  \leq  \eps^2 \lambda^{-1} |\log(\eps)|^{\zeta} N^Y_{L_0} $
\item[(R3): ] $\langle F_{\infty, L_0}, \1_{F_{\infty, L_0} \geq |\log(\eps)|^{\chi}} \rangle_Y  \leq  |\log(\eps)|^{\zeta} N^Y_{L_0}$.
\end{itemize}
\end{definition}

\begin{definition}[Regular region]\label{def:clean}
We call an $L_0$-measurable, connected and bounded set $Y$ regular if  the thickened set $\delta(Y)$ is controlled at scales $L_0 \in \{ \ell/2, \ell, L/16\} $. Otherwise $Y$ is called bad.
\end{definition}
Given $\omega \in \Omega$, let us define 
\[
\mathcal{D}_{\omega} = \{ \mathcal{Y} \subset \Z^2: \mathcal{Y} \text{ connected }, \text{ $L$-measurable and  one of $\ell/2$-, $\ell$- or $L/16$-\text{bad}}\}
\]
and
\begin{equation}\label{eq:dirtyregion}
\mathbb{D}_{\omega} = \bigcup_{\mathcal{Y} \in \mathcal{D}_{\omega}} \textrm{cl}(\mathcal{Y}), \, \, \, \,  \mathbb{D}_{\Lambda} = \mathbb{D}_{\omega} \cap \Lambda
\end{equation}
where $\Lambda \subset \Z^2$.

\begin{lemma}\label{lem:dirty}
There exists $\eps_0>0$, so that for the parameters $\rho, \eta, \chi, \zeta$ fixed as above, there is $\delta \in (0,1)$ so that
 we have
\[
Var(|\mathbb{D}_{\Lambda}|) \lesssim e^{-c |\log (\eps)|^{\delta}} |\Lambda|.
\]
Consequently,
$\mathbb{P}$-a.s. there exists $N_0$ such that for all $N\geq N_0$ and  any $\Lambda_N \uparrow \mathbb{Z}^2$
\[
|\mathbb{D}_{\Lambda_N}| \lesssim e^{-c |\log (\eps)|^{\delta}} |\Lambda_N|.
\]
\end{lemma}

We would like to illustrate dirty and clean boxes in the following figure.

\begin{figure}[htb]
\centering
\includegraphics[scale=0.5]{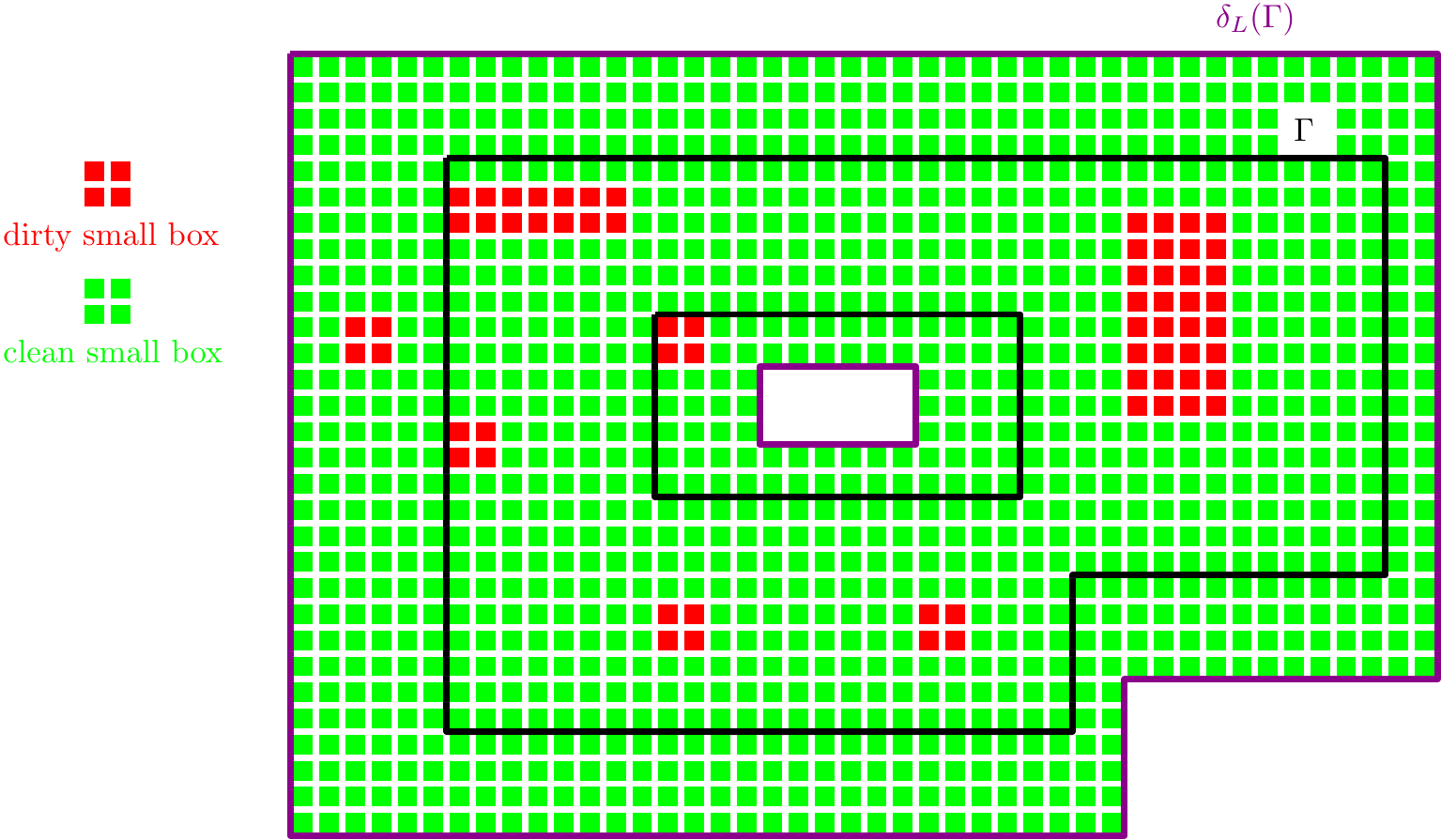}
\caption{Example dirty and clean boxes in $\delta(\Gamma)\setminus \Gamma$}
\end{figure}

\section{Surgery on a Contour}\label{S:surgery}
\subsection{Bulk of a regular contour}\label{sec:bulkmod}

Let  $\overline{\delta}(\Gamma) = \delta_{L/2}(\Gamma)\cap \Lambda_N$  and denote
\[
\mathcal{Q}_{L_0}(\Gamma) = \{ Q \in \mathcal{Q}_{L_0}: Q \cap \overline{\delta}(\Gamma) \neq \varnothing \}.
\]
Given a region $R\subset \Lambda_N$, define the following boundary condition $\textrm{ext}$:
\begin{equation}\label{def:ext}
\textrm{ext}_x = \begin{cases}  e_1 & \text{ if } x \in \partial^o R \cap \Lambda_N^c  \\
 \varnothing & \text{  otherwise }
\end{cases}
\end{equation}
which is equal to the free boundary condition if the outer boundary of a region $R$ does not touch the complement of $\Lambda_N$ and $e_1$ otherwise. 

We want to stitch together approximate ground states defined separately on the microscopic blocks $Q \in \mathcal{Q}_{\ell/2}(\Gamma)$ to make a reference configuration $\bar{\sigma}$ defined on $\bar{\delta}({\Gamma})$.

If a box $Q$  is dirty, we cannot control $g^{\lambda, i}_Q$ well enough to exhibit a gain from the random field.  We therefore set $\overline{\sigma}_{y}:=e_1$ for all $y \in Q$.  Below we will verify that the possible energetic cost from this will be small since we are working on regular contours. 
On the other hand if a box is clean, then we can take advantage of the random field fluctuations.  We set
$\theta_y := g^{\lambda, D}_{Q, y}$
and let
\begin{equation}\label{def:sigma-g}
\overline{\sigma}_{y}: = (\cos(\theta_y), \sin(\theta_y)) \text{ for all $y\in Q$.}
\end{equation}

Let us set $$Max_{Q}(\mathcal{H}, \textrm{ext})=\max_{\sigma \in \mathcal S_Q} \{- \mathcal{H}_Q(\sigma|\textrm{ext}) \}$$ where $\textrm{ext}$ is the boundary condition defined in \eqref{def:ext}.
\begin{lemma}\label{lem:5.1}
Let $\Gamma$ be a clean contour and $\overline{\delta}(\Gamma)$ defined above. There exists $\eps_0 >0$ and constants $C, c >0$ such that if $0<\eps < \eps_0$,  and any $S\subset \bar{\delta}(\Gamma)$ which is $\ell/2$-measurable,
\begin{equation}
\label{E:5.1}
\sum_{Q \in \mathcal{Q}_{\ell/2}(\Gamma):Q\subseteq S} Max_{Q}(\mathcal{H}, \textrm{ext}) + \mathcal{H}_{S}(\overline{\sigma})  \lesssim \,  \eps^2|\log(\eps)|^{ 5/8} |\Gamma|.
\end{equation}
In addition, if $\Gamma$ is a contour for $\sigma$,  then
\begin{equation}
\label{E:5.12}
- \mathcal{H}_{\overline{\delta}(\Gamma)}(\overline{\sigma}| ext) + \mathcal{H}_{\overline{\delta}(\Gamma)}(\sigma|ext)\gtrsim \xi^2 \eps^2 |\log (\eps)|^{1-4s} |\Gamma|,
\end{equation}
where $s\in (0, 1/32)$.
\end{lemma}
\begin{proof}
To prove inequality \eqref{E:5.1}, let 
\[
Max_{Q}(\mathcal{H}, \varnothing)=\max_{\sigma \in Q} \{- \mathcal{H}_Q(\sigma)\}
\]
be the maximum of the Hamiltonian w.r.t. free boundary conditions on $Q$.  
Then
\begin{equation}\label{eq:max-sup}
\begin{split}
&\sum_{\substack{Q \in \mathcal{Q}_{\ell/2}(\Gamma)\\ \textrm{s.t. }Q\subseteq S}} Max_{Q}(\mathcal{H}, \textrm{ext}) + \mathcal{H}_{S}(\overline{\sigma}) \\
& \leq \sum_{\substack{ Q \in \mathcal{Q}_{\ell/2}(\Gamma)\\\textrm{s.t. } Q\subseteq S}} [Max_{Q}(\mathcal{H}, \varnothing) + \mathcal{H}_{Q}(\overline{\sigma}) ]  +  \sum_{Q\neq Q' \in  \mathcal{Q}_{\ell/2}(\Gamma)Q,Q'\subseteq S} \sum_{{e=\langle x,y\rangle; \atop x\in Q, y \in Q'}} \| \nabla_e \overline{\sigma} \|_2^2.
\end{split}
\end{equation}

Let us deal with the second term on the RHS first.  Since $\bar{\sigma}\equiv e_1$ in dirty boxes and since boundary values of $g^{\lambda, D}_{Q} $ may be interpreted as gradients,
\[
\sum_{\substack{Q\neq Q' \in  \mathcal{Q}_{\ell/2}(\Gamma),\\Q, Q'\subseteq S}} \sum_{{e=\langle x,y\rangle; \atop x\in Q, y \in Q'}} \| \nabla_e \overline{\sigma} \|_2^2\lesssim \sum_{\substack{Q \in \mathcal{Q}_{\ell/2}(\Gamma)\\ \textrm{s.t. }Q\subseteq S}}\Xi_{\ell/2}({Q})\ell \| \nabla g^{\lambda, D}_{Q} \|_{\infty}^2\leq (\eps |\log(\eps)|)^3 |S|.
\]

For clean contributions to the first sum on the RHS of \eqref{eq:max-sup}, by  telescoping we have
\begin{equation*}
\begin{split}
& Max_{Q}(\mathcal{H}, \varnothing) +\mathcal{H}_Q(\overline{\sigma}) \\
& = \underbrace{ \mathcal{H}_Q(\overline{\sigma})+\frac{1}{2} \mathcal{E}_Q(g^{\lambda, D}_Q) }_A  +\underbrace{  -\frac{1}{2}[\mathcal{E}_Q(g^{\lambda, D}_Q) +  \mathcal{E}_Q(g^{\lambda, N}_Q) ]}_C+ \underbrace{- \frac{1}{2} \mathcal{E}_Q(g^{\lambda, N}_Q) + Max_{Q}(\mathcal{H}, \varnothing)}_B . 
\end{split}
\end{equation*}
To bound these terms, let us observe that by making the substitutions $\eps \alpha=(-\Delta^i_Q+\lambda) \cdot g^{\lambda, i}_Q$ and summing by parts we obtain
\begin{align}
\label{E:SubN}
- \mathcal{H}_Q(\sigma)&= \frac{1}{2} \mathcal{E}_Q(g^{\lambda, N}_Q) -\frac{1}{2} \mathcal{E}_Q(e_1\cdot \sigma) -\frac{1}{2} \mathcal{E}_Q(e_2\cdot \sigma-g^{\lambda, N}_Q)+ \lambda \sum_{x\in Q} e_2\cdot {\sigma}_x  g^{\lambda, N}_{Q,x},\\
&= \frac{1}{2} \mathcal{E}_Q(g^{\lambda, D}_Q) -\frac{1}{2} \mathcal{E}_Q(e_1\cdot \sigma) -\frac{1}{2} \mathcal{E}_Q(e_2\cdot \sigma-g^{\lambda, D}_Q)+ \lambda \sum_{x\in Q} e_2\cdot {\sigma}_x  g^{\lambda, D}_{Q,x}+\sum_{x\in \partial^i Q} n_x g^{\lambda, D}_{Q, x} e_2 \cdot \sigma_x
\label{E:SubD}
\end{align}
where in the last term on the second line, $n_x$ is the number of external neighbors of $x$.

For Term  $A$, we use \eqref{E:SubD} with $\sigma$ replaced by $\bar{\sigma}$ to obtain
\[
|A|\lesssim  \lambda \left|\sum_{x\in Q} e_2\cdot \overline{\sigma}_x  g^{\lambda, D}_{Q,x}+ +\sum_{x\in \partial^i Q} n_x g^{\lambda, D}_{Q, x} e_2 \cdot \sigma_x\right|+\|g^{\lambda, D}_Q\|_{\infty}\|g^{\lambda, D}_Q\|_2^2.\]
Since $Q$ is clean, we may bound the RHS as
\begin{equation*}
\label{termA}
|A| \lesssim \eps^{2} |\log(\eps)|^{a} |Q|,
\end{equation*}
where $a=1/2 +\eta- \eta_{\lambda}$.

To bound Term $B$ we use \eqref{E:SubN} to obtain that for ANY $\sigma \in \mathcal S_Q$, 
\begin{equation}
\label{E:A1}
- \mathcal{H}_Q(\sigma)- \frac{1}{2} \mathcal{E}_Q(g^{\lambda, N}_Q) =-\frac{1}{2} \mathcal{E}_Q(e_1\cdot \sigma) -\frac{1}{2} \mathcal{E}_Q(e_2\cdot \sigma-g^{\lambda, N}_Q)+ \lambda \sum_{x\in Q} e_2\cdot {\sigma}_x  g^{\lambda, N}_{Q,x}.
\end{equation}
Hence
\[
0\leq B\leq \lambda |Q| \|  g^{\lambda, N}_{Q}\|_{\infty} \leq \eps^{2} |\log(\eps)|^{b} |Q|.
\]
where $b=1/2+ \eta+\eta_{\lambda}/2$.
Finally, to bound Term $C$ we use (C6) from Definition  \ref{def:nice}.

Now we plug these estimates into the sums in equation \eqref{eq:max-sup} by splitting according to clean and dirty boxes.  For the first sum we have
\beq
\label{e:sum1}
\begin{split}
& \sum_{Q \in \mathcal{Q}_{\ell/2}(\Gamma):Q\subseteq S } \left ( Max_{Q}(\mathcal{H}, \varnothing) +\mathcal{H}_Q(\overline{\sigma}) \right ) \\
&= \sum_{Q \in \mathcal{Q}_{\ell/2}(\Gamma):Q\subseteq S } \left ( Max_{Q}(\mathcal{H}, \varnothing) +\mathcal{H}_Q(\overline{\sigma}) \right )\Xi_{\ell/2}(Q) + \sum_{Q \in \mathcal{Q}_{\ell/2}(\Gamma) } Max_{Q}(\mathcal{H}, \varnothing)  [1-\Xi_{\ell/2}({Q})].
\end{split}
\eeq
The largest power of the log-term from our estimates is $b=\eta + \eta_{\lambda}/2+1/2$. We obtain for the first sum on the RHS of \eqref{e:sum1}
\[
\sum_{Q \in \mathcal{Q}_{\ell/2}(\Gamma):Q\subseteq S } \left ( Max_{Q}(\mathcal{H}, \varnothing) +\mathcal{H}_Q(\overline{\sigma}) \right ) \Xi(Q ) \lesssim \eps^2 |\log(\eps)|^{b} |\Gamma|.
\]
For the second sum, over $Q$ such that $\Xi_{\ell/2}(Q) = 0$, we use \eqref{E:SubN}. This shows that, for ANY $\sigma$, 
\[
0\leq Max_{Q}(\mathcal{H}, \varnothing)\lesssim \frac{1}{2} \mathcal{E}_Q(g^{\lambda, N}_Q)+ \lambda\ell \|g^{\lambda, N}_{Q}\|_2.
\]
Using Definition \ref{def:regularY},
\[
 \sum_{Q \in \mathcal{Q}_{\ell/2}(\Gamma):Q\subseteq S }  Max_{Q}(\mathcal{H}, \varnothing)  [1-\Xi_{\ell/2}(Q)] \lesssim \eps^2 |\log(\eps)|^{c}  |\Gamma|
\]
where $c=\max\{1+\chi-\rho, \zeta, (1+\zeta+\eta_{\lambda})/2, (1+\chi+\eta_{\lambda})/2 - \rho\}$.
 Putting everything together we obtain:
\[
\sum_{Q \in \mathcal{Q}_{\ell/2}(\Gamma):Q\subseteq S } \left ( Max_{Q}(\mathcal{H}, \varnothing) + \mathcal{H}_Q(\overline{\sigma}) \right) \lesssim \eps^2 |\log(\eps)|^{d} |\Gamma|
\]
where $d=\max\{a, b, c\}$. Our choice of parameters implies $d\leq 5/8$.  

To prove inequality \eqref{E:5.12},
let
\begin{align*}
\mathcal{Q}_{\mathcal{E}}(\Gamma) & = \left \{ Q\in \mathcal{Q}_L:\: Q\subset sp(\Gamma); \exists z\in Q: \, \psi^{0}_z(\sigma)=0 \right \}, \\
\mathcal{Q}_{Ave}(\Gamma) & = \left \{  Q\in \mathcal{Q}_L: Q\subset sp(\Gamma); \exists z\in Q: \, \psi^{1, \xi}_z(\sigma)=0 \right \} .
\end{align*}
By definition, if $Q \in \mathcal{Q}_{\mathcal{E}}(\Gamma)$ there is a square of side-length $\ell$, $Q_{\ell}$, such that $\textrm{dist}(Q, Q_{\ell}) \leq 2(L + \ell)$ and
\[
\mathcal{E}_{Q_{\ell}} (\sigma) \geq\eps^2 |\log(\eps)|^{1+\chi} \ell^2.
\]
The square $Q_{\ell}$ need not be in $\mathcal Q_{\ell}$ or $\mathcal Q_{\ell}^*$, but 
its existence means we can find $Q'\in \mathcal{Q}_{\ell/2} \cup \mathcal{Q}_{\ell/2}^*$ so that $Q\subset sp(\Gamma)$,
\[
\mathcal{E}_{Q'} (\sigma) \geq \frac{\eps^2}{5} |\log(\eps)|^{1+\chi} \ell^2,
\]
 and  $\textrm{dist}(Q', Q) \leq 3(L+ \ell)$.
Similarly, if $Q \in \mathcal{Q}_{Ave}(\Gamma)$ there is $Q''\in  \mathcal{Q}_{\ell/2}$ so that $Q\subset sp(\Gamma)$,
\[
|\sigma (Q'') \cdot e_1| < 1-\xi/4
\]
and $\textrm{dist}(Q'', Q) \leq 3(L+ \ell)$.

Let us define
\begin{align*}
A_e & = \left \{ Q\in \mathcal Q_{\ell}:  Q \subset \bar{\delta}(\Gamma),  \mathcal{E}_{Q} (\sigma) \geq \frac{\eps^2}{5} |\log(\eps)|^{1+\chi} \ell^2
 \right \}, \\
A_* & = \left \{ Q\in \mathcal Q_{\ell}^*: Q \subset \bar{\delta}(\Gamma), \mathcal{E}_{Q} (\sigma) \geq \frac{\eps^2}{5} |\log(\eps)|^{1+\chi} \ell^2
 \right \}, \\
A_{av} & = \left \{ Q\in \mathcal Q_{\ell}: Q \subset \bar{\delta}(\Gamma),  |\sigma (Q_{\ell}) \cdot e_1| \leq 1-{\xi}/4\right \}.
\end{align*}
If we set $R_1=\cup_{Q\in A_e\cup A_*} Q, R_2=\cup_{\substack{Q\in A_{ave}:\\ Q\cap R_1=\varnothing}} Q$ and $R_3:=\bar{\delta}(\Gamma)\backslash (R_1\cup R_2) $.  Then the previous remarks and the fact that $\Gamma$ is a contour for $\sigma$ give
\begin{equation}
\label{E:s}
\frac{|R_1\cup R_2|}{|\bar{\delta}(\Gamma)|}\gtrsim \frac{\ell^{2} }{L^{2}}=|\log (\eps)|^{-4s}.
\end{equation}
This estimate will be relevant momentarily.

Now, turning to the claimed estimate \eqref{E:5.12}, we start with the bound
\begin{multline}
-\mathcal{H}_{\bar{\delta}(\Gamma)} (\overline{\sigma}|ext) + \mathcal{H}_{\bar{\delta}(\Gamma)} ({\sigma}|ext) \geq [-\mathcal{H}_{R_1} (\overline{\sigma}|ext) 
+ \mathcal{H}_{R_1} ({\sigma}|ext)]+ [-\mathcal{H}_{R_2} (\overline{\sigma}|ext) + \mathcal{H}_{R_2} ({\sigma}|ext)]\\
 +[-\mathcal{H}_{R_3} (\overline{\sigma}|ext) + \mathcal{H}_{R_3} ({\sigma}|ext)] - C\eps^{2} |\log(\eps)|^{5/8}|\Gamma|.\nonumber
\end{multline}
In this inequality,  to control the interaction energy of $\overline{\sigma}$ between the disjoint regions $R_1, R_2, R_3$,  we used in clean squares that $\|\nabla g^{\lambda, N}_Q\|_{\infty} \leq \eps |\log(\eps)|$, that $\frac{|\partial R_i|}{|R_i|}\leq C/\ell$ and that $\overline{\sigma}\equiv e_1$ in dirty squares.  In addition we dropped the corresponding terms for $\sigma$ as they are nonnegative. 
Next, since 
$$-\mathcal{H}_{R_3} (\overline{\sigma}|ext) \leq \max_{\sigma\in \mathcal S_{R_3}} -\mathcal{H}_{R_3} (\overline{\sigma}|ext)\leq \sum_{Q \in \mathcal{Q}_{\ell/2}(\Gamma):Q\subseteq R_3 } Max_{Q}(\mathcal{H}, \varnothing),$$ 
Inequality \eqref{E:5.1} implies 
\[
[-\mathcal{H}_{R_3} (\overline{\sigma}|ext) + \mathcal{H}_{R_3} ({\sigma}|ext)] \gtrsim - C\eps^{2} |\log(\eps)|^{5/8}|\Gamma|.
\]

Now $R_1, R_2$ are $\ell/2$-measurable.  If we tile $R_1$ by squares $Q\in \mathcal Q_{\ell/2}$ and introduce $g^{\lambda, N}_Q$ to complete squares as in \eqref{E:SubN} we have for, any $\sigma$,
\[
\mathcal{H}_{R_1} (\sigma|ext)\geq \mathcal E_{R_1}(\sigma)-C (\eps^{2} |\log(\eps)| |R_1|+ \eps^{2} |\log(\eps)|^{5/8} |\Gamma|),
\]
where regularity was used to control terms involving $g^{\lambda, N}_Q$'s.
In particular
\[
|\mathcal{H}_{R_1} (\overline{\sigma}|ext)|\leq C\eps^{2} |\log(\eps)| |R_1|
\]
and 
\[
-\mathcal{H}_{R_1} (\overline{\sigma}|ext) + \mathcal{H}_{R_1} ({\sigma}|ext)\gtrsim  \eps^2 |\log(\eps)|^{1+\chi} |R_1|- \eps^{2} |\log(\eps)| |R_1|- \eps^{2} |\log(\eps)|^{5/8} |\Gamma|.
\]

To estimate the contribution from $R_2$ we need to refine the device of "completing squares" to control $-\mathcal{H}_Q$.  Let $\mathscr R_{a}$ denote the rotation by angle $a$ on $\mathbb S^{1}$.  For each $Q$ in the sum on the RHS consider the map $\sigma_x\mapsto \sigma_x'=\mathscr R_{-(e_1\cdot \sigma_x) g^{\lambda, D}_{Q, x}} \cdot \sigma_x $.  In terms of $\sigma_x'$ let
\[
-\mathcal{K}_Q(\sigma'|ext)=-\frac 12 \mathcal E_Q(\sigma') +\sum_{x \in Q} \frac{1}{4}(e_1\cdot \sigma'_x)^{2} m_x-\frac 12 \sum_{x\in Q, y\in \Lambda_N^{c}, x\sim y}\|\sigma'_x-e_1\|_2^{2}
\]
where $m_x = \sum_{y\sim x, y \in Q} (g^{\lambda, D}_{Q, y} - g^{\lambda, D}_{Q, x})^2$.  In Lemma \ref{lem:cov} we show  that
\beq
\label{E:cov1}
| \mathcal{H}_Q(\sigma|ext)-\mathcal{K}_Q(\sigma'|ext)|\lesssim \|g^{\lambda, D}_{Q}\|_{\infty}[ \mathcal E_Q(\sigma')+ \mathcal E_Q(g^{\lambda, D}_{Q})+ \lambda \ell^2].
\eeq
Using that $Q$ is clean and that $Q\notin A_e\cup A_s$, the RHS of \eqref{E:cov1} is bounded by 
$$\eps^2\ell^2[|\log (\eps)|^{1/2+\chi +\eta -\eta_{\lambda}/2} + |\log(\eps)|^{1/2+\eta_{\lambda}/2+\eta}].$$

Now Inequality \eqref{E:5.1} implies
\[
-\mathcal{H}_{R_2} (\overline{\sigma}|ext) + \mathcal{H}_{R_2} ({\sigma}|ext)\geq \sum_{Q\in A_{av}} \Xi_{\ell}(Q) [Max_Q(\mathcal{H},ext) + \mathcal{H}_Q(\sigma|ext)]-C\eps^{2} |\log(\eps)|^{5/8}|\Gamma|.
\]
Since $Q\in A_{ave}$,  $\max_{\sigma''}\{- \mathcal{K}_Q(\sigma''|ext)\} =\sum_{x \in Q} \frac{1}{4} m_x$ (and the maximizer is $e_1$) and the event $\mathcal A_Q$ occurs (cf. Definition \ref{def:nice}) we have that
\[
\max_{\sigma''} \{- \mathcal{K}_Q(\sigma''|ext) \} +\mathcal{K}_Q(\sigma'|ext)\gtrsim \xi \eps^2|\log (\eps)| \ell^2.
\]
Using \eqref{E:cov1}, if $Q$ is clean we have
\[
\begin{split}
& Max_Q(\mathcal{H},ext) + \mathcal{H}_Q(\sigma|ext) \gtrsim \\
& \xi \eps^2 |\log (\eps)|\ell^2 -\eps^2[|\log (\eps)|^{1/2+\chi +\eta -\eta_{\lambda}/2} + |\log(\eps)|^{1/2+\eta_{\lambda}/2 +\eta}]\ell^2
\end{split}
\]
so that
\[
\sum_{Q\in A_{av}} \Xi(Q) [Max_Q(\mathcal{H},ext) + \mathcal{H}_Q(\sigma|ext)] \gtrsim \xi \eps^2 |\log \eps| \ell^2 |\{Q\in A_{av}: \Xi(Q) = 1\}|
\]
since $\max\{1/2+\chi +\eta -\eta_{\lambda}/2,  1/2+\eta_{\lambda}/2 +\eta \} \leq \frac{5}{8} $.
$\Gamma$ is regular, therefore
\[
\ell^2 |\{Q\in A_{av}: \Xi(Q) = 1\}|\geq  |R_2|-  |\log(\eps)|^{-\rho}|\Gamma|.
\]
Combining these estimates, recalling that $s\in (0, 1/32)$, $\rho=5/4$ and using \eqref{E:s} we obtain
\[
-\mathcal{H}_{sp(\Gamma)} (\overline{\sigma}|ext) + \mathcal{H}_{sp(\Gamma)} ({\sigma}|ext) \gtrsim   \xi \eps^2 |\log \eps|^{1-4s} |\Gamma|.
\]
\end{proof}

\subsection{Boundary Surgery}\label{sec:BoundSur}

In the previous section we constructed an idealized spin configuration, $\bar{\sigma}$, on the bulk of a given contour $\Gamma$.  If $\sigma$ is a spin configuration having $\Gamma$ as a concrete contour, we want to replace $\sigma$ by $\pm \bar{\sigma}$.  In order to take advantage of the potential savings in energy cost from this replacement, we need to modify $\sigma$ on the boundary of a contour in order to match it with $ \bar{\sigma}$ on the boundary of $\bar{\delta}(\Gamma)$ and show that the energy cost of the modification does not exceed the gain from $\bar{\sigma}$.

To do this surgery, we need to introduce more notation for various regions on the boundary of the contour.
For a given contour $\Gamma$, $\sigma \in \mathbb{X}(\Gamma)$ let
\begin{equation*}
\mathfrak{C}(\Gamma) = \delta_{L} (\Gamma) \cap \textrm{sp}(\Gamma)^c \cap \Lambda_N, \, \, \mathfrak{C}^{\pm} (\Gamma)=\{ z\in \mathfrak{C}(\Gamma): \Psi^{\xi}_z(\sigma)=\pm 1 \},
\end{equation*}
where $\Psi$ was defined in \eqref{def: Psi}. 
Further we isolate a middle portion of the thickened boundary:
\begin{equation}\label{eq:Mpm}
\begin{split}
\mathfrak{M}(\Gamma) & = \left\{x\in \mathfrak{C}(\Gamma): \textrm{dist}(x, \partial^i \textrm{sp}(\Gamma))\in \left[\frac{L}{8}, \frac{3L}{8}\right] \right\}, \, \, \, \mathcal{M} =\{ Q\in \mathcal{Q}_{L/16}:  \textrm{dist}(Q, \mathfrak{M}(\Gamma)) \leq 3\}\\
M(\Gamma) & = \bigcup_{Q \in \mathcal{M}} Q, \, \, \, M^{\pm}(\Gamma) = M(\Gamma) \cap \mathfrak{C}^{\pm}(\Gamma),
\end{split}
\end{equation}
and resp. $\mathfrak{M}^{\pm}(\Gamma) = \mathfrak{M}(\Gamma) \cap  \mathfrak{C}^{\pm}(\Gamma)$.

An example illustrating the collar and middle strip can be found in Figure \ref{fig:collar}.

\begin{figure}[htb]
\centering
\includegraphics[scale=0.6]{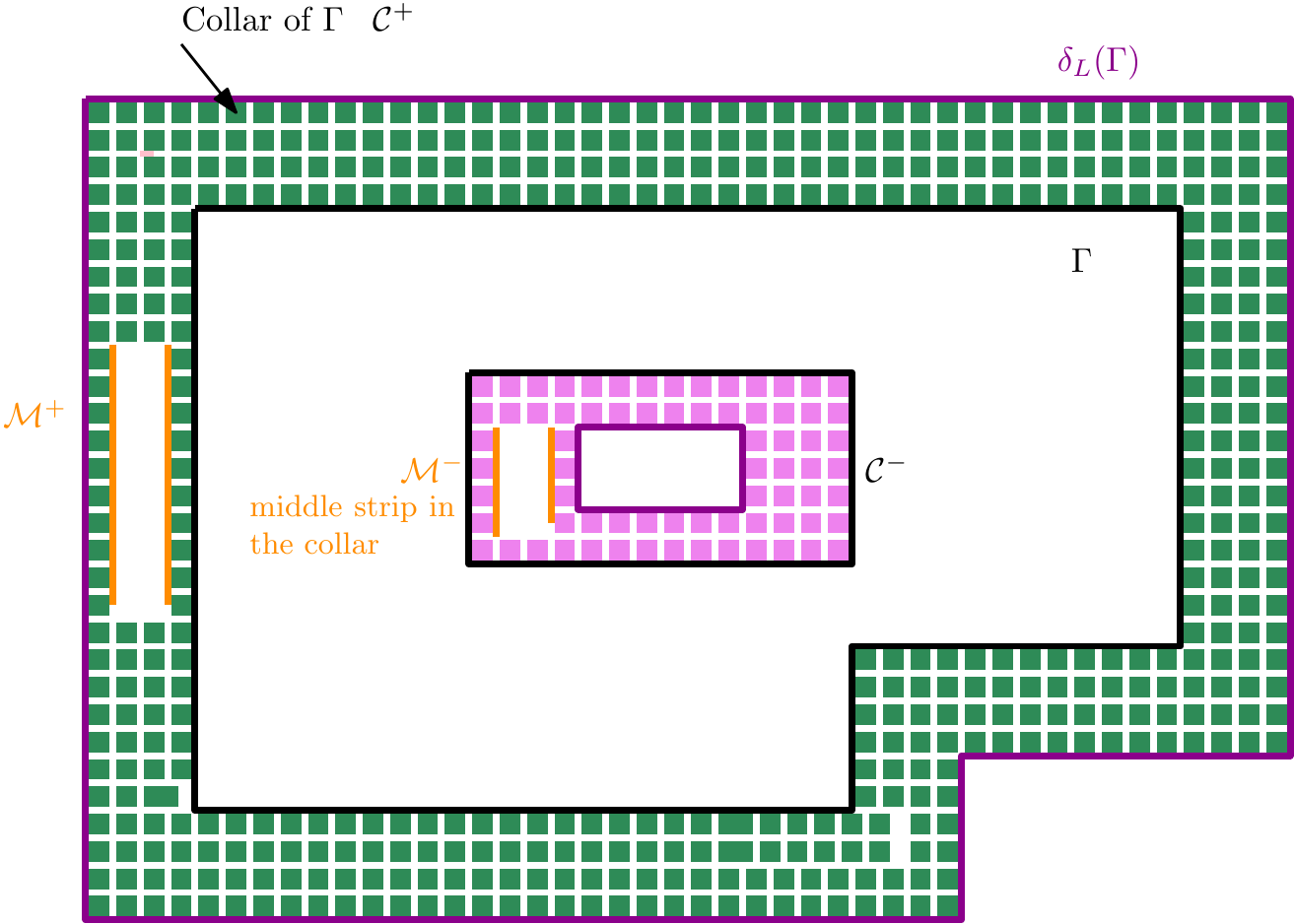}
\caption{Example of the collar and middle strip}
\label{fig:collar}
\end{figure}

The next lemma summarizes the contents of this subsection into one statement.  It combines Lemma \ref{lem:mod3}, whose lengthy proof is postponed to Section \ref{S:Tech},  and  Lemma \ref{lem:mod67}.
\begin{lemma}\label{lem:boundcollar}
Fix $\sigma$ and let $\Gamma$ be a  contour for $\sigma$. There exists $\xi_0, \eps_0(\xi_0), C >0$ such that for $\eps \in (0, \eps_0), \xi \in (0, \xi_0)$ we can construct a configuration $\sigma_x^{\mathfrak{C}}$ such that
\begin{enumerate}
\item If $x\notin  \mathfrak{C}(\Gamma) $ then $\sigma_x =\sigma_x^{\mathfrak{C}}$. 
\item If $x \in \partial^o \mathfrak{M}^{\pm}(\Gamma)$ then $\| \sigma^{\mathfrak{C}}_x \mp e_1 \|_2 \leq  |\log(\eps)|^{-1/2+\eta}$, where $\eta \in (0,\eta_{\lambda}/2), \eta_{\lambda} \in (2s,1/8)$ and $s\in(0,1/32)$.
\item Whenever the boundary condition $e_1$ is compatible with $\Gamma$ and $\textrm{sp}(\Gamma) \subset \Lambda_N$,
\[
- \mathcal{H}_{\Lambda_N}(\sigma^{\mathfrak{C}} | e_1) \geq - \mathcal{H}_{\Lambda_N}(\sigma | e_1)- C \eps^2 |\log(\eps)|^{ 5/8} |\Gamma|.
\]
\end{enumerate}
\end{lemma}

\noindent
\textbf{Modification $1$: flipping spins in boundary component to same halfspace.}

Let $R\subset \bar{\delta}(\Gamma)\backslash \Gamma$ be a  set so that $ \textrm{dist}(R, \Gamma) \geq 3\ell$. We denote by $\mathfrak{A}(R)$ to be the smallest set containing $R$ having the property that for each connected component $\mathfrak c$ of $\mathfrak{A}(R)$
\[
\forall x \in \partial^o \mathfrak c, \, \, \textrm{sgn}(\sigma_x \cdot e_1) \text{ is constant and }|\sigma_x \cdot e_1|\geq 9/10.
\]
\begin{proposition}
\label{prop:defect}
There exist $\xi_0>0$, $\eps_0(\xi_0)>0$ so that if $\eps \in (0, \eps_0)$ and $\xi \in (0, \xi_0)$ we have that
\[
\mathfrak{A}(R) \cup \partial^o \mathfrak{A}(R) \subset \{ x \in \Z^2: \textrm{dist}(x, R) \leq 3\ell \}\subset\delta_{L/2}(\Gamma)\backslash \textrm{sp}(\Gamma).
\]
\end{proposition}
This is a straightforward consequence of the relatively large energetic cost of point defects in low dimension.  We will use this lemma several times in the boundary surgery.  

As a first application,  modifying $\sigma$ we define $\sigma^{(1)}$ in the following way. Let $\mathfrak{A}^{\pm}$ denote $\mathfrak{A}(\mathfrak M^{\pm}(\Gamma))$.  If $x \in \mathfrak{A}^{\pm}$ but $\textrm{sgn}(\sigma_x \cdot e_1)=\mp 1$, we reflect it across the $e_2$ axis and otherwise leave it as is. So, for example, $x \in \mathfrak{A}^+$ and
\begin{equation}\label{mod1}
\sigma^{(1)}_x  = \begin{cases}
 -(e_1 \cdot  \sigma_x) \: e_1+ (e_2 \cdot \sigma_x) \: e_2& \text{ if } \textrm{sgn}(\sigma_x\cdot e_1) =-1 \\
  \sigma_x  & \text{ if } \textrm{sgn}(\sigma_x\cdot e_1) = +1.
\end{cases}
\end{equation}

We record the main properties of this modification used below.  The proof is straightforward and omitted.  
\begin{lemma}\label{lem:flip}
Given $\sigma \in \mathbb{X}(\Gamma)$ and $\sigma^{(1)}$ defined in \eqref{mod1} :
\begin{enumerate}
\item $\sigma^{(1)} \in \mathbb{X}(\Gamma)$
\item If $x \in M^{\pm}(\Gamma)$ and $\sigma_x \cdot e_1 \neq 0$ then $\textrm{sgn}(\sigma^{(1)}_x\cdot e_1)=\pm 1$.
\item For any $Q \subseteq M$ such that $Q\in \mathcal Q_{\ell}$, we have
\[
|\sigma^{(1)} (Q) \cdot e_1| \geq |\sigma (Q) \cdot e_1|.
\] 
\item If $\Gamma$ is compatible with the boundary condition $e_1$, $R\subset \Lambda_N$ we have
\[
- \mathcal{H}_{\Lambda_N}(\sigma^{(1)} | e_1) \geq - \mathcal{H}_{\Lambda_N}(\sigma|e_1)
\]
and $\mathcal{E}_R(\sigma^{(1)}) \leq \mathcal{E}_R(\sigma)$.
\end{enumerate}
\end{lemma}

\noindent
\textbf{Modification $2$: Forcing spins near dirty boxes towards ground states.}

Let
\[
D := \bigcup_{{Q \in \mathcal M \atop \Xi_{L/16}(Q)=0 }} Q,
\] 
denote the union of all relevant dirty boxes. We split $D$ into the subsets 
\begin{equation}\label{def:dirtyD}
D^{\pm} := D\cap \mathfrak{C}^{\pm}.
\end{equation}
and define the neighborhoods
\begin{equation*}
\mathfrak{D}^{\pm} := \left \{ x:  \textrm{dist}(x, D^{\pm}) \leq L/16 \right \}.
\end{equation*}
Note that $\mathfrak D^{\pm}\subseteq \bar{\delta}(\Gamma)$

We define the  linear interpolation $\tau: \Z^2 \rightarrow [0,1]$ forcing angles near dirty boxes towards ground states.  Let
\begin{equation*}
\tau(x) = \begin{cases} 
\frac{16 \cdot \textrm{dist}(x, {D}^{\pm}_{L/16})}{L} \wedge 1 & \text{ if } x\in  \mathfrak{D}^{\pm}_{L/16} \setminus{D}^{\pm}, \\
1 & \text{ if } x\notin \mathfrak{D}^{\pm}_{L/16}, \\
0 & \text{ otherwise}.
\end{cases}
\end{equation*}

We define a second modification $\sigma^{(2)}$ by leaving $\sigma^{(1)}$ alone  outside of $\mathfrak{D}^{\pm}_{L/16}$ and applying the transformation $\tau$ inside.  Let us represent $\sigma^{(1)}=(\cos(\theta_x), \sin (\theta_x))$ where $\theta_x\in [- \frac \pi3, \frac \pi3]$ if $x\in  \mathfrak{D}^{+}$ and $\theta_x\in [\frac{2\pi}{3}, \frac{4\pi}{3}]$ if $x\in  \mathfrak{D}^{-}$.  Then we set
\begin{equation*}
\sigma^{(2)}_x = \begin{cases} 
\sigma^{(1)}_x & \text{ if } x \notin \mathfrak{D}^{\pm}, \\
(\cos(\tau(x) \theta_x), \sin(\tau(x) \theta_x)) & \text{ if } x\in \mathfrak{D}^+, \\
(\cos(\tau(x) [\theta_x-\pi] + \pi), \sin(\tau(x) [\theta_x-\pi] + \pi)) & \text{ if } x\in \mathfrak{D}^-.
\end{cases}
\end{equation*}
\begin{lemma}\label{lem:mod2}
Let $\Gamma$ be a clean contour. There exists $\eps_0>0$ such that for $\eps \in (0, \eps_0)$ small enough
\begin{equation*}\label{eq:mod2}
\left |-\mathcal{H}_{\Lambda_N}(\sigma^{(2)}|e_1) + \mathcal{H}_{\Lambda_N}(\sigma^{(1)}|e_1) \right | \leq \eps^2 |\log(\eps)|^{5/8} |\Gamma|.
\end{equation*}
\end{lemma}

\begin{proof}
Let 
\[
\mathcal D^{\pm} := \{ Q \in \mathcal{Q}_{L/16}: Q \cap \mathfrak D^{\pm} \neq \varnothing \}. 
\]
For $Q\in \mathcal{D}^{\pm}$, we may  complete the square as in \eqref{E:SubN}. Then since $\sigma^{(1)} = \sigma^{(2)}$ on $[\mathfrak D^{\pm}]^{c}$, 
\begin{equation}\label{eq:mod2}
\begin{split}
\left | -\mathcal{H}_{\Lambda_N}(\sigma^{(2)}|e_1) + \mathcal{H}_{\Lambda_N}(\sigma^{(1)}|e_1) \right | & \lesssim  \underbrace{\mathcal{E}_{\mathfrak D^+}(\sigma)}_{I_1} +  \underbrace{\sum_{Q\in \mathcal{D}^+} \mathcal{E}_{Q} \left(g^{\lambda, N}_{Q} \right)}_{I_2} + \underbrace{\sum_{Q\in \mathcal{D}^+} \lambda \| g^{\lambda, N}_{Q} \|_2L}_{I_3} \\&+ \mathcal{E}_{\mathfrak D^-}(\sigma) +  \sum_{Q\in \mathcal{D}^-}[ \mathcal{E}_{Q} \left(g^{\lambda, N}_{Q} \right) + \lambda \| g^{\lambda, N}_{Q} \|_2L],
\end{split}
\end{equation}
where we absorbed the contributions from edges between pairs $Q\neq Q'$ into the first term $I_1$.
Since $\Psi^\xi_z(\sigma) \equiv 1$ on $\mathfrak{C}^+(\Gamma)$ we have
\[
\mathcal{E}_{\mathfrak D^+}(\sigma) \lesssim \eps^2 |\log(\eps)| |\mathfrak D^+|.
\]

Next, using that $\Gamma$ is a regular contour, the term $I_2$ in \eqref{eq:mod2} can be bounded, using (R1) from Definition \ref{def:regularY}, by
\[
\begin{split}
 L^2 \sum_{Q \in \mathcal{D}^+} F_{\lambda, L}^\nabla(Q) \1_{\{F_{\lambda, L}^\nabla(Q)  \leq \eps^2 |\log(\eps)|^{1+\chi} \}} &+ L^2 \sum_{Q \in \mathcal{D}^+} F_{\lambda, L}^\nabla(Q)  \1_{\{F_{\lambda, L}^\nabla(Q)  > \eps^2 |\log(\eps)|^{1+\chi} \}} \\
& \lesssim (|\log(\eps)|^{1+\chi}+ |\log(\eps)|^{\zeta} )\eps^2 |\mathfrak{D}^+|,
\end{split}
\]
where $\chi, \zeta \in (0,\frac{1}{16})$.
Finally  using (R2), we bound
\[
\begin{split}
\| g^{\lambda, N}_{Q} \|^2_{2, \mathcal{D}^+}  & \lesssim L^2 \sum_{Q \in \mathcal{D}^+} F_{\lambda, L}(g) \1_{\{F_{\lambda, L}(g) \leq \eps^2\lambda^{-1} |\log(\eps)|^{\chi} \}} + L^2 \sum_{Q \in \mathcal{D}^+} F_{\lambda, L}( g) \1_{\{F_{\lambda, L}( g) > \eps^2 \lambda^{-1} |\log(\eps)|^{\chi} \}} \\
& \lesssim [ |\log(\eps)|^{\chi} + |\log(\eps)|^{\zeta}]\eps^2 \lambda^{-1}|\mathfrak{D}^+|. 
\end{split}
\]
Using (R0)
\[
|\mathfrak{D}^+| \lesssim |\log(\eps)|^{-\rho} |\Gamma|,
\]
where we recall $\rho =5/4$, so that the  term $I_2+I_3$ in \eqref{eq:mod2} is bounded by $\eps^2 |\log(\eps)|^{\max\{\chi, \zeta\}-\rho} |\Gamma|$.

Analogous estimates hold on $\mathfrak D^-$ and the claim follows from the constraints we imposed on exponents.
\end{proof}

\noindent
\textbf{Modification $3$: Energetic relaxation on the boundary.}
Let \[
\mathfrak{G}^{\pm} = M^{\pm} \setminus D^{\pm}, \quad \mathcal G^{\pm}=\{Q \in \mathcal Q^*_{L/16}: Q\subseteq \mathfrak{G}^{\pm}\},
\]
where $M^{\pm}$ was defined in \eqref{eq:Mpm}, $D^{\pm}$ in \eqref{def:dirtyD} and $\mathcal Q^*_{L_0}$ in \eqref{def:setQ}.
Note that $\sigma^{(2)}= \pm e_1$ on $D^{\pm}$.  This will be important in what follows. 
Let $\mathcal Q^{\pm}=\{Q\in \mathcal Q^*_{L/16}: Q \subseteq  \mathfrak G^{\pm}\}$, 
For simplicity of exposition, let us focus on the regions $\mathfrak{G}^{+}, \mathcal Q^{+}$.
Enumerate $\mathcal Q^+$ in some particular order $(Q_i)_{i=1}^N$.
We are going to perform a sequence of energetic relaxations on the $Q_i$'s in order to pass from the spin configuration $\sigma^{(2)}$ through a sequence $\sigma^{(2)}=\bar{\sigma}^{(0)}, \bar{\sigma}^{(1)}, \dotsc, \bar{\sigma}^{(N)}=:\sigma^{(3)}$ to a new spin configuration $\sigma^{(3)}$ over which we have better control.  

By construction, the squares $Q_i$ are all clean so that the conditions stated in Definition \ref{def:nice} are satisfied.
Moreover, by definition, for any square $Q_{\ell}$ within distance $\ell$ of $Q_i$, $$\mathcal E_{Q_{\ell}}({\sigma}^{(2)})\leq \eps^2 |\log (\eps)|^{1+\chi} \ell^2\text{ and }e_1\cdot \sigma^{(2)} (Q_{\ell})\geq 1-\xi.$$ Therefore using Proposition  \ref{prop:defect}, if $R^0_i=\{x\in Q_i: \textrm{dist}(x, \partial^o Q_i \geq 3\ell\}$ then there is a connected set $R_i$ so that $R_i^0\subseteq R_i \subseteq Q_i$ and $\sigma_x \cdot e_1\geq 9/10$ for all $x\in \partial^o R_i$

Suppose that after going through the boxes $Q_j, j\leq i$, we arrive at the spin configuration $\bar{\sigma}^{(i)}=(\cos(\theta^{(i)}), \sin(\theta^{(i)}))$.
For $x\in Q_{i+1}$, we define the change of variables $\theta^{(i)}\mapsto \phi^{(i+1)}$ by  $\phi^{(i+1)}_x=\theta^{(i)}_x - \cos(\theta^{(i)}_x)g^{\lambda, D}_{x, Q_{i+1}}$.  In the new variables, with the angles on $\partial^o Q_{i+1}$ given by $\phi^{(i+1)}_x$, we set, for $R\subset Q_{i+1}$,

\[
\begin{split}
- \mathcal{K}_{R}( \phi|\phi^{(i+1)}) & = \sum_{\substack{\langle x,y\rangle: \\ x, y\in R}} (\cos( \phi_x- \phi_y)-1) + \frac{1}{4}\sum_{x\in R} m_x \cos^2( \phi_x) \\
&+  \sum_{\substack{\langle x,  y \rangle:\\ x\in R,\\ y\in\partial^o R}} (\cos( \phi_x-\phi^{(i+1)}_y)-1)
\end{split}
\]
where
\[
m_x = \sum_{\substack{\langle x,y\rangle: \\ y\in Q_{i+1}}} (g^{\lambda, D}_{x, Q_{i+1}}  -  g^{\lambda, D}_{y, Q_{i+1}})^2
\]
if $x\in Q_{i+1}$.
We now optimize $- \mathcal{K}_{R_{i+1}}( \cdot |\phi^{(i+1)})$.  Later we show there is a unique  optimizer of $- \mathcal{K}_{R_{i+1}}( \cdot |{\phi}^{(i+1)})$, denoted by  $\nu^{(i+1)}$.  Let
\[
\bar{\phi}^{(i+1)}_x = 
\begin{cases} \nu^{(i+1)}_x  \text{ if } x\in  R_{i+1}\\
\phi^{(i+1)}_x   \text{ if } x\in  Q_{i+1} \backslash R_{i+1}.
\end{cases}
\]
Transforming back to the first set of coordinates, let $\theta^{(i+1)}_x$ be determined through the formula
\[
\bar{\phi}^{(i+1)}_x =\theta^{(i+1)}_x - \cos(\theta^{(i+1)}_x) g^{\lambda, D}_{x, Q_{i+1}}
\]
for $x\in Q_{i+1}$ and $\theta^{(i+1)}_x=\theta^{(i)}_x$ otherwise.  Let $\bar{\sigma}^{(i+1)}$ denote the associated spin configuration.
Finally, we set $\sigma^{(3)}_x$ on $\mathfrak G^+$ to be the spin configuration obtained after exhausting this procedure over $Q\in \mathcal Q^+.$

Given a boundary condition $\tau$ of a region $R$ let
\begin{equation*}\label{eq:DirTau}
\mathcal E_R( \sigma|\tau)=\sum_{\langle x, y\rangle \cap R\neq \varnothing} \|\sigma_x-\sigma_y||_2^2
\end{equation*}
where $\sigma$ is set equal to $\tau$ on $\partial^oR$.  The key technical step of our construction is encapsulated in the following lemma to be proved in Section \ref{S:Tech}.
\begin{lemma}\label{lem:mod3}
There is $C>0$ such that for all $i \in \{1,...,N\}$,
\[
\mathcal E_{Q_{i}}( \bar{\sigma}^{(i)}|\bar{\sigma}^{(i-1)}) \leq 2 [ \mathcal E_{Q_{i}}( \bar{\sigma}^{(i-1)}|\bar{\sigma}^{(i-1)})+ \mathcal E_{Q_{i}}(  g^{\lambda, D}_{Q_{i}}|0)]
\]
and
 \begin{equation*}
\label{E:mod3}
 -\mathcal{H}_{\Lambda_N}(\bar\sigma^{(i)}|e_1) \geq - \mathcal{H}_{\Lambda_N}(\bar\sigma^{(i-1)}|e_1)-  C \eps^2 |\log(\eps)|^{5/8}|Q_i|.
\end{equation*}

Moreover for any $\delta$ sufficiently small the following holds: defining $\mathcal R_i:=\cup_{j\leq i} R_j $ then for all $x\in \mathcal R_i$ such that $\textrm{dist}(x,  \partial^o \mathcal R_i)\geq 6 \eps^{-1} |\log (\eps)|^{-1/2+\delta}$, $|e_2\cdot \bar\sigma^{(i)}_x| \leq 32 |\log (\eps)|^{-1/2}$.
\end{lemma}

\noindent
\textbf{Modification $4$: Forcing spins in the middle of a boundary strip towards $\pm e_1$}
Next we force spins in a thin substrip of $\mathfrak M(\Gamma)$ to point exactly along $\pm e_1$ so that we can match $\sigma^{(3)}$ with the configuration $\bar{\sigma}$ defined in Section \ref{sec:bulkmod}, equation \eqref{def:sigma-g}.

Let $(c^{\Gamma}_i)_{i\in [I]}$ be the collection of connected components of $\mathfrak M(\Gamma)$.  By construction these sets may be viewed as  thickened Jordan curves in $\mathbb{Z}^2$ having inner radius $L/8$. Each $c^{\Gamma}_i$ can be split into connected layers $(L_{j, c^{\Gamma}_{i}})_{j\in J}$, chosen so that $|J|\geq \sqrt{|\log (\eps)|}$ and
\[
\frac{|L_{j, c^{\Gamma}_{i}}|}
{|c^{\Gamma}_i|}\in \left (\frac{1}{ 2 |\log (\eps)|^{1/2}}, \frac{2}{ |\log (\eps)|^{1/2}}\right ).
\] 
We also stipulate that each layer is a union of dyadic squares with side length at least $\frac{L}{2^6 |\log(\eps)|^{1/2}}$.
Enumerating the layers in some order,  let $\mathcal L_j = \bigcup_{i\in [I]} L_{j, c^{\Gamma}_{i}}$.

Recall from Lemma \ref{lem:mod3} that there is a universal constant $C>$ such that
\[
\mathcal{E}_{\mathfrak{C}(\Gamma)} (\sigma^{(3)}) \leq C \eps^2 |\log(\eps)|^{1+\chi} |\mathfrak{C}(\Gamma)|.
\]
We call the union of layers $\mathcal L_{j}$ \textit{good} if 
\begin{equation*}\label{goodlayer}
\mathcal{E}_{\mathcal L_{ j}}(\sigma^{(3)}) \leq  4 C \eps^2 |\log(\eps)|^{1+\chi} |\mathcal L_j| .
\end{equation*}
Since 
\[
\sum_{j} \mathcal{E}_{\mathcal L_{ j}}(\sigma^{(3)})\leq \mathcal{E}_{\mathfrak{C}(\Gamma)} (\sigma^{(3)})
\]
there must be at least one good union of layers.

Fix any good union of layers $\mathcal L_0=\mathcal L_{j_0}$, and let
\[
\mathcal L_{mid} = \left\{x\in \mathcal L_{0}: \frac{L}{2^7 |\log(\eps)|^{1/2}}-100\leq  \textrm{dist}(x, \partial^i \mathcal L_{0})\right \}.
\]
We are going to modify $\sigma^{(3)}, \bar{\sigma}$ to obtain configurations which are $\pm e_1$ on $\mathcal L_{mid}\cap \mathfrak C^\pm$. 
Let us define
\begin{equation*}
\tau(x) = \begin{cases} 
\frac{128\, |\log (\eps)|^{1/2} \textrm{dist}(x,\partial^o \mathcal L_{mid} )}{L }\wedge 1, & \text{ if } x\in \mathcal L_{0}, \\
1 & \text{ otherwise. }
\end{cases}
\end{equation*}

Let $\theta_x$, respectively $\theta'_x$, be the angles defined by $\sigma^{(3)}_x = (\cos(\theta_x), \sin(\theta_x))$, respectively $\bar{\sigma}_x$. We set 
\begin{align*}
&\sigma^{(4)}_{x} = \begin{cases} (\cos(\tau(x) \theta_x), \sin(\tau(x) \theta_x)), &\text{ if } x\in \mathcal L_{0} \cap \mathfrak{C}^+(\Gamma), \\
(\cos(\tau(x) (\theta_x-\pi)+\pi), \sin(\tau(x) (\theta_x-\pi)+\pi)), &\text{ if } x\in \mathcal L_{0} \cap \mathfrak{C}^-(\Gamma),
\end{cases}
\\
&\sigma^{(5)}_{x} = \begin{cases} (\cos(\tau(x) \theta_x'), \sin(\tau(x) \theta_x')), &\text{ if } x\in \mathcal L_{0} \cap \mathfrak{C}^+(\Gamma), \\
(\cos(\tau(x) (\theta_x'-\pi)+\pi), \sin(\tau(x) (\theta_x'-\pi)+\pi)), &\text{ if } x\in \mathcal L_{0} \cap \mathfrak{C}^-(\Gamma).
\end{cases}
\end{align*}
We finally set
\begin{align*}
&{\sigma}^{\mathfrak{C}}_x  = \begin{cases}
\sigma^{(4)}_{x} & \text{ if } x\in \mathcal L_{0}\\
\sigma^{(3)}_x & \text{ otherwise }
\end{cases}
\\
&\bar{\sigma}^{\mathfrak{C}}_x = \begin{cases}
\sigma^{(5)}_{x} & \text{ if } x\in \mathcal L_{0}\\
\bar{\sigma}_x & \text{ if } x\in \bar{\delta}(\Gamma)\backslash \mathcal L_0. 
\end{cases}
\end{align*}

\begin{lemma}\label{lem:mod67}
For $\eps$ small enough, if  $\Gamma$ is regular, we have that
\begin{align*}
&| \mathcal{H}_{\Lambda_N} (\sigma^{(3)}| e_1) -  \mathcal{H}_{\Lambda_N} (\sigma^{\mathfrak{C}}| e_1)| \leq \eps^2 |\log(\eps)|^{5/8} |\Gamma|,\\
&| \mathcal{H}_{\bar{\delta}(\Gamma)} (\bar{\sigma}| e_1) -  \mathcal{H}_{\bar{\delta}(\Gamma)} (\bar{\sigma}^{\mathfrak{C}}| e_1)| \leq \eps^2 |\log(\eps)|^{5/8} |\Gamma|.
\end{align*}
\end{lemma}

\begin{proof}
Let $W=\{ x: \exists y, y\sim x, \sigma^{\mathfrak{C}}_y \neq \sigma^{(3)}_y \}$ and notice that $W\subset \mathcal L_0\cap D^c$, where $D$ was defined in \eqref{def:dirtyD}.
We can write 

\begin{equation}\label{eq_H3}
| \mathcal{H}_{\Lambda_N} (\sigma^{(3)}| e_1) -  \mathcal{H}_{\Lambda_N} (\sigma^{\mathfrak{C}}| e_1)| \lesssim \underbrace{ \left | \mathcal{E}_W(\sigma^{(3)}) - \mathcal{E}_W(\sigma^{\mathfrak{C}})\right |}_{\textrm{I}} + \underbrace{\left | \sum_{x\in \mathfrak{C}(\Gamma) }  \eps \alpha_x (\sigma^{(3)}_x-\sigma_x^{\mathfrak{C}})\cdot e_2 \right |}_{\textrm{II}}.
\end{equation}
Using the fact that $|e_2\cdot \sigma_x^{(3)}|\leq 12 |\log (\eps)|^{-1/2}$, the first term satisfies the bound
\begin{equation}\label{I}
\textrm{I} \lesssim \mathcal{E}_W(\sigma^{(3)}) + |\log (\eps)|^{-1}\mathcal{E}_W(\tau).
\end{equation}
We have
\[
\mathcal{E}_W(\tau)\lesssim  |\log (\eps)| L^{-2} |W|\lesssim \eps^2|\log (\eps)|^{1/2- 2s} |\Gamma|
\]
 and also, since $W$ is a subset of a good layer
\[
\mathcal{E}_W(\sigma^{(3)}) \lesssim  \eps^2 |\log(\eps)|^{1+\chi} |W|\lesssim  \eps^2 |\log(\eps)|^{1/2+\chi}  |\Gamma|.
\]

To bound Term II in \eqref{eq_H3} we replace $\eps \alpha_x$  by $(\Delta^N_Q +\lambda) g^{\lambda, N}_{x, Q}$ as in \eqref{E:SubN} and use a summation-by-parts to estimate
\[
\textrm{II}    \lesssim  \sum_{\substack{Q\in \mathcal{Q}_{L/16} \\ Q \cap W \neq \varnothing}} \sum_{e\subset W} \left | \nabla_e g^{\lambda, N}_Q  e_2\cdot  \nabla_e(\sigma^{(3)} - \sigma^{\mathfrak{C}}) \right | + |\log (\eps)|^{-1/2} \sum_{\substack{Q\in \mathcal{Q}_{L/16} \\ Q\cap W\neq \varnothing}} \lambda \sum_{x\in Q} |g^{\lambda, N}_{x,Q} |
\]
where the factor $|\log (\eps)|^{-1/2}$ comes from our bound on $|e_2\cdot  \sigma^{(3)}|$ in $\mathfrak M(\Gamma)$.  An important detail here is that if $Q\cap W\neq \varnothing$ then $Q$ is clean.
This leads to the bound 
\[
| \mathcal{H}_{\Lambda_N} (\sigma^{(3)}| e_1) -  \mathcal{H}_{\Lambda_N} (\sigma^{\mathfrak{C}}| e_1)| \lesssim (|\log(\eps)|^{1/2-2s}+ |\log(\eps)|^{1/2+\chi}  + |\log(\eps)|^{\eta_{\lambda}/2+\eta})\eps^2  |\Gamma|.
\]
Moving on to $\bar{\sigma}$, since $\lambda=\eps^{2}|\log (\eps)|^{1+\eta_{\lambda}}$, using \eqref{def:sigma-g} and (C2) for the second term
\[
\begin{split}
&| \mathcal{H}_{\bar{\delta}(\Gamma)} (\bar{\sigma}| e_1) -  \mathcal{H}_{\bar{\delta}(\Gamma)} (\bar{\sigma}^{\mathfrak{C}}| e_1)| \\
& \lesssim  \mathcal{E}_W(\bar{\sigma})+ |\log \eps|^{-1+ 2\eta-\eta_{\lambda}}\mathcal{E}_W(\tau)+ |\log (\eps)|^{-1/2+\eta-\eta_{\lambda}/2} \sum_{\substack{Q\in \mathcal{Q}_{L/16} \\ Q\cap W\neq \varnothing}} \lambda \sum_{x\in W} |g^{\lambda, N}_{x,Q} |.
\end{split}
\]
As in the previous estimate, this leads to the bound
\[
| \mathcal{H}_{\bar{\delta}(\Gamma)} (\bar{\sigma}| e_1) -  \mathcal{H}_{\bar{\delta}(\Gamma)} (\bar{\sigma}^{\mathfrak{C}}| e_1)|\lesssim (|\log(\eps)|^{1/2+\chi}+ |\log(\eps)|^{1/2-2s+2\eta - \eta_{\lambda}/2}  + |\log(\eps)|^{2\eta})\eps^2  |\Gamma|.
\]
Finally, the claim follows from the constraints placed on the various exponents.
\end{proof}

\noindent
\subsection{Gluing configurations between bulk and boundary.}
In the final step of our construction we need to glue  $\bar{\sigma}^{\mathfrak{C}}, \sigma^{\mathfrak{C}}$ together. Recall that a contour comes with a function $\psi_z(\Gamma)$.  We call a contour $\Gamma$ a $+$-contour if $\psi_z(\Gamma) =+1$ for $z\in \delta_{ext}(\Gamma)$ and resp. a $-$ - contour if $\psi_z(\Gamma) =-1$ for $z\in \delta_{ext}(\Gamma)$, recall $\psi$ from Definition \ref{def:contoursPsi}. For definiteness in the subsequent discussion we assume $\Gamma $ is a $+$-contour.  

Let $\mathcal L_0, \mathcal L_{mid}, \sigma^{\mathfrak{C}}, \bar{\sigma}^{\mathfrak{C}}$ be as in the previous section.  Let $\tilde{\Gamma}$ denote the maximal connected set containing $\textrm{sp}(\Gamma)$ in $\delta(\Gamma)\backslash \mathcal L_{mid}$ together with $\mathcal L_{mid} $.
Let  $\sigma^{\mathfrak C,*}$ be obtained from $\sigma^{\mathfrak C}$ by reflecting ALL spins across the $e_2$-axis on each interior component of $\tilde{\Gamma}^c$ such that $\psi_z(\Gamma) =-1$ adjacent to $\tilde{\Gamma}$.   
We define
\begin{equation}\label{def:Smod}
S^{+}_{\Gamma,y}(\sigma) = 
\begin{cases}
\sigma^{\mathfrak C,*}_y & \text{ if } y\in \tilde{\Gamma}^c, \\
\bar{\sigma}_y^{\mathfrak{C}} & \text{ if } y \in \tilde{\Gamma}.
\end{cases}
\end{equation}
An analogous construction/notation is used  for $-$-contours, $S^{-}_{\Gamma,y}$.

\begin{lemma}\label{lem:Peir}
Let $\Gamma$ be a regular $\pm$-contour and $\sigma \in \mathbb{X}(\Gamma)$.
There exists $\eps_0 \in (0,1)$ such that for $\eps \in (0, \eps_0)$ we have that
\[
- \mathcal{H}_{\Lambda_N}(S^{\pm}_{\Gamma}(\sigma)|e_1) + \mathcal{H}_{\Lambda_N}(\sigma |e_1) \gtrsim \xi^2 \eps^2 |\log(\eps)|^{1-4s},
\]
where $s\in (0,1/32)$.
\end{lemma}
\begin{proof}

 To begin the estimate, Lemma \ref{lem:boundcollar} allows us to replace $\sigma$ with ${\sigma}^{\mathfrak{C}}$ at small cost. We have 
 \[
-\mathcal{H}_{\Lambda_N}(S^{\pm}_{\Gamma}(\sigma)|e_1) + \mathcal{H}_{\Lambda_N}(\sigma|e_1) \geq \left[-\mathcal{H}_{\Lambda_N}(S^{\pm}_{\Gamma}(\sigma)|e_1) + \mathcal{H}_{\Lambda_N}({\sigma}^{\mathfrak{C}} |e_1)\right ]  - C \eps^2 |\log(\eps)|^{5/8}|\Gamma|.
\]
By construction of  $S^{\pm}_{\Gamma}(\sigma)$,
\[
-\mathcal{H}_{\Lambda_N}(S^{\pm}_{\Gamma}(\sigma)|e_1) + \mathcal{H}_{\Lambda_N}({\sigma}^{\mathfrak{C}} |e_1)=-\mathcal{H}_{\tilde{\Gamma}}(\bar{\sigma}^{\mathfrak{C}}|ext) + \mathcal{H}_{\tilde{\Gamma}}({\sigma}^{\mathfrak{C}} |ext)
\]
where the boundary condition $ext$ is $e_1$ on $\tilde{\Gamma}^{c}\cap \Lambda_n^{c}$ and is free otherwise.

If $Q\cap \mathcal L_{mid}\neq \varnothing$ and if $\psi_z(\Gamma)= 1$ on $Q$, let us define $\widetilde{{\sigma}^{\mathfrak{C}}}$ on $Q\backslash \tilde{\Gamma}$ to be $\bar{\sigma}^{\mathfrak{C}}$.  If $\psi_z(\Gamma)=- 1$ we instead  define $\widetilde{{\sigma}^{\mathfrak{C}}}$ to be the reflection of $\bar{\sigma}^{\mathfrak{C}}$ in $Q$ across the $e_2$ axis.
Then we have
\begin{align*}
-\mathcal{H}_{Q\cap \tilde{\Gamma} }(\bar{\sigma}^{\mathfrak{C}}|ext) + \mathcal{H}_{Q \cap \tilde{\Gamma}}({\sigma}^{\mathfrak{C}} |ext)&=  -\mathcal{H}_{Q}(\bar{\sigma}^{\mathfrak{C}}|ext) + \mathcal{H}_{Q}(\widetilde{{\sigma}^{\mathfrak{C}}} |ext)\\&\geq -[Max_{Q}(\mathcal H, ext)+\mathcal{H}_{Q}(\bar{\sigma}^{\mathfrak{C}}|ext)].
\end{align*}
A similar bound applies directly if $Q\subset \tilde{\Gamma} \backslash \textrm{sp}(\Gamma)$.
Therefore we have
\begin{align*}
-\mathcal{H}_{\tilde{\Gamma}}(\bar{\sigma}^{\mathfrak{C}}|ext) + \mathcal{H}_{\tilde{\Gamma}}({\sigma}^{\mathfrak{C}} |ext) &\geq  -\mathcal{H}_{sp(\Gamma)}(\bar{\sigma}|ext) + \mathcal{H}_{sp(\Gamma)}({\sigma} |ext)\\
&-\sum_{Q\in \mathcal Q_{\ell}: Q\cap \tilde{\Gamma}\backslash sp(\Gamma) \neq \varnothing} [Max_{Q}(\mathcal{H}, ext)  +\mathcal{H}_{Q}(\bar{\sigma}^{\mathfrak{C}}|ext)- C \eps^2|\log(\eps)|^{5/8} |\Gamma|.
\end{align*}
where we neglected the interaction between vertices of neighboring regions for ${\sigma}^{\mathfrak{C}}$ and used regularity $\Gamma$ to bound the corresponding terms in $\bar{\sigma}^{\mathfrak{C}}$, see Inequality \eqref{eq:max-sup}.  Using Lemmas \ref{lem:5.1} and \ref{lem:mod67}
\[
\sum_{Q\in \mathcal Q_{\ell}: Q\cap \tilde{\Gamma}\backslash sp(\Gamma) \neq \varnothing} [Max_{Q}(\mathcal{H}, ext)  +\mathcal{H}_{Q}(\bar{\sigma}^{\mathfrak{C}}|ext)]\leq C \eps^2|\log(\eps)|^{5/8} |\Gamma|.
\]
All in all, this gives the bound
\[
-\mathcal{H}_{\Lambda_N}(S^{\pm}_{\Gamma}(\sigma)|e_1) + \mathcal{H}_{\Lambda_N}(\sigma|e_1) \geq -\mathcal{H}_{sp(\Gamma)}(\bar{\sigma}|ext) + \mathcal{H}_{sp(\Gamma)}({\sigma} |ext)-C \eps^2|\log(\eps)|^{5/8} |\Gamma|.
\]
Applying Inequality \eqref{E:5.12}, 
\[
-\mathcal{H}_{\Lambda_N}(S^{\pm}_{\Gamma}(\sigma)|e_1) + \mathcal{H}_{\Lambda_N}(\sigma|e_1) \gtrsim \xi \eps^2|\log(\eps)|^{1-4s} |\Gamma|.
\]
\end{proof}

\section{The Peierls Argument}\label{S:Peierls}

In this section we will develop energy and entropy estimates for the Peierls argument.
Let $\mu_{R}^{\tau}$ denote the finite-volume Gibbs measure on $R$ with boundary condition $\tau$ and Hamiltonian defined in \eqref{def:Ham1}.  
Let $(\Gamma, \psi_z(\Gamma))$ be a $+$-contour.  Let $\delta^{(N)}(\Gamma) = \delta(\Gamma) \cap \Lambda_N$ and consider a configuration $\sigma_{\delta^{(N)}(\Gamma)^c}\in \mathcal{S}_{\delta^{(N)}(\Gamma)^c}$.  
Let us introduce the event $E_{\Gamma}\subset \mathcal{S}_{\delta^{(N)}(\Gamma)}$ given by
\[
E_{\Gamma}=\{\sigma': \forall z\in sp(\Gamma), \, \, \psi_{z}(\sigma)=\psi_z(\Gamma), \Psi_z(\sigma)\neq 0 \text{ on $ \delta(\Gamma) \setminus sp(\Gamma)$}\}
\]
and note that 
$$\mathbb{X}(\Gamma)=E_{\Gamma}\cap \{\Psi_z(\sigma)\neq 0 \text{ on $ \delta(\Gamma) \setminus sp(\Gamma)$}\}.
$$
We call $\sigma_{\delta^{(N)}(\Gamma)^c}$ \textit{compatible with $\Gamma$}, denoted $\sigma_{\delta^{(N)}(\Gamma)^c}\sim \Gamma$, if 
\[
\sigma_{\delta^{(N)}(\Gamma)^c}|_{\Lambda^c_N} \equiv e_1
\]
and if
\[
\mu^{\sigma_{\delta^{(N)}(\Gamma)^c}}_{\delta^{(N)}(\Gamma)} (\mathbb{X}(\Gamma) ) \neq 0.
\]
Note here that  $\mathbb{X}(\Gamma)\notin \mathcal{S}_{\delta^{(N)}}(\Gamma)$ so the previous condition needs to be interpreted using the extended configuration $(\sigma_{\delta^{(N)}(\Gamma)^c}, \sigma_{\delta^{(N)}(\Gamma)})$.

Recall $R^{\pm}=\{z: \Psi_z(\sigma)=\pm 1\}$.  For $\sigma \in \mathbb{X}(\Gamma)$, let  $\mathscr R^+ \cdot \sigma_{\delta^{(N)}(\Gamma)^c}$ be the configuration in $\mathcal S_{\delta^{(N)}(\Gamma)^c}$ obtained by making a global reflection of $\sigma$ on each component of $R^- \cap [c(\Gamma)\backslash \delta^{(N)}(\Gamma)]$.  This transformation is clearly well defined when $\sigma_{\delta^{(N)}(\Gamma)^c}\sim \Gamma$. 
We next consider the function
\begin{equation*}\label{def:W}
W(\Gamma; \sigma_{\delta^{(N)}(\Gamma)^c}) =\frac{Z^{\sigma_{\delta^{(N)}(\Gamma)^c}}_{\delta^{(N)}(\Gamma)}}{Z^{ \mathscr R^+ \cdot (\sigma_{\delta^{(N)}(\Gamma)^c})}_{\delta^{(N)}(\Gamma)}}
 \frac{\mu^{\sigma_{\delta^{(N)}(\Gamma)^c}}_{\delta^{(N)}(\Gamma)} \left (\mathbb{X}(\Gamma) \right )}{\mu^{\mathscr R^+ \cdot \sigma_{\delta^{(N)}(\Gamma)^c}}_{\delta^{(N)}(\Gamma)} \left(\sigma': \forall z\in\delta^{(N)}(\Gamma), \, \, \Psi_z^{\xi}(\sigma')=1 \right)}
\end{equation*}
and let 
\[
\| W(\Gamma; \cdot )\|_{\infty} = \sup_{\sigma_{\delta^{(N)}(\Gamma)^c} \sim \Gamma} W(\Gamma; \sigma_{\delta^{(N)}(\Gamma)^c}).
\]
The function $W(\Gamma; \sigma_{\delta^{(N)}(\Gamma)^c})$ is the ratio of the (not normalized) exponential weights of configurations $(\sigma_{\delta^{(N)}(\Gamma),\delta^{(N)}(\Gamma^c)})$ which are compatible with the contour $\Gamma$ and the compatible configurations which are reflected on each components of $R^- \cap [c(\Gamma)\backslash \delta^{(N)}(\Gamma)]$.

Notions for $-$-contours are
defined similarly with the provisos for $+$ and $-$ reversed.

\begin{lemma}
\label{L:4.1}
There exist $\xi_0, \eps_0, \beta_0, C>0$ such that for $\eps \in (0, \eps_0)$, $\beta \in (\beta_0, \infty)$ and $\xi \in (0,\xi_0)$ and a clean contour $\Gamma$ we have that
\[
\| W(\Gamma; \cdot )\|_{\infty} \leq e^{- c(\eps) \beta |\Gamma|}
\]
where $c(\eps) = C \xi^2 \eps^2 |\log(\eps)|^{1-4s}$.
\end{lemma}

\begin{proof}
Assume that $\Gamma$ is a $+$-contour. Let us fix $\sigma_{\delta^{(N)}(\Gamma)^c}\sim \Gamma$ and let $\sigma \in \mathbb{X}(\Gamma)$ agree with $\sigma_{\delta^{(N)}(\Gamma)^c}$ on $\delta^{(N)}(\Gamma)^c$, so $\sigma= (\sigma_{\delta^{(N)}(\Gamma)}, \sigma_{\delta^{(N)}(\Gamma)^c})$. Let $S^+_{\Gamma}(\sigma)$ be defined as in \eqref{def:Smod}.  Then
\[
S^+_{\Gamma}(\sigma)|_{\delta^{(N)}(\Gamma)^c} \equiv \mathscr R^+ \cdot (\sigma_{\delta^{(N)}(\Gamma)^c}).
\]
We consider the event
\[
F_{\sigma} = \left \{ \sigma' \in \mathcal{S}_{\delta^{(N)}(\Gamma) }: \forall x \in \delta^{(N)}(\Gamma), \, \, \| \sigma_x' -S^+_{\Gamma, x}(\sigma) \| _2 \leq \eps^{3} \right \}.
\]
Then
\[
\mu^{ \mathscr R^+ \cdot (\sigma_{\delta^{(N)}(\Gamma)^c})}_{\delta^{(N)}(\Gamma)} \left(\sigma': \forall z\in\delta^{(N)}(\Gamma), \, \, \Psi_z^{\xi}(\sigma' )=1 \right) \geq \mu^{ \mathscr R^+ \cdot (\sigma_{\delta^{(N)}(\Gamma)^c})}_{\delta^{(N)}(\Gamma)} \left(F_{\sigma} \right).
\]

Let $\nu_{\delta^{(N)}(\Gamma)}$ denote the uniform distribution on $\mathcal S_{\delta^{(N)}(\Gamma)}$.
By Lemma \ref{lem:Peir} we can write
\[
W(\Gamma, \sigma_{\delta^{(N)}(\Gamma)^c}) \leq e^{-\beta q(\eps)|\Gamma|} \int\textrm{d}\nu_{\delta^{(N)}(\Gamma)}(\sigma)\frac{e^{-\beta \mathcal{H} \left(S^+_{\Gamma}(\sigma)|_{\delta^{(N)}(\Gamma)} |  \mathscr R^+ \cdot (\sigma_{\delta^{(N)}(\Gamma)^c}) \right )}}{Z^{ \mathscr R^+ \cdot (\sigma_{\delta^{(N)}(\Gamma)^c})}_{\delta^{(N)}(\Gamma)}\mu^{ \mathscr R^+ \cdot (\sigma_{\delta^{(N)}(\Gamma)^c})}_{\delta^{(N)}(\Gamma)} \left(F_{\sigma} \right)}
\]
where
\[
q(\eps) = C_1 \xi^2 \eps^2 |\log(\eps)|^{1-4s}.
\]
Now, from (C5) in Definition \ref{def:nice} we have
\[
\eps \sum_{x\in \delta(\Gamma)} |\alpha_x| \lesssim \eps |\log(\eps)| |\Gamma|
\]
so that on $F_{\sigma}$,
\[
\left|-\mathcal{H} \left(S^+_{\Gamma}(\sigma)|_{\delta^{(N)}(\Gamma)} |  \mathscr R^+ \cdot (\sigma_{\delta^{(N)}(\Gamma)^c}) \right ) + \mathcal{H} \left(\sigma' |  \mathscr R^+ \cdot(\sigma_{\delta^{(N)}(\Gamma)^c}) \right ) \right| \lesssim \eps^3|\Gamma|.
\]
It follows that for $\eps$ sufficiently small, 
\[
W(\Gamma, \sigma_{\delta^{(N)}(\Gamma)^c}) \leq e^{[-\beta q(\eps)/2+e(\eps)]|\Gamma}|
\]
where $e(\eps)=C_2 \eps |\log(\eps)|$ represents the entropy cost in localizing to $F_{\sigma}$.
\end{proof}
With the previous Lemma \ref{L:4.1} in place we now implement a Peierls argument.
 Define the event
\[
\Gamma^* = \{ \omega: \Gamma \text{ is regular, } \textrm{cl}(\Gamma) \not\subset \mathbb{D}_{\omega} \}
\]
where $\mathbb{D}_{\omega}$ was defined in \eqref{eq:dirtyregion}, and let
\[
\mathbb{X}(\Gamma_1,...,\Gamma_{k}) = \bigcap_{i=1}^{k} \mathbb{X}(\Gamma_i) 
\]
be the set of all spin configurations compatible with the system of contours $(\Gamma_1,...,\Gamma_k)$. 

\begin{lemma}[Peierls estimate]\label{lem:Peierls}
Fix $\xi>0$.  There exists $\eps_0>0$ so that for $0<\eps < \eps_0$ we can find $\beta(\eps)>0$ such that for all $\beta > \beta(\eps)$, on the event $ \cap_{i=1}^m \Gamma_i^*$,
\[
\mu^{e_1}_{\Lambda_N} (\mathbb{X}(\Gamma_1,...,\Gamma_m, \Gamma_{m+1},...,\Gamma_{m+n})) \leq e^{- \beta c(\eps) \sum_{i=1}^{m} |\Gamma_i|}
\]
where $c(\eps) = C \xi^2 \eps^2 |\log(\eps)|^{1-4s}$, $s\in (0,1/32)$.
\end{lemma}

\begin{proof}
We will prove by induction over $m$ that 
\[
\mu^{e_1}_{\Lambda_N} (\mathbb{X}(\Gamma_1,...,\Gamma_m, \Gamma_{m+1},...,\Gamma_{m+n})) \leq  \prod_{i=1}^m  \| W(\Gamma_i; \cdot )\|
\]
on the event $\cap_{i=1}^m \Gamma^*_i$.

For $m=0$ the product is empty. Assume that the statement holds for a fixed $m$.  We prove that the estimate holds when the first $m+1$ contours are clean.  Reordering the contours as needed, we may assume that 
\[
cl(\Gamma_{m+1}) \cap \bigcup_{i=1}^m \textrm{sp}(\Gamma_i) = \varnothing.
\]
and that, by symmetry of our argument, $\Gamma_{m+1}$ is a +-contour. Let $$\{\tilde{\Gamma}_1,...,\tilde{\Gamma}_r \} \subset \{ \Gamma_{m+2},...,\Gamma_{m+n+1}\}$$ denote the subset of contours such that  $\delta(\tilde{\Gamma}_i) \subset \textrm{Int}(\Gamma_{m+1})$.  We relabel the remaining ones by
\[
\{\Gamma'_1,...,\Gamma'_{n-r} \} = \{ \Gamma_{m+2},...,\Gamma_{m+n+1}\} \setminus \{\tilde{\Gamma}_1,...,\tilde{\Gamma}_r \}.
\]
In terms of this rearrangement
\[
\mathbb{X}(\Gamma_1,...,\Gamma_{m+1}, \Gamma_{m+2},..., \Gamma_{m+n+1}) = \mathbb{X}(\Gamma_1,...,\Gamma_{m}, \Gamma'_{1},..., \Gamma'_{n-r}) \cap \mathbb{X}(\Gamma_{m+1}, \tilde{\Gamma}_1,...,\tilde{\Gamma}_r).
\]
The DLR equations imply
\[
\begin{split}
& \mu_{\Lambda_N}^{e_1} (\mathbb{X}(\Gamma_1,...,\Gamma_{m+1}, \Gamma_{m+2},..., \Gamma_{m+n+1}) ) \\
& = \int \1_{\mathbb{X}(\Gamma_1,...,\Gamma_m, \Gamma'_1,...,\Gamma'_{n-r})} Z^{\sigma_{cl(\Gamma_{m+1})^c}}_{\sigma_{cl(\Gamma_{m+1})}}\mu^{\sigma_{cl(\Gamma_{m+1})^c}}_{cl(\Gamma_{m+1})}(\mathbb{X}(\Gamma_{m+1}, \tilde{\Gamma}_1,...,\tilde{\Gamma}_r)) \mu^{e_1}_{\Lambda_N}(d \sigma)
\end{split}
\]
For a given contour $\Gamma$ denote by $- \Gamma$ the pair $(sp(\Gamma), -\psi(\Gamma))$ and
\[
T_{j}(\Gamma) = \begin{cases}
\Gamma, & \text{ if } \delta_{ext}(\Gamma) \subset \delta^+_{in} (\Gamma_{m+1}) \\
- \Gamma & \text{ otherwise}.
\end{cases}
\]
Using the DLR equations we have that
\begin{multline}
Z^{\sigma_{cl(\Gamma_{m+1})^c}}_{\sigma_{cl(\Gamma_{m+1})}} \mu^{\sigma_{cl(\Gamma_{m+1})^c}}_{cl(\Gamma_{m+1})}(\mathbb{X}(\Gamma_{m+1}, \tilde{\Gamma}_1,...,\tilde{\Gamma}_r)) \\
\leq \| W(\Gamma_{m+1}; \cdot) \| Z^{\sigma_{cl(\Gamma_{m+1})^c}}_{\sigma_{cl(\Gamma_{m+1})}} \mu^{ \sigma_{cl(\Gamma_{m+1})^c}}_{cl(\Gamma_{m+1})}(\mathbb{X}( T_{m+1}(\tilde{\Gamma}_1),..., T_{m+1}(\tilde{\Gamma}_r))).\nonumber
\end{multline}
The induction hypothesis then finishes the proof.
\end{proof}

\begin{proof}[Proof of Main Theorem]
The proof of Theorem \ref{T:Main2} can be obtained from the following argument involving Lemma \ref{lem:Peierls} and Lemma \ref{lem:dirty}. Let $x\in \mathbb{Z}^2$ such that $Q_L(x) \cap \mathbb{D}_{\omega} =\varnothing$.  If $\Lambda_N \uparrow \Z^2$ in the van Hove sense then for $N$ sufficiently large $Q_L(x)\subset \Lambda_N$.   Consider the event
\[
E:=\{ \sigma \in \mathcal{S}_{\Lambda_N} : \Psi^{\xi}_{x} (\sigma) \neq 1\}.
\]
If $\sigma \in E$, then since $\sigma|_{\Lambda_N^c}\equiv e_1$, our coarse-graining scheme implies that there exists a largest contour ${\Gamma}$ such that $Q_L(x) \subset \textrm{cl}(\Gamma)$.  Since $Q_L(x) \cap \mathbb{D}_\omega =\varnothing$ and  $\mathbb D_{\omega}=\textrm{cl} (\mathbb D_{\omega})$, ${\Gamma}^*$ must occur. We decompose the event $E$ into disjoint subsets according to possible realizations of the support of the largest contour.  Let $F_{\Gamma}=\{\sigma:  \text{ $\Gamma$ is the largest contour surrounding $Q_L(x)$}\}$.  Then
\[
\mu^{e_1}_{\Lambda_N} \left (\Psi^{\xi}_{x} (\sigma) \neq 1\right ) \leq \sum_{{\substack{\varGamma \text{ connected $L$-measurable}: \\ Q_L(x) \subset \textrm{cl}(\varGamma)}}} \1_{\varGamma^*}(\omega)\sum_{\Gamma=(\varGamma,\psi({\Gamma}))} \mu_{\Lambda_N}^{e_1} (\Psi^{\xi}_{Q_L(x)} (\sigma) \neq 1, F_{\Gamma}).
\]
The inner sum is over the possible choices for $\psi({\Gamma})$ so that $sp(\Gamma)=\varGamma$.  The number of such choices is at most $C_1^{|\varGamma|/\ell^2}$.
Next, there are constants $a, C_2>0$ so that the number of $L$-measurable connected sets $\varGamma$ surrounding a block $Q_L$ and with $N^L_\varGamma =r$ is at most $C_2r^2 a^r$.  By  Lemma \ref{lem:Peierls} we then get that
\[
\mu^{e_1}_{\Lambda_N} \left ( \Psi^{\xi}_{x} (\sigma) \neq 1\right ) \leq \sum_{r=1}^{\infty} C_2 r^2 a^{r} C_1^{\frac{rL^2}{\ell^2}} e^{- c(\eps)\beta r} \lesssim e^{- \frac{1}{2} c(\eps) \beta}
\]
provided $c(\eps)\beta > 2|\log (\eps)|^{4s} \log (C_1)+2\log (a)$. Choosing $\beta$ large enough yields the claim.
\end{proof}

\section{Proofs}

\subsection{Defects}\label{S:defects}
\begin{proof}[Proof of Proposition \ref{prop:defect}]
To prove Proposition \ref{prop:defect}, we begin by showing that for any $\ell$-by-$2\ell$ rectangle $B$ and $\sigma$ so that $\sigma(B)\cdot e_1\geq 1-\xi$ and $\mathcal E_B(\sigma) \leq  \eps^2|\log (\eps)|^{1+\chi} \ell^2$, we can find a line segment $l_1$ with length $2\ell$ so that $\sigma_x\cdot e_1\geq 9/10$ for any $x\in l_1$.  To find $l_1$, we observe that if $\mathcal L$ denote the $\ell$ line segments of length $2\ell$ in $B$, then
\[
\sum_{l\in \mathcal L} \mathcal E_l(\sigma)\leq 2\eps^2|\log (\eps)|^{1+\chi} \ell^2.
\]
This implies 
$$|\{l\in \mathcal L: \mathcal E_l(\sigma)\leq 8\eps^2|\log (\eps)|^{1+2\chi} \ell\}|\geq \frac{3\ell}{4}.$$
Similarly if $\sigma(l)$ denotes the spin average along $l$, then since $\sigma(B)\cdot e_1\in  (1-\xi, 1]$
$$|\{l\in \mathcal L:\sigma(l)\cdot e_1\leq 1-5\xi\}|\leq \ell/4$$
so that we can find (many) $l\in \mathcal L$ so that $\sigma(l)\cdot e_1\geq 1-5\xi$, $\mathcal E_l(\sigma)\leq 8 \eps^2|\log (\eps)|^{1+2\chi} \ell$.  For $\xi, \eps$ sufficiently small, this implies $\sigma_x\cdot e_1\geq 9/10$ for all $x\in l$.
In a similar manner, we can find two lines, $l_2, l_3$, perpendicular to $l_1$, having  of length $\ell$, and located one each in the two disjoint sub-squares of $B$ and such that $\sigma_x\cdot e_1\geq 9/10$ for all $x\in l_i$ and $\mathcal E_{l_i}(\sigma)\leq 2 \eps^2|\log (\eps)|^{1+2\chi} \ell$ for $i=2,3$.

To finish the proof, given $R$, consider the set of $\ell$-measurable  squares $Q\subset R^c$ so that $\textrm{dist}(Q, R)\leq 2\ell$ and let $A$ denote their union.  Then each component of $A$ can be tiled by $\ell$-by-$2\ell$ rectangles $B$ which overlap in $\ell$-measurable squares $Q$.   By hypothesis, $\psi(\sigma)\neq 0$ on each such rectangle and has a constant sign on each component of $A$.   Applying the first paragraph to each component provides boundary circuits on  which $\sigma_x\cdot e_1\geq 9/10$ for $x$ in the circuit.
\end{proof}
\subsection{Proof of Lemma \ref{lem:mod3}}
\label{S:Tech}

\subsubsection{Change of variables and energy estimates}\label{S:aux}
Let us write $\sigma_x=(\cos(\theta_x), \sin(\theta_x))$, where  $\theta_x$ is defined modulo $2\pi$.  For any pair of vertices $x, y$ such that $e=\langle x,y\rangle$, let $|\nabla_e \theta|$ to be the minimal value of $|\theta_x-\theta_y|$ over angles which may represent $\sigma_x, \sigma_y$. Then 
\[ 
(\sigma_x-\sigma_y)^2 \asymp |\nabla_e \theta|^2. 
\]
In this section we consider the effect of the changes of variables
\begin{align*}
&g_x:= g^{\lambda, D}_{x, R}=  (-\Delta^{D}_R + \lambda) ^{-1}\eps \, \alpha_x ,
\quad \theta'_x = \theta_x - \cos(\theta_x)g_{x}.
\end{align*}
Note that the second transformation is well defined no matter how we choose the map $x \mapsto \theta_x$.  It is also non-singular as long as $|g_{x}| < 1$ (which will be enforced following the first lemma). The first lemma expresses the Hamiltonian of a spin configuration in a box in terms of it's Dirichlet energy using only the first transformation.  We will omit its proof since it is analogous to Proposition 8.1 in \cite{Craw_order}.
Recall that
\[
\mathcal E_R( \sigma|\tau)=\sum_{\langle x,y \rangle \cap R\neq \varnothing}\|\sigma_x-\sigma_y\|_2^2
\]
where $\sigma$ is set equal to $\tau$ on $\partial^oR$.

\begin{lemma}\label{lem:cov}
Fix $\ell\in \N$ and let $R\subset \Z^2$ be a bounded, $\ell$-measurable region such that for any $\ell$-measurable square intersecting $R$,  $\|g\|_{\infty}\leq 1$.  Then for any $\lambda > 0$ and boundary condition $\tau$ we have that
\begin{equation*}
\begin{split}
-\mathcal{H}_R(\sigma|\tau)  &=-\mathcal K_R(\theta'|\tau)  + \frac 12 \sum_{\langle x, y\rangle): x\in \partial ^i R, y\in \partial^o R} [g_x-g_y]e_2\cdot \tau_y + Err_R(\sigma),
\end{split}
\end{equation*}
where 
\[
-\mathcal K_R(\theta'|\tau):=\sum_{e \cap R \neq \varnothing} [\cos(\nabla_e \theta') -1] + \sum_{x \in R} \frac{1}{4}\cos^2(\theta'_x) m_x 
\]
with $m_x = \sum_{y\sim x} (g_{x} - g_{y})^2$ and where 
\[
 Err_R(\sigma)=\mathcal{O}\left(\|g\|_{\infty} \mathcal E_R( \sigma|\tau)+ \|g\|_{\infty}\left[ \mathcal E_{R}( g|0)+ \lambda |R|\right]\right).
\]
\end{lemma}

\begin{proof}
The original Hamiltonian can be written  as
\begin{equation}\label{eq_Ham}
-\mathcal{H}_R(\sigma|\tau) = \sum_{e\cap R\neq \varnothing} [\cos(\nabla_e \theta)-1]+ \eps \sum_{x\in R} \alpha_x\sin(\theta_x). \end{equation}

We begin by carefully rewriting the first term on the RHS.
For any edge $\langle x, y \rangle \subset R$ choose $\theta_x, \theta_y$ which achieve $(\nabla_e \theta)^2$ and
let $\theta'_x=\theta_x-g_x'$ where $g'_x=\cos(\theta_x)g_{x}$.
Let $\bar{\theta}_e=\frac 12(\theta_x+ \theta_y)$ so that
\[
\theta_x = \overline{\theta}_e + \frac{1}{2}\nabla_e \theta ; \, \, \qquad \theta_y = \overline{\theta}_e - \frac{1}{2}\nabla_e \theta.
\]
Then
\[
\begin{split}
\nabla_e g' &= \cos(\theta_x) g_x - \cos(\theta_y)g_y \\
& = \cos \left(\overline{\theta}_e + \frac{1}{2}\nabla_e \theta \right) g_x - \cos \left (\overline{\theta}_e - \frac{1}{2}\nabla_e \theta \right)g_y 
\end{split}
\]
so that by Taylor expansion around $\bar{\theta}_e$,
\begin{equation*}\label{eq2}
\begin{split}
\nabla_e g' & = \cos(\overline{\theta}_e) g_x  + \mathcal{O}(|\nabla_e \theta| |g_x|)  - \cos(\overline{\theta}_e) g_y +   \mathcal{O}(|\nabla_e \theta| |g_y|) \\
& = \cos(\overline{\theta}_e) \nabla_e g   + \mathcal{O}(|\nabla_e \theta| (|g_x|+|g_y|)).
\end{split}
\end{equation*}

Using the sum of angles formula for cosine, the identity
\[
\nabla_e \theta' = \nabla_e \theta - \nabla_e g',
\]
Taylor expansions of sine and cosine around $0$, and the sup norm bound on $g$ we get
\begin{align*}\label{eq1}
\cos(\nabla_e \theta')-1 &= \cos(\nabla_e \theta)\cos( \nabla_e g')-1 +  \sin( \nabla_e \theta)\sin( \nabla_e g') \\
\nonumber &= \cos(\nabla_e \theta)-1 -\frac 12  \cos(\bar{\theta}_e)^2(\nabla_e g)^2 +   \cos(\bar{\theta}_e) \nabla_e \theta \nabla_e g+ \mathcal{O}(\| g \|_{\infty} |\nabla_e \theta|^2+ |\nabla g|^3).
\end{align*}

Inverting this relation
\begin{equation*}\label{eq:ABC}
\cos(\nabla_e \theta)-1 = \cos(\nabla_e \theta')-1 +\frac 12  \cos(\bar{\theta}_e)^2(\nabla_e g)^2 -   \cos(\bar{\theta}_e) \nabla_e \theta \nabla_e g+ \mathcal{O}(\| g \|_{\infty} |\nabla_e \theta|^2+ |\nabla g|^3).
\end{equation*}

Now we insert this identity into $\mathcal H_R$, dividing the sum over edges into those interior to $R$ and those edges which have endpoints in $R$ and $R^c$. We get
\begin{align*}
\nonumber \sum_{e \cap R \neq \varnothing} \cos(\nabla_e \theta) 
& = \sum_{e \cap R\neq \varnothing}   \left ( \cos(\nabla_e \theta') - \nabla_e \theta \cos(\overline{\theta}_e) \nabla_e g  + \frac{1}{2} (\cos(\overline{\theta}_e) \nabla_e g  )^2 \right ) + E^{(1)}_R(\sigma)\nonumber
\end{align*}
where
\begin{align}
&E^{(1)}_R(\sigma) =\mathcal O(\|g\|_{\infty} \mathcal E_R( \sigma)+ \|\nabla g\|_{\infty} \mathcal E_{R}( g|0))).
\end{align}

Turning to the random field term in \eqref{eq_Ham}, we have
\begin{equation*}\label{eq:sbp}
\begin{split}
\eps \sum_{x\in R} \alpha_x \sin(\theta_x)& = \sum_{x\in R} (-\Delta^{D}_R +\lambda) \cdot g_x \sin(\theta_x). 
\end{split}
\end{equation*}
Summation by parts gives
\begin{align*}
 \nonumber \sum_{x\in R} (-\Delta^{\lambda, D}_R) \cdot g_x \sin(\theta_x) 
& =  \sum_{e \cap R \neq \varnothing} \nabla_e g \nabla_e \sin(\theta) -\frac 12 \sum_{\substack{\langle x, y \rangle: \\ x \in \partial^i R, y\in \partial^o R}}  [g_x-g_y] e_2 \cdot \tau_y. \\
\end{align*}
Using the estimate
\begin{equation*}
\begin{split}
\sin(\theta_x) - \sin(\theta_y) & = \sin \left(\overline{\theta}_e +\frac{1}{2} \nabla_e \theta \right) - \sin \left(\overline{\theta}_e  - \frac{1}{2} \nabla_e  \theta \right) \\
& = \cos(\overline{\theta}_e) \nabla_e \theta + \mathcal{O}((\nabla_e \theta)^2)
\end{split}
\end{equation*}
we have
\begin{equation*}
\sum_{e \cap R \neq \varnothing} \nabla_e g \nabla_e \sin(\theta) = \sum_{e \cap R \neq \varnothing}   \nabla_e g  \cos(\overline{\theta}_e) \nabla_e \theta + \mathcal{O}(\| \nabla g \|_{2, R} \| \nabla \theta \|^2_{2, R}).
\end{equation*}
By the  Cauchy-Schwartz inequality
\begin{equation*}
\left| \lambda  \sum_{x\in R} g_{x} \sin(\theta_x) \right | \leq   \mathcal{O}(\lambda  \| g \|_{2, R}\sqrt{|R|} ).
\end{equation*}

Combining terms we obtain
\[
-\mathcal{H}_R(\sigma|\tau)= \sum_{\substack{e \cap R \neq \varnothing}}   \left ( \cos(\nabla_e \theta') -1+\frac{1}{2} (\cos(\overline{\theta}_e) \nabla_e g  )^2 \right ) -\frac 12 \sum_{\substack{\langle x, y \rangle:\\ x \in \partial^i R, y\in \partial^o R}}  [g_x-g_y] e_2 \cdot \tau_y \\
+ E^{(2)}_R(\sigma)\nonumber.
\]
where $E^{(2)}_R(\sigma)=E^{(1)}_R(\sigma)+ \mathcal{O}(\lambda  \| g \|_{2, R}\sqrt{|R|} ).$
Since $\theta_x=\theta_x'+g'_x$ we can write
\begin{equation*}
\cos(\overline{\theta}_e)=\cos(\theta'_x) + \mathcal{O}(|g| + |\nabla_e \theta|).
\end{equation*}
Using \begin{equation*}
\begin{split}
  \cos^2(\overline{\theta}_e)  & = \frac{1}{2}(\cos^2(\theta'_x) + \cos^2(\theta'_y)) + \mathcal{O}(|g_x|+|g_y| + |\nabla_e \theta|) ,
\end{split}
\end{equation*}
we have
\begin{align}
\nonumber \frac{1}{2} (\cos(\overline{\theta}_e) \nabla_e g  )^2 
& = \frac{1}{4}(\cos^2(\theta'_x) + \cos^2(\theta'_y))(\nabla_e g )^2 + \mathcal{O}(|\nabla_e g|^2(|g_x|+|g_y| + |\nabla_e \theta|)) \\
\nonumber & = \frac{1}{4} \cos^2(\theta'_x)(\nabla_e g )^2 + \frac{1}{4}\cos^2(\theta'_y)(\nabla_e g )^2 +\mathcal{O}(|\nabla_e g|^2(|g_x|+|g_y|+ |\nabla_e \theta|)).
\end{align}
Hence we arrive at 
\[
-\mathcal{H}_R(\sigma|\tau)= \sum_{\substack{e \cap R \neq \varnothing}}    ( \cos(\nabla_e \theta') -1)+\sum_{x\in  R} \frac{1}{4} (\cos(\theta'_x)m_x)^2 -\frac 12 \sum_{\substack{(x, y):x \in \partial^i R, y\in \partial^o R}}  [g_x-g_y] e_2 \cdot \tau_y
+ Err_R(\sigma)\nonumber
\]
where $m_x = \sum_{y\sim x} (g_{x} - g_{y})^2$.  For reference we note that upon collecting the errors and keeping only the dominant contributions
\[
 Err_R(\sigma)= \mathcal{O}\left(\|g\|_{\infty} \mathcal E_R( \sigma|\tau)+ \|g\|_{\infty} \left[ \mathcal E_{R}( g|0)+ \lambda|R| \right]\right).
\]
\end{proof}
\begin{lemma}(Uniqueness of the Maximers)\label{lem:unique}
Consider optimizers of $-\mathcal{K}_{R}(\cdot|\tau)$ with  boundary condition $\tau$ satisfying $\| \tau \|_{\infty, \partial^o R} \leq \pi/5$. There is a unique maximizer    $(\nu_x)_{x\in R}$ and it satisfies $\| \nu \|_{\infty, R}\leq \| \tau \|_{\infty, R}$.
\end{lemma}
\begin{proof}
Let $(\nu_x)_{x\in R}$ maximize $-\mathcal{K}_{R}$, then it is necessarily the case that $\nu_x \in [-\frac{\pi}{2}, \frac{\pi}{2}]$ since $-\mathcal{K}_{R}(\cdot|\tau)$ is reduced by reflecting spins into the right halfplane of $\mathbb R^2$. 

Now let 
\[
X(\theta) : = \begin{cases} - \left(\theta - \frac{\pi}{5} \right) & \text{ if } \theta > \frac{\pi}{5} \\
- \left(\theta + \frac{\pi}{5} \right) & \text{ if } \theta <- \frac{\pi}{5} \\
0 & \text{ otherwise }
\end{cases}
\]
and consider the flow of angles $\partial_t \theta_x(t) = X(\theta_x(t))$. Take any initial configuration $\theta_x(0) \in [-\frac{\pi}{2}, \frac{\pi}{2}]$, then $-\mathcal{K}_R(\theta(t)|\tau)$ will be strictly decreasing unless the condition $|\theta_x(t)| \leq  \frac{\pi}{5}, \forall x\in R$ is satisfied.

This implies $|\nu_x|\leq \pi/5$ throughout $R$. Now note that $-\mathcal{K}$ is convex if $\nu_x \in [-\frac{\pi}{5}, \frac{\pi}{5}]$ for all $x\in R$.  Hence there is a unique maximizer subject to the given boundary condition.
\end{proof}

\subsubsection{Ground states of auxiliary elliptic PDE}\label{S:Groundstates}
   We analyze the ground states of 
\[
- \mathcal{K}_{R}(\theta|\tau) = \sum_{\substack{e=\langle x,y \rangle, \\ e\cap R\neq \varnothing }} (\cos(\theta_x-\theta_y)-1) + \frac{1}{4}\sum_{x\in R} m_x \cos^2(\theta_x),
\]
where $m_x=\sum_{y \sim x}(g_y-g_x)^2$.
We exploit the fact that stationary points of $\mathcal{K}_{R}(\cdot|\tau)$ are also solutions to discrete elliptic PDE's with random mass, where the elliptic operator $-L_C$ defining the PDE  is given by 
\begin{equation*}
-L_C\cdot f=\sum_{y\sim x} C_{xy} [f_x-f_y]
+V\cdot f_x=V_x f_x
\end{equation*}
where
\[
C_{xy}=\frac{\sin( \nu_x-\nu_y)}{\nu_x-\nu_y},\quad V_x=\frac {\sin(\nu_x)\cos(\nu_x) m_x}{2 \nu_x}.
\]
As a result of Lemma \ref{lem:unique}, $C_{xy}$ are uniformly elliptic and  $V_x\geq c m_x$ for some universal constant $c>0$.
Let us recall the event the event $\mathcal{A}_Q $:
\begin{multline}
\mathcal{A}_{Q} : = \biggl \{\omega: \forall x\in Q\text{ s.t. }\textrm{dist}(x, \partial^o B) \geq \eps^{-3/4}  \text{ and for }   r\in [\eps^{-1/2}, \eps^{-3/4}],\\ \; \frac{1}{r^2} \sum_{\| y-x\|_2 \leq r} m_{y}(\omega) \geq A \eps^2 |\log(\eps)|\biggr \}.\nonumber
\end{multline}

\begin{lemma}[Bulk behaviour]\label{lem:bulk}
Let  $Q\in \mathcal Q^*_{L/16}$ satisfy the conditions  of Definition \ref{def:nice}. Given $\phi$ defined on $Q\cup \partial^o Q$, let $R$ denote the maximal subset of $Q\cup \partial^o Q$ so that $|\phi_x|\leq \pi/5$ for $x \in \partial^o R$ (assuming such a set exists). For $c_0>0$ let 
$$\mathfrak T_R=\{x\in \partial^o R: |\phi_x|\geq c_0|\log(\eps)|^{-1/2}\}.$$ 
Let $\nu$ maximize $-\mathcal{K}_{R}(\nu|\phi)$, $-L_C$ denote the associated operator on $R$ and let $\psi$ solve the equation $-L_C\cdot \psi=0$ in R. Then $\psi=\nu$ and there is $\delta_0$ so that if  $\delta\in (0, \delta_0)$ then for all $\eps$ small enough and for all $x\in R$ satisfying $\textrm{dist}(x, \mathfrak T_R) \geq \eps^{-1}|\log (\eps)|^{-1/2+\delta}$,
\beq
\label{E:bound}
|\nu_x| \leq 2c_0 |\log (\eps)|^{-1/2}.
\eeq
If  $x\in R$ in addition satisfies $\textrm{dist}(x, \partial^o R) \geq \eps^{-1}|\log (\eps)|^{-1/2+\delta}$,
\begin{equation}
\label{E:bulk}
|\nu_x| \lesssim e^{-c|\log (\eps)|^{\delta}}.
\end{equation}
\end{lemma}

\begin{proof}
Let us abbreviate $l:=L/16$.  Recall the conductances $(C_e)_{e\in E(R)}$ defined by the optimizer $\nu$.   The main property we will use is that they are uniformly bounded from above and from below on $R$ as follows from Lemma \ref{lem:unique}. Consider the continuous time walk $(X_t)_{t\geq 0}$ starting at $x\in \mathbb{Z}^2$ with  natural filtration $(\mathcal{F}_t)_{t\geq 0}$. We will use the notation $\mathbb{E}_x (\cdot)$ resp. $\mathbb{Q}_x(\cdot)$ for the expected value resp. probability with respect to the random walk $(X_t)_{t\geq 0}$. 

Let $T_{R}$ denote the exit time from $R$ and $\psi$  the solution of $-L_C \cdot \psi=0$ in $R$.  
To bound $\psi_x$ we will use the Feynman-Ka\'c representation w.r.t. $(X_t)_{t\geq 0}$:
\begin{equation}
\label{eq_m}
\psi_x = \mathbb{E}_x \left ( e^{-\int_0^{T_{R}} V_{X_t} dt} \phi_{X_{T_{R}}}\right ).
\end{equation}
Given $a\in (\frac 12, \frac 34)$, let us consider a sequence of times
\[
0=s_0< s_1 < s_2 < ...< s_{N+1} = l^{2}
\]
such that, except for the final time interval,  $s_{i+1}-s_i= 2^k$ where $k=\lfloor 2a \log_2( l )  \rfloor$. Letting $t_i:=s_i\wedge T_R$,
we define the martingale increment $\mathcal{I}_i$ by 
\begin{equation*}
\mathcal{I}_i = \int_{t_i}^{t_{i+1}} V_{X_t} dt - \mathbb{E}_x \left (\int_{t_i}^{t_{i+1}} V_{X_t} dt | \mathcal{F}_{t_i}\right ).
\end{equation*}
Using positivity and the strong Markov property the exponent in \eqref{eq_m} is bounded  by
\begin{equation*}\label{eq_lb1}
\begin{split}
\int_{0}^{T_R} V_{X_t} dt & \geq \sum_{i=0}^{N} \mathcal{I}_i +\sum_{i=0}^N \mathbb{E}_{X_{t_i}} \left (\int_{0}^{t_{i+1}-t_i} V_{X_t} dt \right ).
\end{split}
\end{equation*}
The Cauchy-Schwartz inequality implies
\begin{equation}\label{eq_2bd}
\mathbb{E}_x  \left ( e^{- \int_0^{T_{R}} V_{X_t} dt}\phi_{X_{T_{R}}}\right )^2 \leq   \mathbb{E}_x \left ( e^{- 2 \sum_{i=0}^{N} \mathcal{I}_i} \right ) \mathbb{E}_x \left ( e^{-2  \sum_{i=0}^N \mathbb{E}_{X_{t_i}} \left (\int_{0}^{t_{i+1}-t_i} V_{X_t} dt \right )} \phi_{X_{T_{R}}}^2\right ).
\end{equation}

To bound the first term on the RHS of \eqref{eq_2bd} we use Azuma's inequality.
Since $Q$ is clean we have
\[
\|\mathcal{I}_i \|_{\infty} \leq \eps^2 |\log(\eps)|^2(s_{i+1}-s_i).
\]
Since $\eps^2 |\log (\eps)|^2 l^{2a}\leq 1,$
Azuma's inequality gives
\begin{equation}\label{1st}
\begin{split}
\mathbb{E}_x \left ( e^{- 2 \sum_{i=1}^{N+1} \mathcal{I}_i} \right ) &\leq e^{c \, (N+1)[\eps^2 |\log(\eps)|^2(s_{i+1}-s_i)]^2} \\
& \leq C,
\end{split}
\end{equation}
where the second inequality comes from the fact that $(N+1)l^{2a}\leq 2 l^2$ and $l^{2(1+a)}\eps^4 |\log(\eps)|^4\leq 1$.

Let $p_t(x,y)$ denote the transition kernel for the conductance model stopped upon hitting the boundary of $R$. 
Standard estimates for elliptic random walks give us that for $z$ such that $B_{l^{a}}(z)\subseteq R$, $t\in \left(\frac{[s_{i+1}-s_i]}{2}, [s_{i+1}-s_i] \right)$ and $l$ sufficiently large, 
\beq
\label{E:hk}
\sum_{y} p_t(z, y)m_{y}\geq C l^{-2a}\sum_{|y-z|< l^{a}} m_{ y}
\eeq
for some constant $C>0$ depending on the ellipticity contrast only.
Let
\[
I(z):=  \mathbb{E}_z \left ( \int_0^{s_{i+1}-s_i} m_{X_t} dt \cdot \1_{\{ s_{i+1}-s_i < T_{R}\}} \right ).
\]
and recall that the event $\mathcal{A}_Q$ occurs.
Letting $$R_1=\{z\in R: \textrm{dist}(z, \partial ^o R)> \eps^{-7/8}\}$$ then for $z\in R_1$ the heat kernel estimates above imply
\begin{equation*}
\begin{split}
I(z)&  \geq c\eps^2|\log (\eps)| (s_{i+1}-s_i)-\eps^2|\log (\eps)|^2 (s_{i+1}-s_i)\mathbb Q_z\left ({s_{i+1}-s_i \geq T_{R}} \right ).
\end{split}
\end{equation*}
Using \eqref{E:hk} again, there is $\delta'>0$ so that
\[
\mathbb Q_z\left ({s_{i+1}-s_i \geq T_{R}} \right )\leq Ce^{-c\eps^{-\delta'}}.
\]
Thus for some slightly smaller constant $c'>0$
 \[
 \mathbb{E}_{X_{t_i}} \left (\int_{0}^{t_{i+1}-t_i} V_{X_t} dt \right )\geq c I(X_{t_i})\geq c'\eps^2|\log (\eps)| (s_{i+1}-s_i) \1_{R_1}(X_{t_i}).
 \]
 Now let  $R_2= \{z\in R:  \textrm{dist}(z, \partial ^o R)> 2 \eps^{-7/8}\}$.
Taking $\delta'$ smaller if necessary
\[
\mathbb Q_x(\exists s, t\leq T_R: |s-t|\leq l^{a},\: X_t\in R_2, X_s\in R_1^{c})\leq C e^{-c\eps^{-\delta'}}.
\]
Therefore
\[
 \mathbb{E}_x \left ( e^{-2 \sum_{i=0}^N \mathbb{E}_{X_{t_i}} \left (\int_{0}^{t_{i+1}-t_i} V_{X_t} dt \right )}\phi_{X_{T_{R}}}^2\right )\lesssim  \mathbb{E}_x \left ( e^{-c\eps^2|\log (\eps)| |\{t\leq T_R: X_t\in R_2\}|} \phi_{X_{T_{R}}}^2\right)+ C e^{-c\eps^{-\delta'}}.
\]
Finally, if $\textrm{dist}(x, \partial^o R) \geq \eps^{-1}|\log (\eps)|^{-1/2+\delta}$ then it follows from ellipticity that
\[
\mathbb Q_x(|\{t\leq T_R: X_t\in R_2\}|< \eps^{-2}|\log (\eps)|^{-1+3\delta/2})\lesssim e^{-c|\log (\eps)|^{\delta}}
\]
and so 
\beq
\label{E:CS2}
 \mathbb{E}_x \left ( e^{-2 \sum_{i=0}^N \mathbb{E}_{X_{t_i}} \left (\int_{0}^{t_{i+1}-t_i} V_{X_t} dt \right )}\phi_{X_{T_{R}}}^2\right )\lesssim e^{-c|\log (\eps)|^{\delta}}.
\eeq

Combining \eqref{eq_2bd}, \eqref{1st}, and \eqref{E:CS2} justifies \eqref{E:bulk}.

To justify \eqref{E:bound}, suppose $\textrm{dist}(x, \mathfrak T_R) \geq \eps^{-1}|\log (\eps)|^{-1/2+\delta}$ and let $$T_0=\inf\{ t: \textrm{dist}(x, \mathfrak T_R)\leq \eps^{-3/4}\}.$$
Then
\begin{multline}
 \mathbb{E}_x \left ( e^{-c\eps^2|\log (\eps)| |\{t\leq T_R: X_t\in R_2\}|} \phi_{X_{T_{R}}}^2\right)\leq \\
 c_0^2 |\log (\eps)|^{-1}  + C \mathbb{E}_x \left ( e^{-c\eps^2|\log (\eps)| |\{t\leq T_R: X_t\in R_2\}|} \1_{\{T_0<T_R\}}\right).\nonumber
\end{multline}
 
For any $\delta>0$ small enough, and for all $\eps$ sufficiently small depending on $\delta$,
\[
\mathbb Q_x( |\{t\leq T_R: X_t\in R_2\}|< l^{-2} |\log (l) |^{\delta/2-2s}; T_0< T_R)\lesssim |\log (\eps)|^{-2}.
\]
Since $l=L/16$ and $L=\eps^{-1} |\log (\eps)|^{-1/2+s}$ it follows from this estimate that 
\[
\mathbb{E}_x \left ( e^{-c\eps^2|\log \eps| |\{t\leq T_R: X_t\in R_2\}|} 1_{\{T_0<T_R\}}\right) \lesssim |\log (\eps)|^{-2} +e^{-c|\log (\eps)|^{\delta/2}}
\]
and the claimed estimate on $\nu_x$, \eqref{E:bound}, follows.
\end{proof}

{\begin{proof}[Proof of Lemma \ref{lem:mod3}]
Recall the notation introduced just prior to the statement of Lemma \ref{lem:mod3}.  We proceed to justify the lemma by induction. 
If $i=1$, we set $\bar{\sigma}^{(0)}={\sigma}^{(2)}$.  Applying Lemma \ref{lem:cov} with respect to $R_1$ and $\bar{\sigma}^{(0)}$, optimizing $-\mathcal K_{R_1}$, and then inverting the change of variables, we produce $\bar{\sigma}^{(1)}$ so that $\bar{\sigma}^{(1)}_x=\bar{\sigma}^{(0)}_x$ on $B_1^c$, 
$\mathcal E_{Q_1}(\bar{\sigma}^{(1)}|\bar{\sigma}^{(0)})\leq 2[\mathcal E_{Q_1}(\bar{\sigma}^{(0)}|\bar{\sigma}^{(0)})+\mathcal E_{Q_1}(g^{\lambda, D}_{Q_1}|0)]$,  
\begin{equation*}
 -\mathcal{H}_{\Lambda_N}(\bar{\sigma}^{(1)}|e_1) \geq - \mathcal{H}_{\Lambda_N}(\bar{\sigma}^{(0)}|e_1)-  C_0 \eps^2 |\log(\eps)|^{5/8}|Q_1|
\end{equation*}
and the estimates of Lemma \ref{lem:bulk} are valid.

Passing to the inductive step, suppose we constructed $\bar{\sigma}^{(i)}$ satisfying the conditions of the Lemma and consider $Q_{i+1}$.  If $Q_{i+1}$ is disjoint from $\cup_{j\leq i} Q_j$ then the construction proceeds as  in the base case. Otherwise it may happen that for some $x \in Q_{i+1}$,  $\bar{\sigma}^{(i)}_x\neq \bar{\sigma}^{(0)}_x$.  

Using Lemma \ref{lem:unique},
no matter how $x\in Q_{i+1}$ was previously updated, it remains true that $$\bar{\sigma}^{(i)}_x\cdot e_1\geq 9/10-12\max_{j\leq i} \|g^{\lambda, D}_{Q_j}\|_{\infty}\geq \cos(\pi/6)$$ provided $\eps$ is small enough.  This means that the condition $|\phi^{(i+1)}_x|\leq \frac{\pi}{5}$ remains enforced on $\partial^o R_{i+1}$ and we can apply Lemma \ref{lem:bulk} in $R_{i+1}$.  By construction we have
\[
- \mathcal{K}_{R_{i+1}}( \nu^{(i+1)}|\phi^{(i+1)}) \geq - \mathcal{K}_{R_{i+1}}( \phi^{(i+1)}|\phi^{(i+1)})
\]
which implies 
\[
\mathcal E_{Q_{i+1}}( \bar{\phi}^{(i+1)}|\phi^{(i+1)}) \leq  \mathcal E_{Q_{i+1}}( \phi^{(i+1)}|\phi^{(i+1)})+\mathcal E_{Q_{i+1}}(g^{\lambda, D}_{Q_{i+1}}|0) .
\]
From the form of the change of variables and Lemma \ref{lem:cov}, we thus have
 \begin{equation*}
 -\mathcal{H}_{\Lambda_N}(\bar \sigma^{(i+1)}|e_1) \geq - \mathcal{H}_{\Lambda_N}(\bar \sigma^{(i)}|e_1)-  C_0 \eps^2 |\log(\eps)|^{5/8}|Q_{i+1}|
\end{equation*}
and 
\[
\mathcal E_{Q_{i+1}}( \bar{\sigma}^{(i+1)}|\bar{\sigma}^{(i)}) \leq 2 [\mathcal E_{Q_{i}}( \bar{\sigma}^{(i)}| \bar{\sigma}^{(i)})+ \mathcal E_{Q_{i}}(  g^{\lambda, D}_{Q_{i}}|0)].
\]

Finally, let us denote $r(\eps)= \eps^{-1} |\log(\eps)|^{-1/2+\delta}$ for $\delta$ sufficiently small but fixed throughout.  We want to check that 
$$
 |e_2\cdot \sigma^{(i+1)}_x| \leq 32 |\log (\eps)|^{-1/2} \text{ if } \textrm{dist}(x,  \partial^o \mathcal R_{i+1})\geq 6 r(\eps).$$ If $\textrm{dist}(x,  \partial^o \mathcal R_{i+1})\geq 6 r(\eps)$, then by construction the set $\{j \leq i+1: x\in R_j, \textrm{dist}(x,  \partial^o R_{j})\geq 6 r(\eps) \}$ is nonempty. If $j_0$ is the largest such index, it follows that from Lemma \ref{lem:bulk}, Equation \eqref{E:bulk}, that $|\sigma^{(j_0)}_z\cdot e_2|\leq |\log (\eps)|^{-1/2}$ for $z\in B_{5r(\eps)}(x)$.  But then, if $n(j)$ denote the number of updates to some  $z\in B_{2r(\eps)}(x)$ in $[j_0, j)$,  due to the assumed bound $\|g^{\lambda, D}_{Q_j}\|_{\infty}\leq |\log (\eps)|^{-1/2}$ for all $j$,  Equation \eqref{E:bound} from Lemma \ref{lem:bulk} implies
$|\sigma^{(j)}_z\cdot e_2|\leq 2^{n(j)+1} |\log(\eps)|^{-1/2}$ for $z\in B_{2r(\eps)}(x)$.  Since the number of updates after $j_0$ to any such vertex is at most $4$, so the claim follows from.
\end{proof}

\section{Estimates on the randomness}\label{S:Rand}

\subsection{Estimates on the supremum of the random field}\label{S:RandSub}

In the following let us fix a sidelength $l=2^k$ for some $k$ and let $Q_{l} = \{1,..., l \}^2$.  In this section we obtain estimates on the fields $g^{\lambda, D}_{x,Q_{l}}$ and $g^{\lambda, N}_{x,Q_{l}}$. Recall that
\begin{equation}\label{eq:fields}
 g^{\lambda, D}_{x,Q_{l}} = \eps (-\Delta^D_{Q_{l}} + \lambda)^{-1} \alpha_x \text{ resp. }  g^{\lambda, N}_{x,Q_{l}} = \eps (-\Delta^N_{Q_{l}} + \lambda)^{-1} \alpha_x. 
\end{equation}
 Following the construction in Section 9.5 from \cite{Glimm}, equation 9.5.11 and Proposition 9.5.3, we have, with $\Lambda_{l}^*= \{\frac{\pi}{l+1}, ..., \pi(1-\frac{1}{l+1}) \}^2$ \begin{equation}\label{eq:gD}
\eps^{-1} g^{\lambda, D}_{x, Q_{l}} = \frac{1}{l}  \sum_{k \in \Lambda_{l}^*} \frac{1}{\zeta_k + \lambda} e^D_k(x) \widehat{\alpha}(k),
\end{equation}
where 
\begin{equation}\label{eq:evD}
e^D_k(x) = 2\prod_{i=1}^2 \sin(k_i x_i) , \, \, \, \, \, 
\zeta_k = 4 \sum_{i=1}^2 \sin^2\left( \frac{ k_i}{2} \right).
\end{equation}
Similarly with $\tilde{\Lambda}_l^*= \{0, ..., \pi(1-\frac{1}{l}) \}^2$ we can write 
\begin{equation}\label{eq:gN}
\eps^{-1} g^{\lambda, N}_{x, Q_{l}} = \frac{1}{l} \sum_{k \in \tilde{\Lambda}_l^*} \frac{1}{\zeta_k + \lambda} e^N_k(x) \widehat{\alpha}(k)
\end{equation}
where $e^N_k(x) = 2\prod_{i=1}^2 \cos(k_i x_i)$, see also \cite{Gott}, Equations (2.4) and (2.6).
Define the momentum space annuli for $s=0,..., \lceil \log_2( l+1) \rceil$ by
\[
A_s = \left \{ k\in \Lambda_l^*: \frac{\pi}{2^{1+s}} \leq \|k\|_2 \leq \frac{\pi}{2^s}\right \}.
\]
We want to control the field $g^{\lambda,D}_{x,Q_{l}}$ at different scales, therefore rewrite \eqref{eq:fields} as
\begin{equation*}\label{g_split}
\begin{split}
\eps^{-1} g^{\lambda, D}_{x , Q_{l}} & =  \sum_{s=0}^{\lceil \log_2(l+1) \rceil} \left ( \frac{1}{l} \sum_{k\in A_s} \frac{1}{\zeta_{k } + \lambda} e^D_{k}\left( x  \right) \widehat{\alpha}\left(k  \right) \right )   \\
& =  \sum_{s=0}^{\lceil \log_2(l+1) \rceil}  \overline{g}^{\lambda, D}_{x,A_s}.
\end{split}
\end{equation*}
Using $\tilde{\Lambda}_l^*$ in place of $\Lambda_l^*$ in the definition of $A_s$, a similar decomposition gives
\begin{equation*}\label{g_split_N}
\begin{split}
\eps^{-1} g^{\lambda, N}_{x , Q_{l}} & =  \sum_{s=0}^{\lceil \log_2(l+1) \rceil} \left (l^{-1} \sum_{k\in { A}_s} \frac{1}{\zeta_{k } + \lambda} e^N_{k}\left( x  \right) \widehat{\alpha}\left(k  \right) \right )   \\
& =  \sum_{s=0}^{\lceil \log_2(l+1) \rceil}  \overline{g}^{\lambda, N}_{x, {A}_s}.
\end{split}
\end{equation*}
where in ${A}_{\lceil \log_2(l+1) \rceil}$ we include $k=0$.
We  split the respective RHS's into parts with low frequency modes and high frequency modes. The cut-off will be chosen at some scale $S$ which we specify later.

\begin{lemma}\label{lem:var}
Let the $g^i$-field, for $i\in \{N, D\}$ be decomposed as
\begin{equation*}\label{eq:split}
  \sum_{s=0}^{\lceil \log_2(l+1) \rceil}  \overline{g}^{\lambda, i}_{x,A_s} 
  \end{equation*}
Then we have that
\begin{enumerate}
\item for each $s=0,...,  \lceil \log_2(l+1) \rceil$ the variance $\sigma^{2}_{A_s}$ is given by
\[
\sigma^{2}_{A_s}:=Var(\overline{g}^{\lambda, i}_{x,A_s})  \simeq \mathcal{O} \left ( \frac{2^{-2s}}{(\lambda +2^{-2s})^2}\right ).
\]

\item The variance $\sigma^{2}_{xy, s}$ between  field values at two points $x,y\in Q_l$
\[
 \sigma^{2}_{xy, s} \simeq  \mathcal{O} \left( \frac{2^{-4s}}{(\lambda + 2^{-2s})^2}\| x-y\|^2_2\right).\]

\end{enumerate}
\end{lemma}
\begin{proof}
These are simple computations using that the $g$-field is composed by a sum of independent Gaussian random variables. 
First of all note that the eigenvalues $
\zeta_k $ can be bounded from below by 
\[
\zeta_k\geq  \frac{4}{\pi^2}  \|k\|^2
\]
see also Equation 9.5.24 from \cite{Glimm}.
We can then estimate the variance $\sigma^{2}_{A_s}$ of $\overline{g}^{\lambda, i}_{x,A_s}$ for a fixed point $x$:
\begin{equation*} \label{sigma_As}
\begin{split}
Var(\overline{g}^{\lambda, i}_{x,A_s}) & =\frac{1}{ l^2 } \sum_{k\in A_s} \left(e^i_{k}\left ( x \right)\right)^2 \frac{1}{(\zeta_{k}+\lambda)^2} \lesssim \frac{1}{ l^2 } \sum_{k\in A_s} \frac{1}{(\| k \|^2  + \lambda)^2}  \\
& \sim \int_{2^{-(1+s)}}^{2^{-s}} \frac{r}{(r^2+\lambda)^2} dr = \mathcal{O} \left ( \frac{2^{-2s}}{(\lambda +2^{-2s})^2}\right )
 \end{split}
\end{equation*}
where in the second line we used the Euler-Maclaurin formula, Theorem 1 from \cite{Apo}.

To prove statement (2), observe that
for two points $x,y\in A_s$:
\[
\begin{split}
\sigma^{2}_{xy, s} = Var\left (\overline{g}^{\lambda, i}_{x, A_s} - \overline{g}^{\lambda, i}_{y, A_s}  \right ) 
& = \frac{1}{ l^2 } \sum_{k\in A_s} \frac{1}{(\zeta_{k } +\lambda )^2} \left ( e^i_{k } \left ( x\right ) -   e^i_{k}\left ( y\right )\right )^2 \\
& \lesssim \frac{1}{ l^2 } \sum_{k\in A_s} \frac{1}{(\zeta_{k } +\lambda )^2} \| \nabla e^i_{k}(\xi) \|^2 \|x-y\|^2_2
\end{split}
\]
for some $\xi$ on the line segment connecting $x$ to $y$.  Note that
\[
\| \nabla e^i_{k}(\xi) \|^2 \leq 4 \| k\|^2.
\]
Hence we can further bound the RHS of the previous estimate by
\[
Var\left (\overline{g}^{\lambda, i}_{x, A_s} - \overline{g}^{\lambda, i}_{y, A_s} \right ) \lesssim \frac{1}{ l^2 } \sum_{k\in A_s} \frac{\| k\|^2}{(\|k \|^2 +\lambda )^2}  \|x-y\|^2_2.
\]
Using Euler-Maclaurin again we can estimate analogously to before
\[
\begin{split}
  \sum_{k\in A_s} \frac{\| k\|^2}{(\|k \|^2  +\lambda )^2} &\lesssim   \int_{2^{-(1+s)}}^{2^{-s}} \frac{r^3}{(r^2 +\lambda)^2} dr = \mathcal{O} \left( \frac{ 2^{-4s}}{(\lambda + 2^{-2s})^2}\right).
\end{split}
\]
\end{proof}

From now on, we fix $\eps>0$ and let $z \in (0,1/2)$. We choose 
\begin{equation}\label{def:lg}
l\in [\eps^{-1} |\log (\eps)|^{-1/2-z}, \eps^{-1} |\log (\eps)|^{-1/2+z}].
\end{equation}
Recall that $\lambda=\eps^2 |\log (\eps)|^{1+\eta_{\lambda}}$, $\eta_{\lambda}\in (2s,1/8)$ and $s\in(0,1/32)$.
The next proposition will provide a probabilistic upper bound for the supremum of the field. We will split this probability into three terms and show how to bound each term. Let us define
\begin{equation}
\overline{\zeta}^2_2 :=\int_{[0,\pi]^2} \frac{1}{(\|k \|^2 +\lambda)^2} d^2 k
\end{equation}
and note that $\overline{\zeta}^2_2$ is of order $\mathcal{O}(\lambda^{-1})$.

\begin{proposition}\label{prop:ub_supg}
  Let  $M>0$. Then there exists  $c>0$ so that with $\lambda=\eps^2 |\log (\eps)|^{1+\eta_{\lambda}}$, with $\eta_{\lambda}\in (2s,1/8), s\in(0,1/32)$ and $l$ defined in \eqref{def:lg} and for all $\eps$ small enough
\begin{equation*}\label{prop:tail}
\begin{split}
\mathbb{P} \left (\| g^{\lambda, i}_{Q_{l}} \|_{\infty}  \geq M \eps \overline{\zeta}_2 \right ) & \lesssim |\log (\eps)|^7 e^{-c \frac{M^2}{\log \log(1/\eps)}}.
\end{split}
\end{equation*}
\end{proposition}
As a corollary we obtain:
\begin{corollary}\label{cor:bound_L2}
Fix $z\in (0,1/2)$, $\delta \in (1,2)$, $M>0$ and let $S = \lfloor \log_2 \left( l /(\log_2(\eps))^{(1+\delta)/2} \right) \rfloor$. Then there exists $c>0$ such that 
for all $\eps>0$ small enough \begin{equation*}
\begin{split}
\mathbb{P} \left ( \|g^{\lambda, i}_{Q_{l}} \|_{2}  \geq \frac{M \eps l}{\sqrt{\lambda}} \right ) & \leq |\log (\eps)|^7 e^{-c \frac{M^2}{\log \log(1/\eps)^2}}.
\end{split}
\end{equation*}
\end{corollary}

\begin{proof}[Proof of Proposition \ref{prop:ub_supg}]
Let  $S = \lfloor \log_2 \left( l /\log(\eps)^2 \right) \rfloor$. Splitting the scales using this particular choice, we have 
\begin{equation}\label{eq:low+high}
\mathbb{P} \left ( \|\eps^{-1} g^{\lambda, i}_{Q_{l}} \|_{\infty} \geq  M \overline{\zeta}_2 \right ) \leq \mathbb{P} \left ( \sum_{s=0}^{S-1} \| \overline{g}^{\lambda, i}_{A_s} \|_{\infty} \geq M \overline{\zeta}_2/2\right ) + \mathbb{P} \left ( \sum_{s=S}^{ \lceil \log_2(l+1) \rceil } \| \overline{g}^{\lambda, i}_{A_s} \|_{\infty} \geq   M \overline{\zeta}_2/2 \right ).
\end{equation}

For the first term on the RHS, using a simple union bound, Gaussian tail estimates and our previous variance bounds
\[
\mathbb{P} \left ( \sum_{s=0}^{S-1} \| \overline{g}^{\lambda, i}_{A_s} \|_{\infty} \geq M \overline{\zeta}_2/2\right ) \leq \sum_{x\in Q_l} \sum_{s=0}^{S-1} \mathbb{P} \left( \left | \overline{g}^{\lambda, D}_{x,A_s} \right | \geq  \frac{M \overline{\zeta}_2}{2|S|} \right)  \lesssim   \frac{ l^2|S| \sigma_{A_S}}{M \overline{\zeta}_2}   e^{-  \frac{M^2\overline{\zeta}^2_2 }{8 \sigma^{2}_{A_S}|S|^2 }}= \mathcal O(e^{- c M^2 |\log_2(\eps)|^{2}}).\]
since, for $s\leq S$,  $\frac{\overline{\zeta}^2_2 }{ \sigma^{2}_{A_s} }\gtrsim |\log (\eps)|^{4}$.

For the second term in \eqref{eq:low+high} we cannot simply apply a union bound.  On the other hand the cardinality of $\{s: S\leq s\leq \lceil \log_2(l+1) \rceil\}$ is at most $4 \log_2\log (1/\eps))$.  Thus, in order that the event in the second term be satisfied, for at least one $s$ the event $\{ \| \overline{g}^{\lambda, i}_{A_s} \|_{\infty} \geq   \frac{M \overline{\zeta}_2}{4\log_2 \log (1/\eps)}\}$ must be satisfied. 
To estimate the probability of this last event, the essential point is that $\overline{g}^{\lambda, i}_{A_s}$ varies slowly over length scales of order  at most $2^s$.  Let us partition the box $Q_{l}$ into dyadic squares of sidelength $l_*=O(2^s/|\log_2(\eps)|) $. With  $N_*:=l^2\log_2(\eps)^2 /2^{2s}$, So to control the sup norm of $\overline{g}^{\lambda, i}_{A_s}$, we should only need to control one vertex per partition set   we have $\mathcal O(N_*)$ boxes $B^j_{l_*}$ of size $l_* \times l_*$ in such a partition where $j=1,...,N_*$.  

Let $M_0>0$
and, for each $j$, let $x_j$ be a fixed point of box $B^j_{l_*}$.  Define the event $B_s$ as
\[
B_s =\{\omega: \exists j\in \{1,..., N_*\}: |\overline{g}^{\lambda, D}_{x_j, A_s} (\omega)| > M_0 \overline{\zeta}_2 /4\}.
\]
Then
\[
\begin{split}
\mathbb{P} \left( \left \| \overline{g}^{\lambda, i}_{A_s} \right \|_{\infty} \geq  M_0 \overline{\zeta}_2/2 \right) & = \mathbb{P} \left( \left\{\left \| \overline{g}^{\lambda, i}_{A_s} \right \| \geq  M_0 \overline{\zeta}_2/2 \right\} \cap B_s\right) + \mathbb{P} \left( \left \{\left \| \overline{g}^{\lambda, i}_{A_s} \right \|_{\infty} \geq  M_0 \overline{\zeta}_2/2 \right \} \cap B_s^c\right)\\
& \leq  \mathbb{P} (B_s) + \mathbb{P} \left(  \left \{ \left \| \overline{g}^{\lambda, i}_{A_s} \right \| \geq  M_0 \overline{\zeta}_2 /2  \right \} \cap B_s^c\right).
\end{split}
\]
To bound the first term on the second line, we employ a union bound which is now efficient since for $s\geq S$, $N_*\lesssim |\log (\eps)|^{6}$.  We have
\[
\mathbb{P} (B_s) \leq \sum_{j=1}^{N_*}\mathbb{P}(|\overline{g}^{\lambda, D}_{x_j, A_s} (\omega)| > M_0 \overline{\zeta}_2 /4)\lesssim  (\log (\eps))^6 e^{-c \frac{M_0^2 \overline{\zeta}_2^2 }{\sigma_s^2}}
\]
where $\sigma_s^2= \frac{2^{-2s}}{(\lambda +2^{-2s})^2}$ and $c>0$ is a universal constant.  Since $s\geq S$ we have
\[
 \frac{\overline{\zeta}_2^2 }{\sigma_s^2}\simeq 2+2^{2s}\lambda+2^{-2s}\lambda^{-1}\geq 2
\]

To bound the second term on the RHS let $\sigma^{2}_{max} = \sup_{x,y\in B^j_{l_*}, s} \{\sigma^{2}_{xy, s} \}$ and note that
\[
\frac{ \overline{\zeta}^2_2}{ \sigma^{2}_{max}}\gtrsim (\log (\eps))^2
\]
by our choice of $S$ and the previous lemma.  We thus have
\[
\begin{split}
\mathbb{P} \left(\left \{ \left \| \overline{g}^{\lambda, i}_{x,A_s} \right \|_{\infty}  \geq   M_0 \overline{\zeta}_2 /2 \right \} \cap B_s^c\right)& \leq \sum_{j=1}^{N_*}
 \mathbb{P} \left(\left \{\exists y \in B^j_{l_*}:  | \overline{g}^{\lambda, D}_{y,A_s}  |  \geq M_0 \overline{\zeta}_2 /2 \right \} \cap \left \{|\overline{g}^{\lambda,D}_{x_j, A_s}| < M_0 \overline{\zeta}_2/4 \right \}  \right)\\
&  \lesssim \frac{N_* \sigma_{max} |B^j_{l_*}|^2}{M_0 \overline{\zeta}_2} e^{-  \frac{M_0^2 \overline{\zeta}^2_2}{32 \sigma^{2}_{max}}} \lesssim e^{-c M_0^2 (\log (\eps))^2}
 \end{split}
\]
since  $\frac{N_* \sigma_{max} |B^j_{l_*}|^2}{\overline{\zeta}_2} \lesssim \eps^{-5}$ for $\eps$ sufficiently small.
The lemma  follows by combining the bounds and observing that the bound on $\mathbb{P} (B_s)$ represents the main contribution.  Taking $M_0=\frac{M}{4\log_2 (\log (1/\eps)})$,
\[
\sum_{s\geq S} \mathbb{P} (B_s)\leq |\log (\eps)|^7 e^{-c \frac{M^2}{\log_2( \log(1/\eps))^2}}.
\]
\end{proof}

\begin{lemma}\label{lem:DirichletDiff}
Let the fields $(g^{\lambda,D}_{x,Q_{l}})_{x\in Q_{l}}$ and   $(g^{\lambda,N}_{x,Q_{l}})_{x\in Q_{l}}$ be defined as in \eqref{eq:fields}. Set
\[
\Delta \mathcal{E}^{D,N}(g):=\mathcal{E}(g^{\lambda, D}_{Q_{l}}) - \mathcal{E}(g^{\lambda, N}_{Q_{l}}).
\]
Then for $\rho>0$
\[
\mathbb{P}\left(\left | \Delta \mathcal{E}^{D,N}(g) -\mathbb{E}(\Delta \mathcal{E}^{D,N}(g))\right | \geq \rho \right) \lesssim \rho^{-2} |\log(\eps)|^{-(2+2s+\eta_{\lambda})},
\]
where $|\mathbb{E}(\Delta \mathcal{E}^{D,N}(g))|= \mathcal{O}\left(\frac{\eps^2}{\sqrt{\lambda} l} \right).$
\end{lemma}
\begin{proof}
We have
\begin{align*}
\E[ \mathcal{E}(g^{\lambda, D}_{Q_{l}}) - \mathcal{E}(g^{\lambda, N}_{Q_{l}})  ] &= \sum_{k\in \Lambda^*_l}  \frac{\eps^2\zeta_{k}}{(\zeta_{k} +\lambda)^2} -\sum_{k\in \tilde{\Lambda}_l^*}  \frac{\eps^2\zeta_{k}}{(\zeta_{k} +\lambda)^2} \\
\end{align*}
Let $v=(\pi/l, \pi/l)$ and  observe that $\frac{(l+1)}{l} \Lambda^*_l-v=\tilde{\Lambda}_l^*$. The difference above can thus be rewritten
\[
 \sum_{k\in \tilde{\Lambda}^*_l}  \frac{\eps^2\zeta_{\frac{l}{l+1}(k+v)}}{(\zeta_{\frac{l}{l+1}(k+v)} +\lambda)^2} - \frac{\eps^2\zeta_{k}}{(\zeta_{k} +\lambda)^2} 
 \]
Now
\[
\left| \frac{\zeta_{\frac{l(k+v)}{l+1}}}{(\zeta_{\frac{l(k+v)}{l+1}} +\lambda)^2} - \frac{\zeta_{k}}{(\zeta_{k} +\lambda)^2} \right|\lesssim \frac{| \zeta_{\frac{l}{l+1}(k+v)}-\zeta_{k}|}{ (\zeta_{k} +\lambda)(\zeta_{\frac{l}{l+1}(k+v)} +\lambda)}\lesssim  \frac{\frac{\|k\|}{l}+\|v\|^2 }{ (\zeta_{k} +\lambda)(\zeta_{\frac{l}{l+1}(k+v)} +\lambda)}.
\]
Thus
\[
|\E[ \mathcal{E}(g^{\lambda, D}_{Q_{l}}) - \mathcal{E}(g^{\lambda, N}_{Q_{l}})  ] |\lesssim {\mathcal{O}\left(\frac{\eps^2}{\sqrt{\lambda} l} \right)} 
\]
using that
\[
\int_0^{l} \frac{r^2}{(r^2+\lambda l^2)^2} dr =  {\mathcal{O}\left(\frac{1}{\sqrt{\lambda} l} \right)}. 
\]

To bound the variance, observe that even though the $\alpha_k$'s are correlated, by using the Cauchy-Schwartz inequality it suffices to bound the individual variances.  With $I\in \{D, N\}$, 
\[
\begin{split}
& \mathbb{E}\left ( \left[  \mathcal{E}(g^{\lambda, I}) -\mathbb{E}( \mathcal{E}(g^{\lambda, I}))  \right ]^2 \right )\\
& \lesssim 2 \sum_{\substack{e=\langle x^{(1)},y^{(1)}\rangle, \\ e'=\langle x^{(2)},y^{(2)}\rangle, \\ e,e' \cap Q_{l} \neq \varnothing}}  \sum_{{k \in \atop \{1,...,l \}^2 }} \frac{\eps^4 }{l^4(\zeta_{k\pi/l +1} + \lambda)^4} \prod_{i \in \{1,2\}}  \left( e^I_k \left(\frac{x^{(i)}\pi}{l+1} \right) - e^I_k \left(\frac{y^{(i)} \pi}{l +1} \right)\right )^2 . 
\end{split}
\]
The RHS is bounded by
\[
\begin{split}
 \sum_{\substack{e=\langle x^{(1)},y^{(1)}\rangle, \\ e'=\langle x^{(2)},y^{(2)}\rangle, \\ e,e' \cap Q_{l} \neq \varnothing}}  \sum_{{k \in \atop \{1,...,l \}^2 }} \frac{\eps^4 \| k\|^4 \pi^4 \| \nabla e^D \|_{\infty}^4}{(\|k\|^2 + \lambda (l+1)^2)^4} \prod_{i \in \{1,2\}}  \| x^{(i)} - y^{(i)}\|_2^2 \lesssim \mathcal{O}(|\log(\eps)|^{-(2+2s+\eta_{\lambda})})
\end{split}
\]
using that
\[
\int_0^{l} \frac{r^5}{(r^2+\lambda l^2)^4} dr = \mathcal{O}\left(\frac{1}{\lambda l^2} \right). 
\]
\end{proof}

The final local task we have is to bound $\mathbb P(\mathcal A_Q^c)$.  Recall
that we defined 
\begin{multline}
\mathcal{A}_{Q_{L_0}} : = \biggl \{\omega: \forall x\in Q_{L_0}: \textrm{dist}(x, \partial^o Q_{L_0}) \geq L_0^{3/4} \text{ and for }   r\in[ L_0^{1/2},L_0^{3/4}]\\ \; \frac{1}{r^2 \eps^2} \sum_{\| y-x\|_2 \leq r} m_{y}(\omega) \geq A  |\log(r )|\biggr \}.\nonumber
\end{multline}
\begin{proposition}
There is $A>0$ so that for any $L_0\in [\ell, L]$ and any square $Q_{L_0}$ with side-length $L_0$
\[
\mathbb P \left(\mathcal A_{Q_{L_0}}^c \right)\lesssim e^{-c \log^2 (\eps)}.
\]
\end{proposition}
\begin{proof}
This proof appeared as Lemma 11.1.5 in \cite{Craw_order}.  We present it here for completeness and to clean up some inefficiencies.
Let us define 
\begin{align*}
&\mathcal C _{Q, r}= \{ Q \subset Q_{L_0}: Q \text{ is a square of sidelength $r$}\},\\
&F=\cap_{r \geq L_0^a} \cap _{Q \in \mathcal C _{Q, r}} \{ \omega: \eps^{-2}\|\nabla g^{\lambda, D}_Q\|_2^2 \geq A r^2 \log(r)\}.
\end{align*}
Then by Lemma 11.1.3 of \cite{Craw_order} for any $a>0$ and choosing $A$ sufficiently small
\[
\mathbb P(F^c) \lesssim \exp(-c \log ^{2} (L_0)).
\]
for all $\eps$ small enough.

We prove that $F \subset \mathcal A_r(Q_{L_0})$ for appropriate choices of $r, A>0$. 
Let $Q\in  \cup_{r \geq L_0^a}  \mathcal C _{Q, r}$ and fix $\omega \in F$.  For $x \in Q$, we may express the field $g^{\lambda, D}_{Q_{L_0}, x}$ via
\[
g^{\lambda, D}_{Q_{L_0}, x}=g^{\lambda, D}_{Q, x}+ g^{(1)}_x  
\]
where $g^{(1)}$ is satisfies the Laplace equation $-\Delta^D g^{(1)}\equiv 0$ on $Q$ subject to the boundary condition $g^{(1)}_x=g^{\lambda, D}_{Q_{L_0}, x}$ for  $x \in \partial^o Q$.  

Notice that
\[
\sum_{e \cap Q \neq \varnothing} [\nabla_e(g^{\lambda, D}_{Q, x}+ g^{(1)}_x  )]^2 = \mathcal E_{Q}(g^{\lambda, D}_{Q, x}|0)+ 
\sum_{e \cap Q_{L_0} \neq \varnothing} [\nabla_e g^{(1)}_x  ]^2
\]
since the cross term vanishes.  This is because $g^{(1)}$ is harmonic in $Q$ and $g^{\lambda, D}_{Q}$ vanishes on $ \partial^o Q$.
Hence
\[
\eps^{-2} \sum_{x \in Q} m^2_x \geq \mathcal E_{Q\cup   \partial^o Q_{L_0}}(g^{\lambda, D}_{Q_{L_0}})\geq A  r^2 \log(r)
\]
because we restricted attention to $F$.  Thus $F\subset \mathcal A_{Q_{L_0}}$ whenever $L_0^a  \leq r $.

To finish the proof, we provide an upper bound on $\mathbb{P}(F^c)$.  Using a union bound,  it is clearly sufficient to prove there is $c>0$ and $A>0$ sufficiently small so that for all $r>0$, if $Q$ is a square of sidelength $r$
\[
\mathbb{P}( \eps^{-2}\|\nabla g^{\lambda, D}_Q\|_2^2 \leq A r^2 \log(r))\leq e^{-c r}.
\]
This bound is a standard application of large deviations for squares of Gaussian variables, the proof is omitted.
\end{proof}

\subsection{Estimates on bad regions}\label{S:dirty}

\begin{proposition}\label{prop:dirty}
Fix $\eta, \chi \in (0,1/16)$ as in Definition \ref{def:regularY}. Let $\mathcal{Y}$ be a $L_0$-measurable and bounded region. Then for $\eps$ small enough, 
\[
\mathbb{P}(\mathcal{Y} \text{ is bad }) \leq e^{- c|\log(\eps)|^{\delta} N^\mathcal{Y}_{L_0}}.
\]
where $\delta = \min\{\eta, \chi\}$
\end{proposition}

\begin{proof}
The region is bad if at least one of the conditions (R0)-(R3) from Definition \ref{def:regularY} is not satisfied. We consider only (R1), since the proofs for the other cases are similar. We will prove that
\begin{equation}\label{eq:claimregular}
\mathbb{P} \left (\langle F_{\lambda, L_0}(\nabla g), \1_{F_{\lambda, L_0}(\nabla g) \geq \eps^2 |\log(L_0)|^{1+\chi}}\rangle_\mathcal{Y}  > C \eps^2 |\log(\eps)|^{\zeta} L_0^2N^\mathcal{Y}_{L_0}  \right ) \leq e^{- c|\log(\eps)| N^\mathcal{Y}_{L_0}},
\end{equation}
where $c,C>0$ are some constants. For $\eta\in \{(0,0) (0, L_0/2), (L_0/2, 0), (L_0/2, L_0/2)\}$ let $\mathcal{Q}_{\mathcal{Y}, \eta}$ denote the collection of squares $Q$ such that $Q+\eta$ is an  $L_0$-measurable squares contained in $\mathcal{Y}$.
Let $m_Q$ denote the minimal integer greater or equal to $2$ such that
\[
\| \nabla g^{\lambda}_{Q}\|^2_2 <  m_Q \eps^2 |\log(\eps)|^{1+\chi} L_0^2.
\]
We will bound probabilities for the collection of events
\[
\left \{ \sum_{Q\in  \mathcal{Q}_{\mathcal{Y}, \eta}} m_Q \1_{m_Q \geq 2} \geq  |\log(\eps)|^{\zeta-\chi-1} N^\mathcal{Y}_{L_0}\right \}.
\]

First of all, we claim that for some $c_1,c_2,  C >0$ 
\[
\mathbb P(\eps^{-2}\| \nabla g^{\lambda}_{Q}\|^2_2 > c_1 m  |\log(L_0)| L_0^2)\leq C e^{-c_2 m \log L_0^{1-\zeta/2}}.
\]
Indeed 
\[
\eps^{-2}\| \nabla g^{\lambda}_{Q}\|_2^2 =\frac{1}{L_0^2}  \sum_{k \in \Lambda_{L_0}^*} \frac{\zeta_k}{[\zeta_k + \lambda]^2}|\widehat{\alpha}(k)|^2.
\]
Decomposing $\Lambda_{L_0}^*$ in terms of the annuli $A_s$, if $2^{-2s}\geq \lambda$ we have
\[
 \sum_{k \in A_s} \frac{\zeta_k}{[\zeta_k + \lambda]^2}|\widehat{\alpha}(k)|^2\sim 2^{2s}  \sum_{k \in A_s} |\widehat{\alpha}(k)|^2,
\]
where as if $2^{-2s}\leq \lambda$ we have
\[
 \sum_{k \in A_s} \frac{\zeta_k}{[\zeta_k + \lambda]^2}|\widehat{\alpha}(k)|^2\sim 2^{-2s}\lambda^{-2}  \sum_{k \in A_s} |\widehat{\alpha}(k)|^2.
\]
Since the $\hat{\alpha}_k$'s are i.i.d Gaussian variables, we have that for all $M>1$,
\[
\mathbb P ( \sum_{k \in A_s} |\widehat{\alpha}(k)|^2\geq M |A_s|)\leq Ce^{-cM|A_s|}.
\]
It follows that if we take $M_s=2m$ for $s\leq \lfloor \log_2 \left( L_0 /\log(L_0) \right) \rfloor$ and $M_s=m |\log (L_0)|^{1-\zeta/2}$ otherwise, we have
\[
\mathbb P ( \exists s: \sum_{k \in A_s} |\widehat{\alpha}(k)|^2\geq M_s |A_s|)\leq  Ce^{-cm |\log (L_0)|^{1-\zeta/2}}.
\]
The claim follows immediately from this using a union bound.

Now, fix a subset $\mathcal{A} \subset \mathcal{Q}_{\mathcal{Y}, \eta}$ a collection of integers $\{ n_Q; Q \in \mathcal{A}, n_Q \geq 2\} $, then by independence
\[
\mathbb{P}(\forall Q \in \mathcal{A}: m_Q = n_Q) \leq C^{|\mathcal{A}|} e^{-c_2 \sum_{Q \in \mathcal{A}} (n_Q-1) | \log (L_0)|^{1+\chi-\zeta/2}}.
\]
which leads to the bound
\[
\mathbb{P} \left (\sum_{Q \in \mathcal{Q}_{\mathcal{Y}, \eta}} m_Q \1_{m_Q \geq 2} \geq M N^\mathcal{Y}_{L_0}\right ) \leq C^{N^\mathcal{Y}_{L_0}} e^{-c M |\log(L_0)|^{1+\chi-\zeta/2} N^\mathcal{Y}_{L_0}}.
\]
The claim follows for (R1) by using a union bound with respect to\\ $\eta\in \{(0,0) (0, L_0/2), (L_0/2, 0), (L_0/2, L_0/2)\}$.
\end{proof}
We are now ready to verify the main probabilistic estimate stated in Section \ref{S:Prob}.
\begin{proof}[Proof of Lemma \ref{lem:dirty}]

Let $Q(x)$ denote the $L$-measurable box containing $x$ and let
\[
A(x) = \{\omega: \exists \, \mathcal{\mathcal{Y}} \text{ bad and } Q(x) \subset \textrm{cl}(\mathcal{\mathcal{Y}}) \}.
\]
For each $\mathcal{\mathcal{Y}}$ we want to use Proposition \ref{prop:dirty}. Using the lattice isoperimetric inequality there is for some constant $C>1$, the number of $\mathcal{\mathcal{Y}}$ such that $cl(\mathcal Y)$ contains $x$ and $N^L_\mathcal{Y}=r$ is at most $r^2 C^r$ . Then
\[
\mathbb{P}(Q(x) \subset \mathcal{\mathcal{Y}}, \mathcal{\mathcal{Y}} \text{ bad }) \lesssim \sum_{r=1}^{\infty} (2C)^r e^{- c |\log(\eps)|^{\delta} r}
\]
and using a discrete isoperimetric inequality we obtain
\begin{equation}\label{eq:dirty}
\mathbb{P}(A(x)) \lesssim \sum_{r=1}^{\infty} r^2 (2C)^r e^{- c |\log(\eps)|^{\delta} r} \lesssim e^{-c'|\log(\eps)|^{\delta}}
\end{equation}
for $\eps$ small enough.  By linearity of the expectation, 
\[
\mathbb{E}(|\mathbb{D}_{\Lambda}|) \lesssim |\Lambda| e^{-c|\log(\eps)|^{\delta}}
\]
for any finite square $\Lambda$ in $\mathbb Z^2$.

Let  $\textrm{dist}_L(x,y)$ denote the minimal number of blocks in an $L$-measurable block path from $Q(x)$ to $Q(y)$.  To estimate $Var(|\mathbb{D}_{\Lambda}|) $ we first estimate $|\mathbb{P}(A(x_1) \cap A(x_2)) - \mathbb{P}(A(x_1))  \mathbb{P}(A(x_2))|$ for $x_1$ and $x_2$ in disjoint $L$-measurable squares.
Note that if $\mathcal{\mathcal{Y}}_1$ and $\mathcal{\mathcal{Y}}_2$ be two regions such that 
\[
\delta_{2L}(\textrm{cl}(\mathcal{\mathcal{Y}}_1)) \cap \delta_{2L}(\textrm{cl}(\mathcal{\mathcal{Y}}_2)) = \varnothing.
\]
then the events $\{ \mathcal{\mathcal{Y}}_1 \text{ bad}\}$ and $\{ \mathcal{\mathcal{Y}}_2 \text{ bad}\}$ are independent.  Given $x_1, x_2$ such that $\textrm{dist}_L(x_1,x_2)\geq 6$ and let $B_1=\{y: |x_1-y|\leq |x_1-x_2|/3\},$   Independence implies
\begin{align*}
 |\mathbb{P}(A(x_1) \cap A(x_2)) - \mathbb{P}(A(x_1))  \mathbb{P}(A(x_2))|   
&\leq 4 \mathbb{P}(\exists  \mathcal{\mathcal{Y}}_1,  \text{ bad s.t. }Q(x_1) \subset \textrm{cl}(\mathcal{\mathcal{Y}}_1),  \mathcal{Y}_1\not \subset B_1) \\
& \lesssim L^2 e^{- c |\log(\eps)|^{\delta} \textrm{dist}_L(x_1,x_2)},
\end{align*}
where the second inequality follows by summing the tail in the union bound for $r\geq  \textrm{dist}_L(x_1,x_2)/3$ in  \eqref{eq:dirty}.
From this inequality it follows that
\[
Var(|\mathbb{D}_{\Lambda}|) \lesssim L^2 |\Lambda| e^{-c|\log(\eps)|^{\delta}}.
\]
Finally taking $\Lambda = \Lambda_N$ for $N=2^k$, Chebyshev's inequality implies
\[
\sum_{k=1}^{\infty} \mathbb{P} \left(|\mathbb{D}_{2^k}| > |\Lambda_{2^k}| e^{-c|\log(\eps)|^{\delta}} \right) < \infty
\]
By the Borel-Cantelli Lemma, for almost all $\omega$ we can find  $k_0(\omega) \in \N$
\[
|\mathbb{D}_{\Lambda_{2^k}}| \leq  |\Lambda_{2^k}| e^{-c|\log(\eps)|^{\delta}}
\]
Since $\mathbb{D}_{\Lambda_{N}}\subset \mathbb{D}_{\Lambda_{2^{k}}}$ and $1 \leq \frac{2^{2k}}{N^2}\leq 2$
for $k=\lceil \log_2 N\rceil$,  it holds that
\[
|\mathbb{D}_{\Lambda_{N}}| \leq 2 |\Lambda_{N}| e^{-c|\log(\eps)|^{\delta}}
\]
for all $N\geq 2^{k_0}$.
\end{proof}


\begin{thebibliography}{aaa}



\bibitem{Aharony}
A.~Aharony, \textit{Spin-flop multicritical points in systems with random fields and in spin glasses}, {Phys. Rev. B}, {18}, {7}, p. {3328--3336}, ({1978}).

\bibitem{ALL} 	
D. A.~Abanin,  P. A.~Lee,   L. S.~Levitov,
\textit{Randomness-Induced XY Ordering in a Graphene Quantum Hall
 Ferromagnet,} Phys. Rev Lett. 98(15):156801, (2007). 

\bibitem{AP}
M.~Aizenman, R.~Peled, \textit{A power-law upper bound on the correlations in the 2D random field Ising model}, Comm.
Math. Phys., Vol. 372, Nr. {12}, p. 865--892, (2019). 

\bibitem{AW}
M.~Aizenman, J.~Wehr, \textit{Rounding effects of quenched randomness on first-order phase transitions,} Comm.
Math. Phys. Vol. 130, Nr. 3, p. 489--528, (1990).  

\bibitem{Apo}
T. M.~Apostol, \textit{An Elementary View of Euler’s Summation Formula}, The American
Mathematical Monthly, 106(5): p. 409--418, (1999).


\bibitem{BCK}
M. ~Biskup, L.~Chayes and S.A.~Kivelson, \textit{Order by disorder, without order, in a two-dimensional spin system with $O(2)$ symmetry}, Ann. Henri Poincar\`e 5, Nr. 6, p. 1181--1205, (2004).  

\bibitem{BK}
J.~Bricmont, A.~Kupiainen, \textit{Phase transition in the 3d random field Ising model}, Comm. Math. Phys., Vol. 116, Nr 4, p. 539--572, (1988). 

\bibitem{collet}
F.~ Collet, W.~Ruszel, \textit{Synchronization and spin-flop transitions for a mean-field XY model in random field},
Journ. Stat. Phys. 164, Nr. 3, p. 645--666, (2016).

\bibitem{NC}
N.~Crawford \textit{On Random Field Induced Ordering in the Classical XY Model,}
Jour. Stat. Phys., Vol. 142, Nr. 1, p. 11--42, (2011). 

\bibitem{Craw_order}
N.~Crawford, \textit{Random field induced order in low dimension I}, Comm. Math. Phys. 328, p. 203--249, (2014).

\bibitem{Delmotte}
T.~ Delmotte, \textit{Parabolic Harnack Inequality and Estimates of Markov Chains on Graphs,}
Rev. Mat. Iberoamericana, Vol. 15, Nr. 1, (1999).  

\bibitem{Ding}
J.~Ding, J.~Xia, \textit{Exponential decay of correlations in the two-dimensional random field Ising model},
Invent. mat., Vol. 224, Nr. {06}, p. 1--47, (2021).


\bibitem{DS}
R.L.~Dobrushin, S.B.~Shlosman, \textit{Absence of Breakdown of Continuous Symmetry in Two-dimensional Models of Statistical Physics,} Comm. Math. Phys. Vol. 42, Nr. 31, (1975).  


\bibitem{DF1}
V.S.~Dotsenko, M.V.~Feigelman,
\textit{2D Random-axes XY Magnet},
Jour. Phys. C: Solid State Phys. 14, L823, (1981). 


\bibitem{DF2}
V.S.~Dotsenko, M.V.~Feigelman,
\textit{Spin Glass Phase Transition in Random-axes XY Magnet},
Jour. Phys. C: Solid State Phys. 15, L565, (1982).


\bibitem{vE0}
A.C.D.~van Enter, C.~K\"uelske, \textit{Non-existence of random gradient Gibbs measures in continuous interface models in d=2,}
Ann. Appl. Probab. 18, p. 109--119, (2008). 


\bibitem{vE1}
A.C.D.~van Enter, W.M.~ Ruszel,  \textit{Gibbsianness versus Non-Gibbsianness of time-evolved planar rotor
models,} Stoch. Proc. Appl. 119, p. 1866--1888, (2009).  

\bibitem{vE2}
A.C.D.~van Enter, W.M.~Ruszel, \textit{ Loss and recovery of Gibbsianness for XY spins in a small external
field,} Jour. Math. Phys. 49, 125208, (2008).  


\bibitem{vE3}
A.C.D.~van Enter, C.~K\"{u}lske, A.A.~Opoku, W.M.~Ruszel, 
\textit{Gibbs-non-Gibbs properties for n-vector
lattice and mean-field models,}
Braz. J. Prob. Stat. 24, p.226--255, (2010). 

\bibitem{Feld1}
D. E.~Feldman, \textit{Exact zero-temperature critical behaviour of the ferromagnet in the uniaxial random field,} Jour. Phys. A, Vol. 31, (1998).  

\bibitem{Feld2}
D. E.~Feldman, \textit{Critical behavior of a degenerate ferromagnet in a uniaxial random field: Exact results in a space of arbitrary dimension,} Jour. Exp. Theor. Phys. 88, (1999). 

\bibitem{FP}
J.~Fr\"{o}hlich, C.~Pfister,
\textit{On the absence of spontaneous symmetry breaking and of crystalline ordering in two-dimensional systems,}
Comm. Math. Phys. Vol. 81, Nr. 2, p. 277--298, (1981).


\bibitem{FS}
J.~Fr\"{o}hlich, T.~Spencer, \textit{The Kosterlitz-Thouless transition in two-dimensional Abelian spin systems and the Coulomb gas,}
Comm.  Math. Phys., Vol. 81, Nr. 4, p. 527--602, (1981). 


\bibitem{FSS}
J.~Fr\"{o}hlich, B.~Simon, T.~Spencer, \textit{Infrared bounds, phase transitions and continuous symmetry breaking},
Comm.  Math.  Phys., Vol. 50, Nr. 1, p. 79--95, (1976).

\bibitem{Glimm}
J.~Glimm, A.~Jaffe, \textit{Qunatum physics: A functional integral point of view}, Springer, New York, (1987).

\bibitem{Gott}
H.P.W.~Gottlieb, \textit{Eigenvalue of the Laplacian with Neumann boundary conditions}, Jour. Austral. Math. Soc. Ser. B 26, p. 293--309, (1985).


\bibitem{Hen}
C.L.~Henley, \textit{Ordering due to disorder in a frustrated vector antiferromagnet,} Phys. Rev. Lett. 62, 2056 (1989).  

\bibitem{Imbrie}
J. Z.~Imbrie,
\textit{The Ground State of the Three-Dimensional Random-Field Ising Model,}  Commun.
Math. Phys. 98, p. 145-176, (1985). 

\bibitem{IM} 
Y.~Imry, S. K.~Ma, \textit{Random-field instability
    for the ordered state of continuous symmetry,} Phys. Rev. Lett.
    35, 1399, (1975).  

\bibitem{MS}
O.~McBryan, T.~Spencer,
\textit{On the decay of correlations in $SO(n)$-symmetric ferromagnets}, Comm. Math. Phys. Vol. 53, Nr. 3, p. 299--302, (1977).  

\bibitem{MW} D.~Mermin, H.~Wagner, \textit{Absence of
    ferromagnetism or anti-ferromagnetism in one- or twodimensional
    isotropic Heisenberg models,} Phys. Rev. 17, 1133, (1966). 

\bibitem{MP}
B.J.~Minchau, R.A.~Pelcovits,
\textit{Two-dimensional XY Model in a Random Uniaxial Field,}
Phys. Rev. B, Vol. 32, Nr. 5,  (1985). 


\bibitem{Pres-Book} E.~Presutti, \textit{Scaling limits in statistical
    mechanics and microstructures in continuum mechanics,} Springer
    Verlag, (2008).  


\bibitem{SPL-Nature} L.~Sanchez-Palencia, M.~Lewenstein,
    \textit{Disordered quantum gases under control,} Nature Physics
    6, p. 87--95, (2010). 

\bibitem{SVBO}
U.~Schulz, J.~Villain, E.~Br\'{e}zin, H.~Orland,
\textit{Thermal Fluctuations in Some Random Field Models,}
Jour. Stat. Phys, Vol. 51, Nos. 1/2, (1988). 

\bibitem{Wehr-et-al-1}
A.~Bera, D.~Rakshit, M.~Lewenstein, A.~Sen(De), U.~Sen, J.~Wehr, 
\textit{Classical spin models with broken symmetry: Random-field-induced order and persistence of spontaneous magnetization in the presence of a random field}, Phys. Rev. B, Vol. 90, 174408, (2014).


\bibitem{Wehr-et-al-2} J.~Wehr, A.~Niederberger, L.~Sanchez-Palencia,
    M.~Lewenstein,  \textit{Disorder versus the
    Mermin-Wagner-Hohenberg effect: From classical spin systems to
    ultracold atomic gases,} Phys. Rev. B 74, 224448 (2006).  

\end{thebibliography}
\end{document}